\documentclass[a4paper, 11pt]{amsart}
\usepackage{amsmath, amscd, amssymb, amsfonts, amsthm, enumerate, mathtools, tikz-cd,  latexsym, graphicx}
\usepackage[margin=1in]{geometry}

\newtheorem{thm}{Theorem}[section]
\newtheorem{lem}[thm]{Lemma}
\newtheorem{defi}[thm]{Definition}
\newtheorem{rem}[thm]{Remark}

\newtheorem{prop}[thm]{Proposition}
\newtheorem{cor}[thm]{Corollary}

\newtheorem{thmA}{Theorem}

\def\QQ{\mathbb{Q}}
\def\ZZ{\mathbb{Z}}
\def\FF{\mathbb{F}}
\def\calR{\mathcal{R}}
\def\calI{\mathcal{I}}
\def\TT{\mathbb{T}}
\newcommand{\GL}{\mathrm{GL}}

\def\Ga1{\Gamma_1}
\def\rhob{\bar\rho}

\def\tr{\operatorname{tr}}
\def\ad{\text{ad}}
\def\red{\text{red}}
\def\univ{\text{univ}}

\def \Hom{\text{Hom}}
\def\tan{\text{tan}}
\def\defo{\text{def}}
\def\pd{\text{pd}}
\DeclareMathOperator{\Ext}{Ext}
\DeclareFontFamily{U}{wncy}{}
\DeclareFontShape{U}{wncy}{m}{n}{<->wncyr10}{}
 \DeclareSymbolFont{mcy}{U}{wncy}{m}{n}
 \DeclareMathSymbol{\Sh}{\mathord}{mcy}{"58}

\begin{document}
\baselineskip 18pt

\title{Effect of increasing the ramification on pseudo-deformation rings}
\author{Shaunak V. Deo} 
\email{shaunakdeo@iisc.ac.in}
\address{Department of Mathematics, Indian Institute of Science, Bangalore 560012, India}
\date{}


\begin{abstract}
Given a continuous, odd, semi-simple $2$-dimensional representation of $G_{\QQ,Np}$ over a finite field of odd characteristic $p$ and a prime $\ell$ not dividing $Np$, we study the relation between the universal deformation rings of the corresponding pseudo-representation  for the groups $G_{\QQ,N\ell p}$ and $G_{\QQ,Np}$. As a related problem, we investigate when the universal pseudo-representation arises from an actual representation over the universal deformation ring. Under some hypotheses, we prove analogues of theorems of Boston and B\"{o}ckle for the reduced pseudo-deformation rings. We improve these results when the pseudo-representation is unobstructed and $p$ does not divide $\ell^2-1$. When the pseudo-representation is unobstructed and $p$ divides $\ell+1$, we prove that the universal deformation rings in characteristic $0$ and $p$ of the pseudo-representation for $G_{\QQ,N\ell p}$ are not local complete intersection rings. As an application of our main results, we prove a big $R=\TT$ theorem.
\end{abstract}

\maketitle

\section{Introduction}

In \cite{Bos}, Boston studied the effect of enlarging the set of primes that can ramify on the structure of the universal deformation ring of an odd, absolutely irreducible representation of $\text{Gal}(\overline\QQ/\QQ)$ over a finite field which is attached to a modular eigenform of weight $2$. His results were generalized by B\"{o}ckle in \cite{Bo} to any continuous $2$-dimensional representation of $\text{Gal}(\overline\QQ/\QQ)$ over a finite field such that the centralizer of its image is exactly scalars. The aim of this paper is to study the same problem for pseudo-deformation rings i.e. universal deformation rings of pseudo-representations.

This article has two parts. In the first part, we analyze when a pseudo-representation arises from an actual representation. In the second part, we use the results obtained in the first part to study how the structure of the universal deformation ring of a $2$-dimensional Galois pseudo-representation changes after allowing ramification at additional primes. We will now elaborate on each part.

All the representations and pseudo-representations of pro-finite groups considered in this article are assumed to be continuous unless mentioned otherwise.

\subsection{Pseudo-representation arising from a representation}
Let $G$ be a pro-finite group and $R$ be a complete Noetherian local (CNL for short) ring. Roughly speaking, a $2$-dimensional pseudo-representation of $G$ over $R$ is a tuple of functions $(t,d) : G \to R$ which `behaves like' the trace and determinant of a $2$-dimensional representation of $G$ over $R$. In particular, if $\rho : G \to \GL_2(R)$ is a representation of $G$, then $(\tr(\rho),\det(\rho)) : G \to R$ is a pseudo-representation of $G$ of dimension $2$. But the converse to this statement is not necessarily true.

The notion of pseudo-representation that we are going to use throughout the article was introduced and studied by Chenevier in \cite{C}. Chenevier's theory of pseudo-representations generalized the theory of psuedo-characters developed by Rouquier in \cite{R}. We refer the reader to \cite[Section 1.4]{BK} for definition and properties of $2$-dimensional pseudo-representations and to \cite{C} for general theory of pseudo-representations.

Now suppose $p$ is an odd prime, $\FF$ is a finite field of characteristic $p$ and $G$ is a pro-finite group satisfying the finiteness condition $\Phi_p$ of Mazur (see \cite[Section 1.1]{M}). Denote the ring of Witt vectors of $\FF$ by $W(\FF)$. Suppose $\rhob_0 : G \to \GL_2(\FF)$ is a representation such that $\rhob_0 = \chi_1 \oplus \chi_2$ where $\chi_1, \chi_2 : G \to \FF^\times$ are \emph{distinct} characters (i.e. $\chi_1 \neq \chi_2$). 

Let $R$ be a CNL $W(\FF)$-algebra with residue field $\FF$ and $(t,d) : G \to R$ be a pseudo-representation of $G$ deforming $(\tr(\rhob_0),\det(\rhob_0))$. Then we address the following question in the first part of the article:\\
Does there exist a representation $\rho : G \to \GL_2(R)$ such that $t=\tr(\rho)$ and $d=\det(\rho)$ ?\\
If there does exist such a representation $\rho$, then we say that the pseudo-representation $(t,d)$ arises from a representation.

\subsubsection{\bf Motivation}
In \cite{Bos}, Boston used the techniques and results from the theory of pro-$p$ groups to determine how the deformation ring of an absolutely irreducible Galois representation changes after enlarging the set of ramifying primes. The same techniques were used by B\"{o}ckle in \cite{Bo} to extend Boston's results to residually non-split reducible representations (see \cite[Theorem $4.7$]{Bo}). However, their method crucially depends on working with actual representations (and not just pseudo-representations). So, in order to use their techniques and results, we first investigate when a Galois pseudo-representation arises from an actual representation.

Moreover, this question is also of an independent interest for any pro-finite group (and not just for the Galois groups). Therefore, we do not restrict ourselves to Galois groups in the first part of the article and work with a general pro-finite group. 

\subsubsection{\bf Main results}
Recall that we have $\rhob_0 : G \to \GL_2(\FF)$ with $\rhob_0 = \chi_1 \oplus \chi_2$. Let $\chi := \chi_1\chi_2^{-1}$. For $i \in \{1,-1\}$, we denote the dimension of the cohomology group $H^j(G,\chi^i)$ as a vector space over $\FF$ by $\dim(H^j(G,\chi^i))$.

\begin{thmA}[see Theorem~\ref{reduceprop}, Theorem~\ref{generalthm}]
\label{thma}
Suppose $\dim(H^1(G,\chi^i)) = 1$ and $H^2(G,\chi^i)=0$ for some $i \in \{1,-1\}$ and fix such an $i$. Then:
\begin{enumerate}
\item If $R$ is a \emph{reduced} CNL $W(\FF)$-algebra with residue field $\FF$, then every pseudo-representation $(t,d) : G \to R$ deforming $(\tr(\rhob_0),\det(\rhob_0))$ arises from a representation.
\item Suppose $H^2(G,1) =0 $, $\dim(H^1(G,\chi^{-i})) \in \{1,2,3\}$ and $\dim(H^2(G,\chi^{-i})) < \dim(H^1(G,\chi^{-i}))$. If $R$ is a CNL \emph{$\FF$-algebra} with residue field $\FF$, then every pseudo-representation $(t,d) : G \to R$ deforming $(\tr(\rhob_0),\det(\rhob_0))$ arises from a representation.
\end{enumerate}
\end{thmA}

As a consequence of the theorem above, we get that certain pseudo-deformation rings are isomorphic to appropriate deformation rings of residually reducible, non-split representations (see Theorem~\ref{reduceprop} and Theorem~\ref{generalthm} for more details). In \S\ref{galrepsubsec}, we list the consequences of these results for Galois groups.

\begin{rem}
 The hypotheses $\dim(H^1(G,\chi^i)) = 1$ and $H^2(G,\chi^i)=0$ are used to construct the representations whose existence is claimed in the first part of Theorem~\ref{thma}.
The hypotheses of the second part are used along with results of \cite{WE} to get a description of the structure of the universal mod $p$ deformation ring of $(\tr(\rhob_0),\det(\rhob_0))$. This description is crucially used to construct a representation which gives rise to the universal mod $p$ pseudo-representation deforming $(\tr(\rhob_0),\det(\rhob_0))$.
In Proposition~\ref{repnprop}, we prove that the hypothesis $\dim(H^1(G,\chi^i))=1$ for some $i \in \{1,-1\}$ is necessary for the second part of Theorem~\ref{thma} to hold.
However, it is not clear whether Theorem~\ref{thma} holds without any of the other hypotheses.
\end{rem}

\subsection{Level raising for pseudo-deformation rings}
In the second part, we specialize the set-up introduced in \S\ref{repsec} to the case where $G = G_{\QQ,Np}$ and $\rhob_0$ is an odd representation. To be precise, we consider a reducible, semi-simple, odd representation $\rhob_0 : G_{\QQ,Np} \to \GL_2(\FF)$ where $p$ is an odd prime, $\FF$ is a finite extension of $\FF_p$, $N$ is an integer not divisible by $p$. Thus $\rhob_0 = \chi_1 \oplus \chi_2$ where $\chi_1, \chi_2 : G_{\QQ,Np} \to \FF^\times$ are characters and let $\chi := \chi_1\chi_2^{-1}$.

Let $\calR^{\pd}_{\rhob_0}$ be the universal deformation ring of the pseudo-representation $(\tr(\rhob_0),\det(\rhob_0)) : G_{\QQ,Np} \to \FF$ in the category of CNL $W(\FF)$-algebras with residue field $\FF$. Suppose $\ell$ is a prime not dividing $Np$. Then we have a natural surjective map $G_{\QQ,N\ell p} \twoheadrightarrow G_{\QQ,Np}$ and via this surjective map, we can view $(\tr(\rhob_0),\det(\rhob_0))$ as a pseudo-representation of $G_{\QQ,N\ell p}$. Let $\calR^{\pd,\ell}_{\rhob_0}$ be the universal deformation ring of the pseudo-representation $(\tr(\rhob_0),\det(\rhob_0))$ for the group $G_{\QQ,N\ell p}$ in the category of CNL $W(\FF)$-algebras with residue field $\FF$.

Our aim is to compare $\calR^{\pd,\ell}_{\rhob_0}$ with $\calR^{\pd}_{\rhob_0}$ and determine the structure of $\calR^{\pd,\ell}_{\rhob_0}$ in terms of the structure of $\calR^{\pd}_{\rhob_0}$.
\subsubsection{\bf Motivation}
Our interest in the problem mainly arises from its potential application to determining the structure of characteristic $0$ and characteristic $p$ Hecke algebras (as defined in \cite{BK} and \cite{D}) and to the level raising of modular forms.

In \cite{Bos}, Boston connects the increase in the space of deformations, after allowing ramification at an additional prime $\ell$, to the level raising of modular forms. To be precise, he shows, using the results of Ribet and Carayol, that every new component of the bigger deformation space contains a point corresponding to a modular eigenform which is new at $\ell$. 

When the residual representation is reducible, the level raising results for modular forms are not known in all cases (see \cite{BM}, \cite{Y} and \cite{D2} for known cases of level raising results for reducible $\rhob_0$). So if $\rhob_0$ comes from a newform of level $N$ and the level raising results are not known for it, then results along the lines of \cite{Bos} for pseudo-deformation ring can be treated as evidence for level raising for $\rhob_0$. 

On the other hand, suppose $\rhob_0$ comes from a newform of level $N$ and level raising is known $\rhob_0$. Then, we are interested in studying the relationship between $\mathbb{T}^{\Gamma_1(N\ell)}_{\rhob_0}$, the $\rhob_0$-component of characteristic $0$ Hecke algebra of level $N\ell$ and $\mathbb{T}^{\Gamma_1(N)}_{\rhob_0}$, the $\rhob_0$-component of the characteristic $0$ Hecke algebra of level $N$ (see \cite{BK} and \cite{D} for the definitions of these Hecke algebras).
In particular, we want to explore if the structure of $\mathbb{T}^{\Gamma_1(N\ell)}_{\rhob_0}$ can be obtained from the structure of $\mathbb{T}^{\Gamma_1(N)}_{\rhob_0}$. 

Note that we have surjective maps $\calR^{\pd,\ell}_{\rhob_0} \twoheadrightarrow \mathbb{T}^{\Gamma_1(N\ell)}_{\rhob_0}$ and $\calR^{\pd}_{\rhob_0} \twoheadrightarrow \mathbb{T}^{\Gamma_1(N)}_{\rhob_0}$ which are known to be isomorphisms in certain cases. Thus, exploring this question for deformation rings serves as a good starting point for this study and it also gives us an idea of what to expect in the case of Hecke algebras. We are also interested in exploring similar questions for mod $p$ Hecke algebra of level $N\ell$ and $N$ (as defined in \cite{D} and \cite{BK}).

\subsubsection{\bf Main results}
Recall that we have an odd $\rhob_0 : G_{\QQ,Np} \to \GL_2(\FF)$ with $\rhob_0 = \chi_1 \oplus \chi_2$ and $\chi=\chi_1\chi_2^{-1}$. For $i \in \{1,-1\}$, denote the restriction of $\chi^i$ to the decomposition group at $\ell$ by $\chi^i|_{G_{\QQ_\ell}}$. Let $\omega_p$ be the mod $p$ cyclotomic character, $R^{\pd,\ell}_{\rhob_0} := \calR^{\pd,\ell}_{\rhob_0}/(p)$ and $R^{\pd}_{\rhob_0} := \calR^{\pd}_{\rhob_0}/(p)$. For a ring $R$, we denote by $(R)^{\red}$ its maximal reduced quotient. Using results of \S\ref{galrepsubsec} and \cite{Bo}, we prove:

\begin{thmA}
\label{thmb}
Suppose $\dim(H^1(G_{\QQ,Np},\chi^i)) =1$ and $\dim(H^1(G_{\QQ,Np},\chi^{-i})) =m$ for some $i \in \{1,-1\}$. Let $\ell$ be a prime such that $p \nmid \ell^2-1$ and $\chi^{-i}|_{G_{\QQ_{\ell}}} = \omega_p|_{G_{\QQ_{\ell}}}$. Then:
\begin{enumerate}
\item There exists $r_1,\cdots,r_{n'},\Phi \in W(\FF)[[X_1,\cdots,X_n,X]]$ such that $$(\calR^{\pd,\ell}_{\rhob_0})^{\red} \simeq (W(\FF)[[X_1,\cdots,X_n,X]]/(r_1,\cdots,r_{n'},X(\Phi-\ell)))^{\red}$$ and $(\calR^{\pd}_{\rhob_0})^{\red} \simeq (W(\FF)[[X_1,\cdots,X_n]]/(\bar r_1,\cdots,\bar r_{n'}))^{\red},$ where $r_i \pmod{X} = \bar r_i$. 
\item Suppose $m=1,2$ and $p \nmid \phi(N)$. Then there exists $r_1,\cdots,r_{n'},\Phi \in \FF[[X_1,\cdots,X_n,X]]$ such that $$R^{\pd,\ell}_{\rhob_0} \simeq \FF[[X_1,\cdots,X_n,X]]/(r_1,\cdots,r_{n'},X(\Phi-\ell))$$ and $R^{\pd}_{\rhob_0} \simeq \FF[[X_1,\cdots,X_n]]/(\bar r_1,\cdots,\bar r_{n'}),$ where $r_i \pmod{X} = \bar r_i$. 
\end{enumerate}
\end{thmA}

\begin{rem}
The hypotheses of Theorem~\ref{thmb} make sure that the hypotheses of first and second part of Theorem~\ref{thma} hold for both $ G_{\QQ,Np}$ and $G_{\QQ,N\ell p}$ in the first and second part of Theorem~\ref{thmb}, respectively.
This allows us to combine Theorem~\ref{thma} and results of \cite{Bo} to get Theorem~\ref{thmb}.
However, the description of the structure of $\calR^{\pd,\ell}_{\rhob_0}$ is expected to get more complicated if we relax one or more hypotheses of Theorem~\ref{thmb}.
This is illustrated in the results given below.
\end{rem}

 We call $\rhob_0$ unobstructed when $\dim(H^1(G_{\QQ,Np},\chi^i))=1$ for $i \in \{1,-1\}$. Note that if $N=1$, then any $\rhob_0$ is unobstructed if Vandiver's conjecture is true (\cite[Theorem 22]{BK}). Moreover, \cite[Theorem 22]{BK} also gives some examples of unobstructed $\rhob_0$'s if $N=1$. 
Note that if $\rhob_0$ is unobstructed and $p \nmid \phi(N)$, then $\calR^{\pd}_{\rhob_0} \simeq W(\FF)[[X,Y,Z]]$.
We then prove slightly more precise results after assuming $\rhob_0$ is unobstructed and $p \nmid \phi(N)$.
\begin{thmA}[See Corollary~\ref{strcor} and Theorem~\ref{ramunobsprop}]
\label{thmc}
Suppose $\rhob_0$ is unobstructed, $p \nmid \phi(N)$ and $\ell$ is a prime such that $\ell \nmid Np$, $p \nmid \ell^2-1$ and $\chi^i|_{G_{\QQ_\ell}} = \omega_p$ for some $i \in \{1,-1\}$. Then:
\begin{enumerate}
\item $\calR^{\pd,\ell}_{\rhob_0} \simeq W(\FF)[[X_1,X_2,X_3,X_4]]/(X_4f)$ for some non-zero $f \in W(\FF)[[X_1,X_2,X_3,X_4]]$,
\item Moreover if $p^2 \nmid \ell^{p-1}-1$, then $\calR^{\pd,\ell}_{\rhob_0} \simeq W(\FF)[[X_1,X_2,X_3,X_4]]/(X_4X_2)$.
\end{enumerate}
\end{thmA}

\begin{rem}
The hypotheses that $\rhob_0$ is unobstructed, $p \nmid \ell^2-1$ and $\chi^i|_{G_{\QQ_\ell}} = \omega_p$ for some $i \in \{1,-1\}$ of Theorem~\ref{thmc} make sure that the hypotheses of Theorem~\ref{thmb} are satisfied.
The hypotheses $\rhob_0$ is unobstructed and $p \nmid \phi(N)$ imply that $\calR^{\pd}_{\rhob_0} \simeq W(\FF)[[X,Y,Z]]$.
Moreover, combining these hypotheses with $p^2 \nmid \ell^{p-1}-1$, we get a set of generators of the cotangent space of $R^{\pd,\ell}_{\rhob_0}$.
All this information is then combined with Theorem~\ref{thma} to prove Theorem~\ref{thmc}.
\end{rem}

The case $p | \ell+1$ turns out to be different from the other cases which also happens in \cite{Bos} and \cite{Bo}.
\begin{thmA}[see Theorem~\ref{reducestrthm}, Theorem~\ref{lcithm}, Corollary~\ref{lcicor}]
\label{thmd}
Suppose $\rhob_0$ is unobstructed, $p \nmid \phi(N)$ and $\ell$ is a prime such that $\ell \nmid Np$, $p \mid\mid \ell +1$ and $\chi|_{G_{\QQ_\ell}} = \omega_p$. Then $$(R^{\pd,\ell}_{\rhob_0})^{\red} \simeq \FF[[X,Y,Z,X_1,X_2]]/(X_1X_2,X_1Y,X_2Y).$$
Moreover, both $R^{\pd,\ell}_{\rhob_0}$ and $\calR^{\pd,\ell}_{\rhob_0}$ are not local complete intersection rings.
\end{thmA}

\begin{rem}
The hypotheses $\rhob_0$ is unobstructed and $p \nmid \phi(N)$ imply that $\calR^{\pd}_{\rhob_0} \simeq W(\FF)[[X,Y,Z]]$.
Moreover, combining these hypotheses with $p \mid\mid \ell + 1$, we get a set of generators of the cotangent space of $R^{\pd,\ell}_{\rhob_0}$.
All this information is crucially used to prove Theorem~\ref{thmd}.
\end{rem}

Recall that Mazur's conjecture (\cite{M}) predicts that the mod $p$ universal deformation ring of an absolutely irreducible $2$-dimensional representation of $G_{\QQ,Np}$ over some finite extension of $\FF_p$ has Krull dimension $3$. This also implies that the mod $p$ universal deformation ring is always a local complete intersection ring. From the theorem above, we find examples of mod $p$ universal pseudo-deformation rings of Krull dimension $3$ which are not local complete intersection rings. On the other hand, in \cite{BCh}, Bleher and Chinburg found examples of absolutely irreducible representations of profinite groups such that the corresponding universal deformation rings (in the sense of Mazur) are not locally complete intersection rings.

Finally, as an application, we prove an $R=\TT$ theorem for big $p$-adic Hecke algebras and pseudo-representation rings in \S\ref{heckesec} similar to the ones proved by B\"ockle in \cite{Bo4} (see Theorem~\ref{rthm} and Corollary~\ref{rcor} for more details). We also give examples where the hypotheses of our `big' $R=\TT$ theorem are satisfied. 

\subsection{Outline of the proof of main results}
Since $\chi_1 \neq \chi_2$, it follows, from \cite{BC} and \cite{Bel}, that a pseudo-representation $(t,d) : G \to R$ lifting $(\tr(\rhob_0),\det(\rhob_0))$ arises from a representation of $G$ taking values in a faithful Generalized Matrix Algebra (GMA) $A = \begin{pmatrix} R & B\\ C & R\end{pmatrix}$ over $R$.
The assumption $\dim(H^1(G,\chi^i))=1$ for some $i \in \{1,-1\}$ implies that $A$ can be chosen in such a way that $B$ is generated by at most $1$ element as an $R$-module.
Moreover, if $G = G_{\QQ,Np}$ or $G_{\QQ,N\ell p}$ and $\rhob_0$ is unramified at $\ell$, then this representation is tamely ramified at $\ell$.

Now if $B$ is a free $R$-module of rank $1$ (i.e. the annihilator of $B$ is $(0)$), then it follows that $A$ is isomorphic to a subalgebra of $M_2(R)$ which means $(t,d)$ arises from a representation over $R$.
Faithfulness of $A$ implies that this is equivalent to the annihilator of the ideal $I := m'(B \otimes C) \subset R$, obtained by multiplication of $B$ and $C$, being $(0)$.
Note that $I$ is the reducibility ideal of $(t,d)$ (in the sense of \cite{BC}). 

Now if $R$ is an integral domain and $(t,d)$ is not reducible, then it means $I \neq (0)$ and hence, the previous paragraph implies that $(t,d)$ arises from a representation over $R$.
Since $\dim(H^1(G,\chi^i))=1$, it follows, after changing the basis if necessary, that this representation is a deformation of a fixed reducible, non-split representation $\rhob_{x_0}$ whose semi-simplification is $\rhob_0$.
On the other hand, if $(t,d)$ is reducible, then we construct, using results and techniques of \cite{T}, a deformation of $\rhob_{x_0}$ to $R$ which gives rise to $(t,d)$. This proves the first part of Theorem~\ref{thma}.

To prove the second part of Theorem~\ref{thma}, we first use its hypotheses along with \cite[Theorem $3.3.1$]{WE} to prove that $R^{\pd}_{\rhob_0}$ is a quotient of a power series ring by an ideal generated by at most $2$ elements.
This description, along with some commutative algebra, is then used to prove that the annihilator of the reducibility ideal of the universal mod $p$ pseudo-deformation of $(\tr(\rhob_0),\det(\rhob_0))$ is trivial.
Combining this with the discussion above gives the second part of Theorem~\ref{thma}.

Note that Theorem~\ref{thma} relates certain quotients of $R^{\pd}_{\rhob_0}$ with the corresponding quotients of the deformation ring of $\rhob_{x_0}$.
We use the results of \S\ref{galsubsec} to conclude that these relations hold in the setting of Galois groups appearing in Theorem~\ref{thmb} and combine them with \cite[Theorem $4.7$]{Bo} to prove Theorem~\ref{thmb}.

To prove the first part of Theorem~\ref{thmc}, we combine results of \S\ref{galsubsec}, second part of Theorem~\ref{thma}, the relation between the tame inertia group and the Frobenius at $\ell$ and some basic commutative algebra to prove that $\mathcal{R}^{\pd,\ell}_{\rhob_0}$ is isomorphic to the universal deformation ring of $\rhob_{x_0}$ for $G_{\QQ,N\ell p}$. The result then follows from \cite[Theorem $4.7$]{Bo} and \cite[Theorem $2.4$]{Bo2}.
To prove the second part of Theorem~\ref{thmc}, we first find a set of generators of the tangent space of $\mathcal{R}^{\pd,\ell}_{\rhob_0}$.
Combining this with the relation between the tame inertia group and the Frobenius at $\ell$ and the first part of Theorem~\ref{thmc} yields the theorem.

The proof of Theorem~\ref{thmd} is carried out in several steps. We first find a set of generators of the cotangent space of $R^{\pd,\ell}_{\rhob_0}$ and then use the relation between the tame inertia group and the Frobenius at $\ell$ to prove that $(R^{\pd,\ell}_{\rhob_0})^{\red}$ is a quotient of  $\FF[[X,Y,Z,X_1,X_2]]/(X_1X_2,X_1Y,X_2Y)$.
We then prove that $R^{\pd,\ell}_{\rhob_0}$ has at least $3$ distinct prime ideals $P_0$, $P_1$ and $P_2$ such that $R^{\pd,\ell}_{\rhob_0}/P_j \simeq \FF[[x,y,z]]$ for all $0 \leq j \leq 2$ from which the first part of Theorem~\ref{thmd} follows.
Note that GMAs play a crucial role in obtaining the results mentioned above.
We then use the GMA corresponding to the universal mod $p$ pseudo-representation deforming $(\tr(\rhob_0),\det(\rhob_0))$ and the relation between the tame inertia group and the Frobenius at $\ell$ to get some relations satisfied by the generators of the cotangent space of $R^{\pd,\ell}_{\rhob_0}$ found above.
We then use some basic commutative algebra and first part of Theorem~\ref{thmd} to prove the second part of Theorem~\ref{thmd}.

\subsection{Wayfinding}
In \S\ref{prelimsec}, we collect definitions and background results that we use in the rest of the article. In \S\ref{pdsubsec}, we introduce the pseudo-deformation rings which we will be working with throughout the article. In \S\ref{gmasubsec}, we introduce the notion of Generalized Matrix Algebras (GMAs) and collect results which will be used in the rest of the article. In \S\ref{reducesubsec}, we introduce the notion of reducible pseudo-characters and study its properties. In \S\ref{nonsplitsubsec}, we review the definition and properties of the deformation ring of a residually reducible non-split representation. In \S\ref{galsubsec}, we prove some additional results for Galois groups which will be used later.
 In \S\ref{repsec}, we analyze when a pseudo-representation arises from a representation. In \S\ref{ramsec}, we study how the pseudo-deformation ring changes after enlarging the set of ramifying primes. In \S\ref{heckesec}, we apply results from \S\ref{ramsec} to prove an $R=\TT$ theorem and also give some examples where the hypotheses of the theorem are satisfied.

\subsection{Notations and conventions} 
For a pro-finite group $G$, we will use the following convention: all the representations, pseudo-representations, cohomology groups and $\Ext^i$ groups of $G$ that we will work with are assumed to be continuous unless mentioned otherwise. Given a representation $\rho$ of $G$ defined over $\FF$, we denote by $\dim(H^i(G,\rho))$, the dimension of $H^i(G,\rho)$ as a vector space over $\FF$. 

For a prime $q$, denote by $G_{\QQ_{q}}$ the absolute Galois group of $\QQ_{q}$ and by $I_q$, the inertia group at $q$. Denote the Frobenius element at $q$ by $\text{Frob}_q$.
For an integer $M$, denote by $G_{\QQ,Mp}$ the Galois group of a maximal algebraic extension of $\QQ$ unramified outside \{\text{primes }$q$ s.t. ${q} | Mp$\} $\cup$ $\{\infty\}$ over $\QQ$ and fix an embedding $i_{q,M} : G_{\QQ_q} \to G_{\QQ,Mp}$. For a fixed $M$, such an embedding is well defined upto conjugacy. 

For a representation $\rho$ of $G_{\QQ,Mp}$ denote by $\rho|_{G_{\QQ_q}}$ the representation $\rho \circ i_{q,M}$ of $G_{\QQ_q}$. Moreover, for an element $g \in G_{\QQ_q}$, we denote $\rho(i_{q,M}(g))$ by $\rho(g)$. If $\rho|I_q$ factors through the tame inertia quotient of $I_q$, then, given an element $g$ in the tame inertia group at $q$, we write $\rho(g)$ for $\rho(i_{q,M}(g'))$ where $g'$ is any lift of $g$ in $G_{\QQ_q}$. For a pseudo-representation $(t,d)$ of $G_{\QQ,Mp}$ denote by $(t|_{G_{\QQ_q}}, d|_{G_{\QQ_q}})$ the pseudo-representation $(t \circ i_{q,M}, d \circ i_{q,M})$ of $G_{\QQ_q}$.

We denote the mod $p$ cyclotomic character of $G_{\QQ,Mp}$ by $\omega_p$. For a prime $q$, we will also denote $\omega_p|_{G_{\QQ_q}}$ by $\omega_p$ by abuse of notation.
For a finite field $\FF$, we denote the ring of its Witt vectors by $W(\FF)$ and we will denote the Teichmuller lift of an element $a \in \FF$ to $W(\FF)$ by $\widehat a$.

For a local ring $R$ with residue field $\FF$, denote by $\tan(R)$ the tangent space of $R$ and denote by $\dim(\tan(R))$ the dimension of $\tan(R)$ as a vector space over $\FF$.

{\bf Acknowledgments:} I would like to thank Carl Wang-Erickson for helpful correspondence regarding \cite{WE} and the Introduction section of this article. I would also like to thank Gabor Wiese, Anna Medvedovsky and John Bergdall for many helpful conversations. I would like to thank the anonymous referee for many useful comments and suggestions which helped in improving the exposition.
Most of this work was done when the author was a postdoc at the University of Luxembourg.

\section{Preliminaries}
\label{prelimsec}

Even though we are primarily interested in the deformation rings of Galois pseudo-representations, we are going to take a slightly more general approach in this and the next section. To be precise, instead of $G_{\QQ,Np}$ and odd $\rhob_0$, we are going to consider a profinite group $G$ which satisfies the finiteness condition $\Phi_p$ given by Mazur in \cite[Section $1.1$]{M} and a continuous representation $\rhob_0 : G \to \GL_2(\FF)$ such that $\rhob_0=\chi_1 \oplus \chi_2$ with $\chi_1 \neq \chi_2$ and $\chi=\chi_1/\chi_2$. 

Most of the results that we state/prove in this section are well known.

\subsection{Psuedo-deformation rings}
\label{pdsubsec}
We now introduce the pseudo-deformation rings with which we will be studying for the rest of the article. Let $\mathcal{C}$ be the category whose objects are local complete noetherian rings with residue field $\FF$ and the morphisms between the objects are local morphisms of $W(\FF)$-algebras. Let $\mathcal{C}_0$ be the full sub-category of $\mathcal{C}$ consisting of local complete noetherian $\FF$-algebras with residue field $\FF$.

Now $\rhob_0$ is a $2$-dimensional representation of $G$ over $\FF$. Hence, $(\tr(\rhob_0),\det(\rhob_0)) : G \to \FF$ is a 2 dimensional pseudo-representation of $G$ over $\FF$. Let $D_{\rhob_0}$ be the functor from $\mathcal{C}$ to the category of sets which sends an object $R$ of $\mathcal{C}$ with maximal ideal $m_R$ to the set of continuous pseudo-representations $(t,d)$ of $G$ to $R$ such that $t \pmod{m_R} = \tr(\rhob_0)$ and $d \pmod{m_R} = \det(\rhob_0)$. Let $\bar D_{\rhob_0}$ be the restriction of $D_{\rhob_0}$ to the sub-category $ \mathcal{C}_0$.

From \cite{C}, it follows that the functors $D_{\rhob_0}$ and $\bar D_{\rhob_0}$ are representable by objects of $\mathcal{C}$ and $\mathcal{C}_0$, respectively. Let $R^{\pd}_{\rhob_0}$ and $\calR^{\pd}_{\rhob_0}$ be the local complete Noetherian rings with residue field $\FF$ representing $\bar D_{\rhob_0}$ and $D_{\rhob_0}$, respectively. So we have $\calR^{\pd}_{\rhob_0}/(p) \simeq R^{\pd}_{\rhob_0}$. Let $(t^{\univ}, d^{\univ})$ be the universal pseudo-representation of $G$ to $R^{\pd}_{\rhob_0}$ deforming $(\tr\rhob_0,\det\rhob_0)$. Let $(T^{\univ}, D^{\univ})$ be the universal pseudo-representation of $G$ to $\calR^{\pd}_{\rhob_0}$ deforming $(\tr\rhob_0,\det\rhob_0)$.

As $p$ is odd, it follows that a $2$-dimensional pseudo-representation $(t,d)$ of $G$ to an object $R$ of $\mathcal{C}$ is determined by $t$ which is a pseudo-character of dimension $2$ in the sense of Rouquier (\cite{R}) (see \cite[Section $1.4$]{BK}). Indeed if $p$ is odd and $(t,d) : G \to R$ is a $2$-dimensional pseudo-representation, then $d(g) = \frac{t(g)^2-t(g^2)}{2}$ for all $g \in G$. So, in this case, the theory of pseudo-representations is same as the theory of pseudo-characters. 

Hence, it follows that $\calR^{\pd}_{\rhob_0}$ (resp. $R^{\pd}_{\rhob_0}$) is the universal deformation ring and $T^{\univ}$ (resp. $t^{\univ}$) is the universal pseudo-character of the pseudo-character $\tr(\rhob_0)$ in the category $\mathcal{C}$ (resp. $\mathcal{C}_0$). Therefore, for simplicity,  we will be working with the residual pseudo-character $\tr(\rhob_0)$ and the universal pseudo-characters $T^{\univ}$ and $t^{\univ}$ deforming $\tr(\rhob_0)$ instead of working with the corresponding pseudo-representations. 

Denote the pseudo-character obtained by composing $t^{\univ}$ with the surjective map $R^{\pd}_{\rhob_0} \to (R^{\pd}_{\rhob_0})^{\red}$ by $t^{\univ,\red}$ and the pseudo-character obtained by composing $T^{\univ}$ with the surjective map $\calR^{\pd}_{\rhob_0} \to (\calR^{\pd}_{\rhob_0})^{\red}$ by $T^{\univ,\red}$.

We will frequently specialize to the case where $G=G_{\QQ,Np}$ and $\rhob_0$ is odd. However, even after specializing to this case, we will keep using the notation introduced above unless mentioned otherwise.

\subsection{Reminder on Generalized Matrix Algebras (GMAs)}
\label{gmasubsec}
In this subsection, we recall some standard definitions and results about Generalized Matrix Algebras which will be used frequently in the rest of the article. From now on, we will use the abbreviation GMA for Generalized Matrix Algebra. Our main references for this section are \cite[Section $2.2$]{Bel} (for GMAs of type $(1,1)$), \cite[Section $2.3$]{Bel} (for topological GMAs) and \cite[Chapter $1$]{BC} (for the general theory of GMAs).  For more information, we refer the reader to them.

We first recall the definition of a topological Generalized Matrix Algebra of type $(1,1)$. Let $R$ be a complete Noetherian local ring with maximal ideal $m_R$ and residue field $\FF$. So $R$ is a topological ring under the $m_R$-adic topology which we fix from now on. Let $A = \begin{pmatrix} R & B\\ C & R \end{pmatrix}$ be a topological GMA of type $(1,1)$ over $R$. This means the following:
\begin{enumerate}
\item $B$ and $C$ are topological $R$-modules,
\item An element of $A$ is of the form $\begin{pmatrix} a & b\\ c & d\end{pmatrix}$ with $a, d \in R$, $b \in B$ and $c \in C$,
\item There exists a continuous morphism $m' : B \otimes_{R} C \to R$ of $R$-modules such that for all $b_1$, $b_2 \in B$ and $c_1$, $c_2 \in C$, $m'(b_1 \otimes c_1)b_2 = m'(b_2 \otimes c_1)b_1$ and $m'(b_1 \otimes c_1)c_2 = m'(b_1 \otimes c_2)c_1$.
\end{enumerate}
So $A$ is a topological $R$-algebra with the addition given by $$\begin{pmatrix} a_1 & b_1\\ c_1 & d_1\end{pmatrix} + \begin{pmatrix} a_2 & b_2\\ c_2 & d_2\end{pmatrix} = \begin{pmatrix} a_1+a_2 & b_1+b_2\\ c_1+c_2 & d_1+d_2\end{pmatrix},$$ the multiplication given by $$\begin{pmatrix} a_1 & b_1\\ c_1 & d_1\end{pmatrix}.\begin{pmatrix} a_2 & b_2\\ c_2 & d_2\end{pmatrix} =\\ \begin{pmatrix} a_1a_2+m'(b_1 \otimes c_2) & a_1b_2+d_2b_1\\ d_1c_2+a_2c_1 & d_1d_2+m'(b_2 \otimes c_1)\end{pmatrix}$$ and the topology given by the topology on $R$, $B$ and $C$. 

For the rest of this article, GMA means topological GMA unless mentioned otherwise. By abuse of notation, we will always denote by $m'$ the multiplication map $B \otimes_{R} C \to R$ for any GMA and any $R$. From now on, given a profinite group $G$ and a GMA $A$, a representation $\rho : G \to A^*$ means a continuous homomorphism from $G$ to $A^*$ unless mentioned otherwise. If $\rho : G \to A^*$ is a representation, then we denote the $R$-submodule of $A$ generated by $\rho(G)$ by $R[\rho(G)]$. Note that $R[\rho(G)]$ is a subalgebra of $A$. If $\rho :G \to A^*$ is a representation such that $\rho(g) =\begin{pmatrix} a_g & b_g \\ c_g & d_g\end{pmatrix}$ for every $g \in G$, then we define $\tr(\rho) : G \to R$ by $\tr(\rho)(g) := a_g+d_g$. For a topological $R$-module $B$, we denote by $\Hom_{R}(B/m_RB,\FF)$ the set of all continuous $R$-module homomorphisms from $B/m_RB$ to $\FF$.

\begin{defi}
Let $A = \begin{pmatrix} R & B\\ C & R \end{pmatrix}$ be a GMA with the map $m' : B \otimes_{R} C \to R$ giving the multiplication in $A$. We say that $A$ is faithful if the following conditions hold:
\begin{enumerate}
\item If $b \in B$ and $m'(b \otimes c)=0$ for all $c \in C$, then $b=0$,
\item If $c \in C$ and $m'(b \otimes c)=0$ for all $b \in B$, then $c=0$.
\end{enumerate}
\end{defi}

\begin{defi}
We say that $A'$ is an $R$-sub-GMA of $A$ if there exists an $R$-submodule $B'$ of $B$ and an $R$-submodule $C'$ of $C$ such that $m'(B' \otimes C') \subset R$ and $A' = \begin{pmatrix} R & B'\\ C' & R\end{pmatrix}$ i.e. $A' = \left\{\begin{pmatrix} a & b\\ c & d\end{pmatrix} \in A | b \in B', c \in C'\right\}$ (see \cite[Section $2.2$]{Bel} for the definitions of sub-GMA and $R$-sub-GMA). Note that $A'$ is a sub-algebra of $A$ and hence, a GMA over $R$.
\end{defi}

\begin{defi}
 Let $R$ be an object of $\mathcal{C}$ and $t : G \to R$ be a pseudo-character deforming $\tr(\rhob_0)$. We will say that $t$ is reducible if there exists characters $\eta_1$, $\eta_2 : G \to R^*$ such that $t = \eta_1 + \eta_2$ and $\eta_i$ is a deformation of $\chi_i$ for $i=1,2$.
\end{defi}    

\begin{lem}
\label{genlem}
Let $R$ be a complete Noetherian local ring with maximal ideal $m_R$ and residue field $\FF$. Let $t : G \to R$ be a pseudo-character deforming $\tr(\rhob_0)$. Then, there exists a faithful GMA $A = \begin{pmatrix} R & B\\ C & R\end{pmatrix}$ and a representation $\rho : G \to A^*$ such that
\begin{enumerate}
\item\label{parti} For $g \in G$, if $\rho(g) = \begin{pmatrix} a_g & b_g\\ c_g & d_g\end{pmatrix}$, then $a_g \equiv \chi_1(g) \pmod{m_R}$, $d_g \equiv \chi_2(g) \pmod{m_R}$ and $t(g) = a_g+d_g$ (i.e. $t=\tr(\rho)$),
\item $m'(B \otimes_{R} C) \subset m_R$, where $m'$ is the map giving the multiplication in $A$,
\item $R[\rho(G)]=A$,
\item\label{partiv} $B$ and $C$ are finitely generated $R$-modules,
\item\label{partv} the minimal number of generators of $B$ as an $R$-module is at most $ \dim(H^1(G,\chi))$ and the minimal number of generators of $C$ as an $R$-module is at most $ \dim(H^1(G,\chi^{-1}))$,
\item\label{partvi} $t \pmod{I}$ is reducible, where $I := m'(B \otimes C)$.
\end{enumerate}
\end{lem}
\begin{proof}
As $\chi_1 \neq \chi_2$, $\rhob_0$ is residually multiplicity free. We have assumed that $G$ satisfies the finiteness condition. Hence, the existence of $A$ and $\rho$ with the properties \eqref{parti}-\eqref{partiv} follows from parts (i), (v), (vii) of \cite[Proposition $2.4.2$]{Bel}. 
To prove part~\eqref{partvi}, observe that $a_{gg'} \equiv a_{g}a_{g'} \pmod{I}$ and $d_{gg'} \equiv d_{g}d_{g'} \pmod{I}$.

The proof of part~\eqref{partv} of the lemma is same as that of \cite[Theorem $1.5.5$]{BC}. We only give a brief summary here. Given $f \in \Hom_{R}(B/m_RB, \FF)$, we get a morphism of $R$-algebras $f^* : A \to M_2(\FF)$, such that $$f^*\left(\begin{pmatrix} a & b\\ c & d\end{pmatrix}\right) = \begin{pmatrix} a \pmod{m} & f(b)\\ 0 & d \pmod{m}\end{pmatrix}.$$ From the first assumption, it follows that restriction of $f^*$ to $\rho(G)$ is an extension of $\chi_2$ by $\chi_1$ and hence, an element $\tilde f^*$ of $H^1(G,\chi)$ (see proof of \cite[Theorem $1.5.5$]{BC} for more details). So we get a linear map $j : \Hom_{R}(B/m_RB,\FF) \to H^1(G,\chi)$ sending $f$ to $\tilde f^*$.  Since $R[\rho(G)] =A$, we get that the map $j$ is injective. Hence, Nakayama's lemma gives the assertion about the number of generators of $B$. The assertion about the number of generators of $C$ follows similarly.
\end{proof}

\begin{rem}
\label{redrem}
It follows, from parts~\eqref{partv} and ~\eqref{partvi} of Lemma~\ref{genlem}, that if $H^1(G,\chi^i)=0$ for some $i \in \{1,-1\}$, then $T^{\univ}$ is reducible and hence, it arises from a $2$-dimensional $G$-representation over $\calR^{\pd}_{\rhob_0}$.
\end{rem}

Thus, from Lemma~\ref{genlem}, we see that a psuedo-character $t : G \to R$ deforming $\tr(\rhob_0)$ arises from a representation over $R$ if the GMA found in Lemma~\ref{genlem} corresponding to the tuple $(G,t,R)$ is isomorphic to a subalgebra of $M_2(R)$. 
\begin{lem}
\label{annihlem}
Let $A = \begin{pmatrix} R & B\\ C & R\end{pmatrix}$ be a faithful GMA over $R$ and $\rho : G \to A^*$ be a representation. Then:
\begin{enumerate}
\item If $y \in R$ is an element such that either $yB=0$ or $yC=0$, then $ym'(B \otimes C)=0$,
\item If $B$ is a free $R$-module of rank $1$, then there exists an $R$-algebra isomorphism $\phi$ between $A$ and the $R$-subalgebra of $M_2(R)$ given by $\begin{pmatrix}R & R\\ m'(B \otimes C) & R\end{pmatrix}$ such that $\phi(\tr(\rho(g))) = \tr(\rho(g))$ for every $g \in G$.
\end{enumerate}
\end{lem}
\begin{proof}
\begin{enumerate}
\item Note that $m' : B \otimes C \to R$ is a map of $R$-modules. Hence, for evey $y \in R$, $b \in B$ and $c \in C$, $m'(yb \otimes c) = m'(b \otimes yc) = ym'(b \otimes c)$. The first part follows immediately from this.
\item Fix a generator $\gamma$ of $B$. This choice gives us an $R$-module isomorphism $f_\gamma : B \to R$ such that $b=f_\gamma(b)\gamma$ for every $b \in B$. Consider the map $\tilde f : A \to A'$ which sends $\begin{pmatrix} a & b\\ c & d\end{pmatrix} \in A$ to $\begin{pmatrix} a & f_{\gamma}(b)\\ m'(\gamma \otimes c) & d\end{pmatrix}$. It is easy to check, using the facts that the multiplication map $m' : B \otimes_{R} C \to R$ is $R$-linear and $f_\gamma(b)m'(\gamma \otimes c)=m'(b \otimes c)$, that $\tilde f$ is a continuous homomorphism of $R$-algebras. Note that if $a \in A$, then $\tr(a) = \tr(\tilde f(a))$. This finishes the proof of the second part.
\end{enumerate}
\end{proof}

When $R$ is reduced, it turns out that any GMA representation comes `very close' to being a true representation. To be precise, every GMA representation over a reduced ring comes from a true representation over its total fraction field. We record this as a formal result below.

\begin{lem}
\label{redgmalem}
Let $R$ be a reduced complete Noetherian local ring with maximal ideal $m_R$ and residue field $\FF$. Let $K$ be the total fraction field of $R$. If $A = \begin{pmatrix} R & B\\ C & R\end{pmatrix}$ is a faithful GMA, then there exist fractional ideals $B'$ and $C'$ of $K$ and $R$-module isomorphisms $\phi : B \to B'$ and $\psi: C \to C'$ such that 
\begin{enumerate}
\item For all $b' \in B'$ and $c' \in C'$, $b'.c' \in R$, where $.$ denotes the multiplication in $K$,
\item If $A' = \begin{pmatrix} R & B'\\ C' & R\end{pmatrix} \subset M_2(K)$, then $A'$ is a $R$-sub-algebra of $M_2(K)$,
\item The map $\Phi : A \to A'$ given by $\Phi \left(\begin{pmatrix} a & b\\ c & d \end{pmatrix}\right) = \begin{pmatrix} a & \phi(b)\\ \psi(c) & d \end{pmatrix}$ is an isomorphism of $R$-algebras. 
\end{enumerate}
\end{lem}
\begin{proof}
This follows directly from \cite[Proposition 1.3.12]{BC}.
\end{proof}

\subsection{Reducibility properties of pseudo-characters}
\label{reducesubsec}
We will now define a reducible pseudo-character and study properties of it. We begin by computing tangent space dimension of $R^{\pd}_{\rhob_0}$ under some hypothesis.
\begin{lem}
\label{tandimlem}
Suppose $H^2(G,1)=0$. Let $k=\dim(H^1(G,1))$, $m= \dim (H^1(G,\chi)$) and $n=\dim (H^1(G,\chi^{-1}))$. Then $\dim(\tan(R^{\pd}_{\rhob_0})) = 2k+mn$.
\end{lem}
\begin{proof}
  Recall that $\Ext^1_{G} (\eta, \delta) \simeq H^1(G,\delta/\eta)$ and $\Ext^2_{G} (\eta, \eta) \simeq H^2(G,1)$ for any continuous characters $\eta$, $\delta : G \to \FF^\times$. 
 Now the lemma directly follows from \cite[Theorem $2$]{B} (see also \cite[Proposition 20]{BK}).
\end{proof}

\begin{lem}
\label{maxlem}
If $J$ is an ideal of $R^{\pd}_{\rhob_0}$ such that $t^{\univ} \pmod{J} = \tr(\rhob_0)$, then $J$ is the maximal ideal of $R^{\pd}_{\rhob_0}$.
\end{lem}
\begin{proof}
Let $f : R^{\pd}_{\rhob_0} \twoheadrightarrow R^{\pd}_{\rhob_0}/J$ be the natural surjective homomorphism. Let $g : R^{\pd}_{\rhob_0} \to R^{\pd}_{\rhob_0}/J$ be the morphism obtained by composing the natural surjective morphism $R^{\pd}_{\rhob_0} \to \FF$ with the map $\FF \to R^{\pd}_{\rhob_0}/J$ giving the $\FF$-algebra structure on $R^{\pd}_{\rhob_0}/J$. As $t^{\univ} \pmod{J} = \tr(\rhob_0)$, we see that $f \circ t^{\univ} = g \circ t^{\univ}$. Hence, by the universality of $R^{\pd}_{\rhob_0}$, we get that $f=g$. Therefore, we get that $J$ is the maximal ideal of $R^{\pd}_{\rhob_0}$.
\end{proof}

Before proceeding further, we introduce some more notation. Let $G^{\text{ab}}$ denote the continuous abelianization of $G$. 
\begin{lem}
\label{redlem}
Let $J$ be an ideal of $R^{\pd}_{\rhob_0}$ such that $t^{\univ} \pmod{J}$ is reducible. If $H^2(G,1) =0$ and $\dim(H^1(G,1))=k$, then $\dim(\tan(R^{\pd}_{\rhob_0}/J)) \leq 2k$ and Krull dimension of $R^{\pd}_{\rhob_0}/J$ is at most $2k$.
\end{lem}
\begin{proof}
Denote $R^{\pd}_{\rhob_0}/J$ by $R$ and $t^{\univ} \pmod{J}$ by $t'$ for the rest of the proof. Suppose $t'=\tilde\chi_1+\tilde\chi_2$, where $\tilde\chi_1$, $\tilde\chi_2 : G \to R^*$ are characters deforming $\chi_1$ and $\chi_2$, respectively.

As $H^2(G,1)=0$ and $\dim(H^1(G,1))=k$, we see that $\varprojlim_{i}G^{\text{ab}}/(G^{\text{ab}})^{p^i} \simeq \prod_{i=1}^{k}\ZZ_p$. Let $\{g_1,\cdots,g_k\}$ be a set of topological generators of the abelian pro-$p$ group  $\varprojlim_{i}G^{\text{ab}}/(G^{\text{ab}})^{p^i}$. For all $1 \leq i \leq k$, there exist $x_i$, $y_i \in R$ such that  $\tilde\chi_1(g_i) = \chi_1(g)(1+x_i)$ and $\tilde\chi_2(g_i) = \chi_2(g)(1+y_i)$. Let $I$ be the ideal of $R$ generated by the set $\{x_1,\cdots,x_k,y_1,\cdots,y_k\}$.

Since $\{g_1,\cdots,g_k\}$ is a set of topological generators of $\varprojlim_{i}G^{\text{ab}}/(G^{\text{ab}})^{p^i}$, we see that $t' \pmod{I} = \tr(\rhob_0)$. So, by Lemma~\ref{maxlem}, the kernel of the natural surjective map $R^{\pd}_{\rhob_0} \to R/I$ is the maximal ideal of $R^{\pd}_{\rhob_0}$ and hence, $I$ is the maximal ideal of $R$. This proves the claim about $\dim(\tan(R))$. The claim about the Krull dimension of $R$ follows directly from  $\dim(\tan(R)) \leq 2k$.
\end{proof}

\begin{rem}
Comparing Lemma~\ref{redlem} and Lemma~\ref{tandimlem}, we see that if $H^2(G,1)=0$, $H^1(G,\chi) \neq 0$ and $H^1(G,\chi^{-1}) \neq 0$, then $t^{\univ}$ is not reducible.
\end{rem}

\begin{rem}
Note that Lemma~\ref{redlem} is also true when $H^2(G,1) \neq 0$ but we don't prove it here as we will mostly restrict ourselves to the case $H^2(G,1)=0$ in what follows.
\end{rem}

\subsection{Deformation rings of reducible non-split representations}
\label{nonsplitsubsec}

We have $\rhob_0 = \chi_1 \oplus \chi_2$ for some \emph{distinct} characters $\chi_1$, $\chi_2  : G \to \FF^\times$. Let $\chi = \chi_1/\chi_2$. Thus, $\chi : G \to \FF^\times$ is a non-trivial character. For a non-zero element $x \in H^1(G, \chi)$, denote by $\rhob_x$ the corresponding representation of $G$. So $\rhob_x : G \to \GL_2(\FF)$ is such that $\rhob_x = \begin{pmatrix} \chi_1 & *\\ 0 & \chi_2 \end{pmatrix}$ where $*$ corresponds to $x$. Similarly, for a non-zero element  $y \in H^1(G, \chi^{-1})$, denote by $\rhob_y$ the corresponding representation of $G$.

Let $x \in H^1(G, \chi^i)$ with $i \in \{1,-1\}$ be a non-zero element. Denote by $\calR^{\defo}_{\rhob_x}$ the universal deformation ring of $\rhob_x$ in the category $\mathcal{C}$ in the sense of Mazur (\cite{M}). Note that, for a non-zero $x \in H^1(G, \chi^i)$ with $i \in \{1,-1\}$, the centralizer of the image of $\rhob_x$ is exactly the set of scalar matrices as $\chi \neq 1$. Hence, the existence of $\calR^{\defo}_{\rhob_x}$ follows from \cite{M} and \cite{Ra}. Let $R^{\defo}_{\rhob_x}$ be the universal deformation ring of $\rhob_x$ in characteristic $p$. So we have $\calR^{\defo}_{\rhob_x}/(p) \simeq R^{\defo}_{\rhob_x}$. Let $\rho^{\univ}_x : G \to \GL_2(\calR^{\defo}_{\rhob_x})$ be the universal deformation of $\rhob_x$.

We will frequently specialize to the case where $G=G_{\QQ,Np}$. However, even after specializing to this case, we will keep using the notation introduced above unless mentioned otherwise.

\begin{lem}
\label{shortlem}
Let $x \in H^1(G,\chi^i)$, with $i \in \{1,-1\}$, be a non-zero element. Let $\dim (H^1(G, \chi^i)) =m$, $\dim(H^1(G,\chi^{-i})) =n$ and $\dim (H^1(G,1)) = k$. Then $\dim(H^1(G,\ad(\rhob_x))) = \dim(\tan(R^{\defo}_{\rhob_x})) \leq m+n+2k-1$.
\end{lem}
\begin{proof}
Recall that $\dim(\tan(R^{\defo}_{\rhob_x})) = \dim H^1(G,\text{ad}(\rhob_x))$ (see \cite{M}). As $p$ is odd, $\text{ad}(\rhob_x) = 1 \oplus \text{ad}^0(\rhob_x)$. We have the following two exact sequences of $G$-modules:
\begin{enumerate}
\item $0 \to \chi^i \to \text{ad}^0(\rhob_x) \to V \to 0$,
\item $0 \to 1 \to V \to \chi^{-i} \to 0$.
\end{enumerate}
So, from the second short exact sequence, we get $\dim(H^1(G,V)) \leq \dim(H^1(G,1)) + \dim(H^1(G,\chi^{-i}))=k+n$.
Since $\dim(H^0(G, V))=1$, the exact sequence of cohomology groups arising from the first short exact sequence gives $\dim(H^1(G,\text{ad}^0(\rhob_x))) \leq \dim(H^1(G,V)) + \dim(H^1(G,\chi^{i}))-1$. 
Combining these two inequalities, we get that $\dim(H^1(G,\text{ad}^0(\rhob_x))) \leq k+m+n-1$ and hence, $\dim(H^1(G,\text{ad}(\rhob_x)))=\dim(H^1(G,\text{ad}^0(\rhob_x)))+\dim(H^1(G,1))=\dim(H^1(G,\text{ad}^0(\rhob_x)))+k \leq 2k+m+n-1$.
\end{proof}

\begin{lem}
Suppose $\dim(H^1(G,\chi)) =1$. Then for any non-zero $x,x' \in H^1(G,\chi)$, $\calR^{\defo}_{\rhob_x} \simeq \calR^{\defo}_{\rhob_{x'}}$.
\end{lem}
\begin{proof}
As $\dim(H^1(G,\chi)) =1$, if $x$, $x' \in H^1(G,\chi)$ are both non-zero, then $x'=ax$ for some non-zero $a \in \FF$. Therefore, by conjugating $\rhob_x$ by the matrix $\begin{pmatrix} a & 0\\ 0 & 1\end{pmatrix}$, we get $\rhob_{x'}$. Hence, we see that $\calR^{\defo}_{\rhob_x} \simeq \calR^{\defo}_{\rhob_{x'}}$.
\end{proof}

Note that given any non-zero element $x \in H^1(G, \chi^i)$ with $i \in \{1,-1\}$, one has a map $\Psi_x : \calR^{\pd}_{\rhob_0} \to \calR^{\defo}_{\rhob_x}$ induced by the trace of $\rho^{\univ}_x$. We now recall a result due to Kisin (\cite[Corollary $1.4.4(2)$]{Ki}) on the nature of the map $\Psi_x$:
\begin{lem}
\label{surjlem}
If $\dim(H^1(G,\chi^i)) =1$ for some $i \in \{1,-1\}$ and $x \in H^1(G,\chi^i)$ is a non-zero element, then the map $\Psi_x : \calR^{\pd}_{\rhob_0} \to \calR^{\defo}_{\rhob_x}$ is surjective.
\end{lem} 

\subsection{Some additional results for Galois groups}
\label{galsubsec}
We now turn our attention to the case when $G=G_{\QQ,Mp}$ for some integer $M$ and state some results which will be used later. Throughout this subsection, we assume that $N$ is an integer not divisible by $p$, $\rhob_0 : G_{\QQ,Np} \to \GL_2(\FF)$ is odd and $\rhob_0 = \chi_1 \oplus \chi_2$ where $\chi_i : G_{\QQ,Np} \to \FF^{\times}$ is a character for $i=1,2$.
\subsubsection{Dimension of certain Galois cohomology groups}
We begin by computing dimension of certain Galois cohomology groups. These computations will be used later mainly to compute dimensions of tangent spaces of deformation and pseudo-deformation rings.

\begin{lem}
\label{cohomlem}
 Let $\ell$ be a prime such that $\ell \nmid Np$. Let $\chi : G_{\QQ,Np} \to \FF^{\times}$ be an odd character. Then, the following holds:
\begin{enumerate}
\item If $p \nmid \phi(N)$, then $\dim (H^1(G_{\QQ,Np},1))=1$ and  $\dim (H^2(G_{\QQ,Np},1))=0$,
\item $\dim (H^1(G_{\QQ,Np},\chi)) > 0$ and $\dim (H^2(G_{\QQ,Np},\chi))= \dim (H^1(G_{\QQ,Np},\chi))-1$,
\item If $\dim (H^1(G_{\QQ,Np},\chi))=1$ and $\chi|_{G_{\QQ_\ell}} = \omega_p$, then $\dim (H^1(G_{\QQ,N\ell p},\chi))=2$,
\item If $\dim (H^1(G_{\QQ,Np},\chi))=1$ and $\chi|_{G_{\QQ_\ell}} \neq \omega_p$, then $\dim (H^1(G_{\QQ,N\ell p},\chi))=1$,
\item $\dim(H^1(G_{\QQ,N\ell p},\chi)) - \dim(H^1(G_{\QQ,N p},\chi)) \leq 1$.
\end{enumerate}
\end{lem}
\begin{proof}
 As we have assumed $p \nmid \phi(N)$ in the first part, the Kronecker-Weber theorem implies that $\dim(H^1(G_{\QQ,Np},1))=1$. So, from the global Euler characteristic formula, we get $H^2(G_{\QQ,Np},1)=0$ which proves the first part.

 Since $\chi$ is assumed to be odd, the global Euler characteristic formula implies $\dim(H^1(G_{\QQ,Np},\chi)) - \dim(H^2(G_{\QQ,Np},\chi)) = 1$ which means $\dim (H^1(G_{\QQ,Np},\chi)) > 0$. This proves the second part.

 If $\chi = \omega_p$, then by Kummer theory, $\dim(H^1(G_{\QQ,Np},\omega_p)) = 1 +$ number of distinct primes dividing $N$ (see the proof of \cite[Proposition $24$]{D} and the remark after it). Thus, $\dim(H^1(G_{\QQ,N\ell p},\omega_p)) = 1 + \dim(H^1(G_{\QQ,Np},\omega_p))$. Therefore, if $\dim(H^1(G_{\QQ,Np},\omega_p)) = 1$ then $N=1$ and hence,  $\dim(H^1(G_{\QQ,N\ell p},\omega_p)) = 2$. This proves the third part for $\chi=\omega_p$.

If $\chi \neq \omega_p$ and $\chi$ is odd, then, by the Greenberg-Wiles version of the Poitou-Tate duality (\cite[Theorem $2$]{Wa}), we see that $$\dim(H^1(G_{\QQ,Np},\chi)) = \dim(H^1_0(G_{\QQ,Np},\chi^{-1}\omega_p)) + 1+ \sum_{q | Np}\dim(H^0(G_{\QQ_q},\chi^{-1}\omega_p|_{G_{\QQ_{q}}})),$$ where $$H^1_0(G_{\QQ,Np},\chi^{-1}\omega_p) =  \ker(H^1(G_{\QQ,Np},\chi^{-1}\omega_p) \to \prod_{q|Np} H^1(G_{\QQ_q},\chi^{-1}\omega_p|_{G_{\QQ_{q}}})).$$

Therefore, we get that $\dim(H^1(G_{\QQ,N\ell p},\chi)) - \dim(H^1(G_{\QQ,Np},\chi)) \leq \dim(H^0(G_{\QQ_\ell},\chi^{-1}\omega_p|_{G_{\QQ_{\ell}}})) \leq 1$ which proves the last part of the lemma.

Now from the equality above, we see that if $\dim(H^1(G_{\QQ,N p},\chi)) = 1$, then $H^1_0(G_{\QQ,Np},\chi^{-1}\omega_p) = 0$ and hence, $H^1_0(G_{\QQ,N\ell p},\chi^{-1}\omega_p) = 0$. Hence, we get $\dim(H^1(G_{\QQ,N\ell p},\chi)) - \dim(H^1(G_{\QQ,Np},\chi)) = \dim(H^0(G_{\QQ_\ell},\chi^{-1}\omega_p|_{G_{\QQ_{\ell}}}))$. This finishes the proof of the remaining parts of the lemma.
\end{proof}

\begin{lem}
\label{cohomlem2}
Suppose $p \nmid \phi(N)$. Let $\ell$ be a prime such that $\ell \nmid Np$ and $p \nmid \ell-1$. Let $\rho : G_{\QQ,Np} \to \GL_2(\FF)$ be an odd representation such that $\text{End}_{G_{\QQ,Np}}(\rho) = \FF$. Then, the following holds:
\begin{enumerate}
\item $\dim (H^2(G_{\QQ,Np},\ad(\rho)))= \dim (H^1(G_{\QQ,Np},\ad(\rho)))-3$,
\item If $p | \ell +1$, $\dim( H^1(G_{\QQ,Np},\ad(\rho))) = 3$ and $\rho|_{G_{\QQ_{\ell}}} = \eta \oplus \omega_p\eta$, then $\dim(H^1(G_{\QQ,N\ell p},\ad(\rho))) = 5$.
\end{enumerate}
\end{lem}
\begin{proof}
 As $\rho$ is assumed to be odd and $\text{End}_{G_{\QQ,Np}}(\rho) = \FF$, the first part of the lemma follows directly from the global Euler characteristic formula. 

 To prove the second part of the lemma, observe that  $\dim(H^1(G_{\QQ,Np},\text{ad}^0(\rho)))=2$ because we are assuming $p \nmid \phi(N)$ and $\dim(H^1(G_{\QQ,Np},\text{ad}(\rho)))=3$.
Now, by the Greenberg-Wiles version of the Poitou-Tate duality (\cite[Theorem $2$]{Wa}), we get that $\dim(H^1(G_{\QQ,Np},\text{ad}^0(\rho)))\geq \dim(H^1_0(G_{\QQ,Np},(\text{ad}^0(\rho))^*\otimes\omega_p)) + \dim(H^1(G_{\QQ_p},\text{ad}^0(\rho)))-\dim(H^0(G_{\QQ_p},\text{ad}^0(\rho))) + \dim(H^0(G_{\QQ},\text{ad}^0(\rho)))-\dim(H^0(G_{\QQ},(\text{ad}^0(\rho))^* \otimes \omega_p))+ \dim(H^1(G_{\infty},\text{ad}^0(\rho)))- \dim(H^0(G_{\infty},\text{ad}^0(\rho)))$, where $$H^1_0(G_{\QQ,Np},(\text{ad}^0(\rho))^*\otimes\omega_p) =  \ker(H^1(G_{\QQ,Np},(\text{ad}^0(\rho))^*\otimes\omega_p) \to \prod_{q|Np} H^1(G_{\QQ_q},(\text{ad}^0(\rho))^*\otimes\omega_p|_{G_{\QQ_{q}}})).$$

\begin{enumerate}
\item Note that $H^0(G_{\QQ},\text{ad}^0(\rho))=0$. As $\rho$ is odd, $\dim(H^0(G_{\infty},\text{ad}^0(\rho)))=1$. As $|G_{\infty}|=2$ and $p>2$, we have $H^1(G_{\infty},\text{ad}^0(\rho))=0$,
\item Suppose $\dim(H^0(G_{\QQ},(\text{ad}^0(\rho))^* \otimes \omega_p))=k'$. By the local Euler characteristic formula, $\dim(H^1(G_{\QQ_p},\text{ad}^0(\rho)|_{G_{\QQ_p}}))-\dim(H^0(G_{\QQ_p},\text{ad}^0(\rho)|_{G_{\QQ_p}}))=3+\dim(H^0(G_{\QQ_p},(\text{ad}^0(\rho))^* \otimes \omega_p|_{G_{\QQ_p}}))$ $\geq 3+k'$.
\end{enumerate}

Hence, we get that $\dim(H^1(G_{\QQ,Np},\text{ad}^0(\rho))) \geq 3+k'-1-k' + \dim(H^1_0(G_{\QQ,Np},(\text{ad}^0(\rho))^*\otimes\omega_p)) =2 + \dim(H^1_0(G_{\QQ,Np},(\text{ad}^0(\rho))^*\otimes\omega_p))$. As $\dim(H^1(G_{\QQ,Np},\text{ad}^0(\rho)))=2$, we get that $H^1_0(G_{\QQ,Np},(\text{ad}^0(\rho))^*\otimes\omega_p)=0$.

Hence, we get that for any prime $\ell$, $\dim(H^1(G_{\QQ,N\ell p},\text{ad}^0(\rho)))$ $= \dim(H^1(G_{\QQ,Np},\text{ad}^0(\rho))) + \dim(H^0(G_{\QQ_{\ell}}, (\text{ad}^0(\rho))^* \otimes \omega_p|_{G_{\QQ_{\ell}}}))$. Now let $\ell$ be a prime such that $\ell \equiv -1 \pmod{p}$ and $\rho|_{G_{\QQ_{\ell}}}= \eta \oplus \omega_p \eta$. In this case $\omega_p|_{G_{\QQ_{\ell}}}=\omega_p^{-1}|_{G_{\QQ_{\ell}}}$. Therefore, $\text{ad}^0(\rho)|_{G_{\QQ_{\ell}}} \simeq 1 \oplus \omega_p|_{G_{\QQ_{\ell}}} \oplus \omega_p|_{G_{\QQ_{\ell}}}$ and we get that $\dim(H^1(G_{\QQ,N\ell p},\text{ad}^0(\rho)))$ $= \dim(H^1(G_{\QQ,Np},\text{ad}^0(\rho))) + 2$ $=2+2=4$. As $p \nmid \phi(N\ell)$, we have $\dim(H^1(G_{\QQ,N\ell p},\text{ad}(\rhob_x))) = 5$. 
\end{proof}

\subsubsection{GMA results for $G_{\QQ,N\ell p}$}
We now view $\rhob_0$ as a representation of $G_{\QQ, N\ell p}$ for some prime $\ell \nmid Np$. We will state results which will be used later while analyzing how pseudo-deformation rings change after allowing ramification at an additional prime.
 For a prime $\ell$, denote by $\tilde\ell$ be the Teichmuller lift of $\ell \pmod{p}$ in $\ZZ_p$. So $\ell/\tilde\ell \in 1+p\ZZ_p$. Recall that, for $\alpha \in \FF$, we denoted its Teichmuller lift in $W(\FF)$ by $\hat\alpha$. 
\begin{lem}
\label{gmalem}
Let $R$ be a complete Noetherian local ring with maximal ideal $m_R$ and residue field $\FF$. Let $\ell$ be a prime such that $\ell \nmid Np$ and $\chi|_{G_{\QQ_{\ell}}} \neq 1$. Let $t : G_{\QQ,N\ell p} \to R$ be a pseudo-character deforming $\tr(\rhob_0)$.  Let $g_{\ell}$ be a lift of $\text{Frob}_{\ell}$ in $G_{\QQ_{\ell}}$. Then, there exists a faithful GMA $A = \begin{pmatrix} R & B\\ C & R\end{pmatrix}$ and a representation $\rho : G_{\QQ,N\ell p} \to A^*$ satisfying the properties of Lemma~\ref{genlem} such that
\begin{enumerate}
 \item $t=\tr(\rho)$ and $\rho(g_{\ell}) = \begin{pmatrix} \widehat{\chi_1(\text{Frob}_{\ell})}(1+a) & 0\\ 0 & \widehat{\chi_2(\text{Frob}_{\ell})}(1+d)\end{pmatrix}$,
\item $R[\rho(G_{\QQ_{\ell}})]$ is a sub $R$-GMA of $A$,
\item\label{bhaag3} $\rho|_{I_{\ell}}$ factors through the $\ZZ_p$-quotient of the tame inertia group at $\ell$.
\end{enumerate}
Moreover, if $\ell/\tilde\ell$ is a topological generator of $1+p\ZZ_p$ and $J$ is an ideal of $R$ such that $t \pmod{J}$ is reducible, then the ideal generated by $p$, $a$, $d$ and $J$ is the maximal ideal of $R$.
\end{lem}
\begin{proof}
Since $\rhob_0$ is assumed to be odd, we get that $\chi_1 \neq \chi_2$ and $\rhob_0$ is residually multiplicity free. We know that $G_{\QQ,N\ell p}$ satisfies the finiteness condition. Moreover, we are assuming that $\chi|_{G_{\QQ_\ell}} \neq 1$ which means $\rhob_0(g_\ell)$ has distinct eigenvalues. 
The existence $A$ and $\rho$ satisfying properties of Lemma~\ref{genlem} and the first part of the lemma follow from parts (i), (iii), (v) and (vii) of \cite[Proposition $2.4.2$]{Bel}. As $a \not\equiv d \pmod{m_R}$, the claim that $R[\rho(G_{\QQ_{\ell}})]$ is a sub $R$-GMA of $A$  follows from \cite[Lemma $2.4.5$]{Bel}.

To prove the third part of the lemma, let $K_0$ be the maximal extension of $\QQ$ unramified outside the set of primes dividing $N\ell p$ and $\infty$. So $G_{\QQ,N\ell p}=\text{Gal}(K_0/\QQ)$. Let $K$ be the extension of $\QQ$ fixed by $\ker(\rhob_0)$. So $K$ is a sub-extension of $K_0$ and $\ell$ is unramified in $K$. By \cite[Lemma $3.8$]{C}, the pseudo-character $ t$ factors through $G_{\QQ,N\ell p}/H$, where $H \subset \text{Gal}(K_0/K)$ is the smallest closed normal subgroup of $G_{\QQ,N\ell p}$ such that $\text{Gal}(K_0/K)/H$ is a pro-$p$ quotient of $\text{Gal}(K_0/K)$. 

Let $g \in H$. As $ t$ factors through $G_{\QQ,N\ell p}/H$, we get $ t(xg)= t(x)$ for all $x \in G_{\QQ,N\ell p}$. Thus, we have $\tr(\rho(g'g))=\tr(\rho(g'))$ for all $g' \in G_{\QQ,N\ell p}$. Let $A = \begin{pmatrix} R & B\\ C & R \end{pmatrix}$ and $\rho(g) =\begin{pmatrix} a & b \\ c & d\end{pmatrix}$. As $R[\rho(G_{\QQ,N\ell p})]=A$, we get $\tr\left(\begin{pmatrix} a' & b'\\ c' & d'\end{pmatrix} . \begin{pmatrix} a & b\\ c & d\end{pmatrix}\right) = \tr\left( \begin{pmatrix} a' & b'\\ c' & d'\end{pmatrix}\right)$ for all $\begin{pmatrix} a' & b'\\ c' & d'\end{pmatrix} \in A$. Putting $a'=1$ and $b'=c'=d'=0$ gives us $a=1$. Putting $d'=1$ and $b'=c'=a'=0$ gives us $d=1$. Putting $b'=a'=d'=0$, we get $m'(b \otimes c')=0$ for all $c' \in C$. So faithfulness of $A$ implies $b=0$. Similarly, putting $c'=a'=d'=0$ gives us $c=0$ which proves that $\rho(g)$ is identity.

 As $\ell$ is unramified in $K$, we get that $I_\ell \subset \text{Gal}(K_0/K)$. Therefore, we see that $\rho|_{I_{\ell}}$ factors through the $\ZZ_p$-quotient of the tame inertia group at $\ell$. 

We will now prove the remaining part of the Lemma.
Let $I$ be the ideal of $R$ generated by $p$, $a$, $d$ and $J$ and $t' =t \pmod{I}$. Suppose $\psi_1$, $\psi_2 : G_{\QQ,N\ell p} \to (R/I)^*$ are characters deforming $\chi_1$ and $\chi_2$ such that $t' = \psi_1+\psi_2$.
As $a,d \in I$, we get that $t'(g_\ell) = \chi_1(\text{Frob}_\ell)+\chi_2(\text{Frob}_\ell)$ and $\frac{t'(g_\ell)^2-t'(g_\ell^2)}{2} = \chi_1\chi_2(\text{Frob}_\ell)$.
On the other hand, we have $t'(g_\ell) = \psi_1(g_\ell)+\psi_2(g_\ell)$ and $\frac{t'(g_\ell)^2-t'(g_\ell^2)}{2} = \psi_1\psi_2(g_\ell)$.
Therefore, $\psi_1(g_\ell)$ and $\psi_2(g_\ell)$ are roots of the polynomial $f(x) = x^2-(\chi_1(\text{Frob}_\ell)+\chi_2(\text{Frob}_\ell))x+\chi_1\chi_2(\text{Frob}_\ell) \in R/I[x]$.
As $\chi|_{G_{\QQ_{\ell}}} \neq 1$, $\chi_1(\text{Frob}_\ell) \neq \chi_2(\text{Frob}_\ell)$. Hence, from Hensel's lemma, we get that $\psi_i(g_\ell) = \chi_i(\text{Frob}_\ell)$ for $i=1,2$.

Thus, for $i=1, 2$, $\psi_i$ is a deformation of $\chi_i$ with $\psi_i(g_{\ell})=\chi_i(\text{Frob}_{\ell})$. As $p \nmid \ell-1$, both $\psi_1$ and $\psi_2$ are unramified at $\ell$. Since $p \nmid \phi(N\ell)$ and $\ell/\tilde\ell$ is a topological generator of $1+p\ZZ_p$, it follows that the image of $g_{\ell}$ in $\varprojlim_{i}G^{\text{ab}}_{\QQ,N\ell p}/(G^{\text{ab}}_{\QQ,N\ell p})^{p^i} \simeq \ZZ_p$ is a topological generator of $\varprojlim_{i}G^{\text{ab}}_{\QQ,N\ell p}/(G^{\text{ab}}_{\QQ,N\ell p})^{p^i}$. Therefore, it follows, from \cite[Section $1.4$]{M}, that $\psi_1 = \chi_1$ and $\psi_2=\chi_2$. Thus, we have $t' = \tr(\rhob_0)$. Since the map $R^{\pd,\ell}_{\rhob_0} \to R$ induced by $t$ is surjective, we get, from Lemma~\ref{maxlem}, that $I$ is the maximal ideal of $R$.
\end{proof}

\begin{lem}
\label{gengmalem}
Suppose $\dim(H^1(G_{\QQ,Np},\chi)) = \dim(H^1(G_{\QQ,Np},\chi^{-1})) =1$. Let $\ell$ be a prime such that $\ell \equiv -1 \pmod{p}$ and $\chi|_{G_{\QQ_{\ell}}} = \omega_p|_{G_{\QQ_{\ell}}}$. Let $R$ be a complete Noetherian local ring with maximal ideal $m_R$ and residue field $\FF$. Let $t : G_{\QQ,N\ell p} \to R$ be a pseudo-character deforming $\tr(\rhob_0)$.
Let $A = \begin{pmatrix} R & B\\ C & R \end{pmatrix}$ be the GMA associated to $t$ in Lemma~\ref{gmalem} and $\rho : G_{\QQ,N\ell p} \to A^*$ be the corresponding representation given by Lemma~\ref{gmalem}. Let $i_{\ell}$ be a topological generator of $\ZZ_p$-quotient of $I_{\ell}$ and suppose $\rho(i_{\ell})=\begin{pmatrix} a & b\\ c & d\end{pmatrix}$. Then:
\begin{enumerate}
\item Both $B$ and $C$ are generated by at most $2$ elements,
\item There exist $b' \in B$ and $c' \in C$ such that $B$ and $C$ are generated by $\{b,b'\}$ and $\{c,c'\}$ as $R$-modules, respectively.
\end{enumerate}
\end{lem}
\begin{proof}
As $\dim(H^1(G_{\QQ,Np},\chi)) = \dim(H^1(G_{\QQ,Np},\chi^{-1})) =1$, $\ell \equiv -1 \pmod{p}$ and $\chi|_{G_{\QQ_{\ell}}} = \omega_p|_{G_{\QQ_{\ell}}}$, Lemma~\ref{cohomlem} implies that $\dim(H^1(G_{\QQ,N\ell p},\chi)) = \dim(H^1(G_{\QQ,N\ell p},\chi^{-1})) =2$. The first part of the lemma now follows from part~\eqref{partv} of Lemma~\ref{genlem}.

 By Lemma~\ref{gmalem}, $\rho(i_{\ell})$ is well defined and $\rho(I_{\ell})$ is generated by $\rho(i_{\ell})$. Let $j_1 :  \Hom_{R}(B/ m_RB,\FF) \to H^1(G_{\QQ,N\ell p},\chi)$ and $j_2 :  \Hom_{R}(C/ m_RC,\FF) \to H^1(G_{\QQ,N\ell p},\chi^{-1})$ be the injective maps obtained in the proof of part~\eqref{partv} of Lemma~\ref{genlem}. Let $y$ be an element of the subspace $\Hom_{R}(B/R.b + m_RB,\FF)$ of $\Hom_{R}(B/ m_RB,\FF)$. So, $j_1(y)$ is an element of $H^1(G_{\QQ,N\ell p},\chi)$ such that $j_1(y)(I_{\ell})=0$ i.e. $j_1(y)$ is unramified at $\ell$. Thus, $j_1(y)$ lies in the image of the injective map $H^1(G_{\QQ,Np},\chi) \to H^1(G_{\QQ,N\ell p},\chi)$. Hence, $\dim(\Hom_{R}(B/R.b + m_RB,\FF)) \leq \dim(H^1(G_{\QQ,Np},\chi))=1$, Therefore, by Nakayama's lemma, $B/R.b$ is generated by at most $1$ element. By the same logic, we also get that $C/R.c$ is generated by at most $1$ element. So if $B = R.b$, then we can take $b'=0$. Otherwise, $B/R.b$ is generated by one element and let $b'$ be a lift of the generator in $B$. Thus, $\{b,b'\}$ generates $B$ in both the cases. The lemma for $C$ and $c$ follows similarly.
\end{proof}

\section{Comparison between $\calR^{\MakeLowercase{{pd}}}_{\rhob_0}$ and $\calR^{\MakeLowercase{{def}}}_{\rhob_x}$}
\label{repsec}

In this section, we will explore the question of determining when the universal pseudo-character $T^{\univ}$ comes from a representation defined over $\calR^{\pd}_{\rhob_0}$. 
We do this by first assuming the existence of such a representation to study its implications. Then, we will study if the necessary conditions found this way are sufficient for the existence of such a representation and its consequences on the relationship between $\calR^{\pd}_{\rhob_0}$ and $\calR^{\defo}_{\rhob_x}$. Note that, from Remark~\ref{redrem}, we already know that $T^{\univ}$ comes from a representation if either $H^1(G,\chi)$ or $H^1(G,\chi^{-1})$ is $0$. Hence, for the rest of the article, we are going to assume that both $H^1(G,\chi)$ and $H^1(G,\chi^{-1})$ are non-zero. Note that, when $G= G_{\QQ,Np}$ and $\rhob_0$ is odd, this assumption is satisfied by Lemma~\ref{cohomlem}. In the last subsection, we state the implications of the main results found in the general scenario for the case $G=G_{\QQ,Np}$.

\subsection{Necessary condition for $t^{\univ}$ to come from a representation}
\label{necsubsec}
The existence of a representation over $\calR^{\pd}_{\rhob_0}$ with trace $T^{\univ}$ implies that $t^{\univ}$ is the trace of a representation defined over $R^{\pd}_{\rhob_0}$.
We will first assume the existence of a representation over $R^{\pd}_{\rhob_0}$ with trace $t^{\univ}$ to relate the rings $R^{\pd}_{\rhob_0}$ and $R^{\defo}_{\rhob_x}$. Specifically, we will compare the dimensions of their tangent spaces to get the necessary conditions for the existence of the required representation. This will give us a necessary condition for $T^{\univ}$ to be the trace of a representation. 

\begin{prop}
\label{repnprop}
Suppose $H^2(G,1)=0$. If there exists a continuous representation $\rho : G \to \GL_2(R^{\pd}_{\rhob_0})$ such that $\tr\rho = t^{\univ}$, then either $\dim (H^1(G, \chi)) =1$ or $\dim (H^1(G,\chi^{-1})) =1$.
\end{prop}
\begin{proof}
 From Lemma~\ref{tandimlem}, we know that $\dim(\tan(R^{\pd}_{\rhob_0})) = 2k+mn$. As $m \neq 0$ and $n \neq 0$, $\dim(\tan(R^{\pd}_{\rhob_0})) > 2k$. Let $\mathfrak{m}$ be the maximal ideal of $R^{\pd}_{\rhob_0}$.

Suppose there exists a continuous representation $\rho : G \to \GL_2(R^{\pd}_{\rhob_0})$ such that $\tr\rho = t^{\univ}$. Let $\rhob$ be its reduction modulo $\mathfrak{m}$. As $\tr\rhob = \tr\rhob_0$, it follows, from the Brauer-Nesbitt theorem, that $\rhob$ is isomorphic over $\FF$ to either $\rhob_0$ or $\rhob_x$ for some $x \in H^1(G,\chi)$ or $H^1(G,\chi^{-1})$ with $x \neq 0$. 

Suppose $\rhob \simeq \rhob_0$. So, by changing the basis if necessary, we can assume that $\rhob = \rhob_0$. For $g \in G$, let $\rho(g) = \begin{pmatrix} a_g & b_g\\ c_g & d_g\end{pmatrix}$. Therefore, we see that $b_g$, $c_g \in \mathfrak{m}$, $a_g \equiv \chi_1(g) \pmod{\mathfrak{m}}$ and $d_g \equiv \chi_2(g) \pmod{\mathfrak{m}}$.  Thus, we get two characters $\tilde\chi_1$, $\tilde\chi_2 : G \to (R^{\pd}_{\rhob_0}/ \mathfrak{m}^2)^*$ such that 
$t^{\univ} \pmod{\mathfrak{m}^2} = \tr(\rho)  \pmod{\mathfrak{m}^2} = \tilde\chi_1+\tilde\chi_2$, $\tilde\chi_1(g) = a_g \pmod{\mathfrak{m}^2}$ and  $\tilde\chi_2(g) = d_g \pmod{\mathfrak{m}^2}$.

By Lemma~\ref{redlem}, we get that $\dim(\tan(R^{\pd}_{\rhob_0}/\mathfrak{m}^2)) \leq 2k$. But this contradicts the fact that $\dim(\tan(R^{\pd}_{\rhob_0})) > 2k$. So we conclude that $\rhob \not\simeq \rhob_0$.

Thus, $\rhob \simeq \rhob_x$ for some $x \in H^1(G,\chi^i)$ with $i \in \{1,-1\}$ and $x \neq 0$. So, by changing the basis if necessary, we can assume that $\rhob = \rhob_x$. This means that $\rho$ is a deformation of $\rhob_x$ and hence, there exists a continuous morphism $\phi_x : R^{\defo}_{\rhob_x} \to R^{\pd}_{\rhob_0}$. Moreover, $\phi_x$ is surjective as the elements $t^{\univ}(g) = \tr(\rho(g))$ with $g \in G$ are topological generators of $R^{\pd}_{\rhob_0}$ as a local complete $\FF$-algebra (\cite[Remark $3.5$]{C}). So, in particular, $\dim(\tan(R^{\defo}_{\rhob_x})) \geq \dim(\tan(R^{\pd}_{\rhob_0}))$.

From Lemma~\ref{shortlem}, we know that $\dim(\tan(R^{\defo}_{\rhob_x})) \leq 2k+m+n-1$. So, we get that $2k+m+n-1 \geq 2k+mn$ which implies that $0 \geq (m-1)(n-1)$. Therefore, we conclude that either $m=1$ or $n=1$.
\end{proof}

\begin{rem}
Proposition~\ref{repnprop} also follows from \cite[Theorem $4$]{B}.
\end{rem}

\begin{rem}
It is not clear how to prove Proposition~\ref{repnprop} when $H^2(G,1) \neq 0$ by employing the techniques used above or \cite[Theorem $4$]{B}. This is primarily because one can not determine the exact dimension of $\tan(R^{\pd}_{\rhob_0})$ using \cite[Theorem $2$]{B} when $H^2(G,1) \neq 0$.
\end{rem}

\subsection{Existence of the representation over $(\calR^{\MakeLowercase{{pd}}}_{\rhob_0})^{\red}$}
\label{redrepsubsec}
We will now explore whether the necessary condition for $T^{\univ}$ to be the trace of a representation defined over $\calR^{\pd}_{\rhob_0}$ obtained in Proposition~\ref{repnprop} is sufficient or not. We begin by proving that any deformation of $\tr(\rhob_0)$ to a domain comes from a representation when $\dim(H^1(G,\chi^i)) =1$ for some $i \in \{1,-1\}$.

Note that we don't need the hypothesis that $H^2(G,1) =0$ for the results proved in this subsection.

\begin{prop}
\label{primeprop}
Suppose there exists an $i \in \{1,-1\}$ such that $\dim(H^1(G,\chi^i)) =1$, $H^2(G,\chi^i)=0$. For such an $i$, fix a non-zero $x \in H^1(G,\chi^i)$. Let $P$ be a prime of $\calR^{\pd}_{\rhob_0}$. Then there exists a representation $\rho : G \to \GL_2(\calR^{\pd}_{\rhob_0}/P)$ such that $\rho$ is a deformation of $\rhob_x$ and $\tr(\rho) = T^{\univ} \pmod{P}$.
\end{prop}
\begin{proof}
Without loss of generality, assume $\dim(H^1(G,\chi)) =1$, $H^2(G,\chi)=0$. For the rest of the proof, denote $\calR^{\pd}_{\rhob_0}/P$ by $R$ and $T^{\univ} \pmod{P}$ by $t$. Let $K$ be the fraction field of $R$ and $m$ be the maximal ideal of $R$. 

Suppose $t$ is not reducible. Let $A=\begin{pmatrix} R & B \\ C & R\end{pmatrix}$ be the faithful GMA obtained for the pseudo-character $t : G \to R$ in Lemma~\ref{genlem} and $\rho$ be the corresponding representation. By Lemma~\ref{redgmalem}, we can take $A$ to be an $R$-subalgebra of $M_2(K)$. 

As $t$ is not reducible, we have $B$, $C \neq 0$. Hence, by Part~\eqref{partv} of Lemma~\ref{genlem}, $B$ is generated by $1$ element over $R$. As $B$ is a non-zero fractional ideal of the quotient field $K$ of $R$, it follows that the annihilator of $B$ is $0$. So $B$ is a free module of rank $1$ over $R$. Hence, by second part of Lemma~\ref{annihlem}, we get a representation $\rho' : G \to \GL_2(R)$ such that $\tr(\rho') = \tr(\rho)=t$ and $\rho' \pmod{m} = \rhob_{x_0}$ for some non-zero $x_0 \in H^1(G,\chi)$. As $\dim(H^1(G,\chi) ) =1$, for any non-zero $x \in H^1(G,\chi)$, $\rhob_x \simeq \rhob_{x_0}$. Hence, given a non-zero $x \in H^1(G,\chi)$, we can conjugate $\rho'$ by a suitable matrix to get a deformation of $\rhob_x$ with trace $t$.

Now suppose $t$ is reducible. So we have $t = \tilde\chi_1 + \tilde\chi_2$ where $\tilde\chi_i$ is a deformation of $\chi_i$ for $i=1,2$. Let $\tilde\chi = \tilde\chi_1\tilde\chi_2^{-1}$. For every $n >0$, denote $\tilde\chi \pmod{m^n} : G \to (R/m^n)^*$ by $\bar\chi_n$. This makes $R/m^n$ into a $G$-module for every $n > 0$. So $\bar\chi_1 =\chi$. For every $n > 0$, the natural map $R \to R/m^n$ is a map of $G$-modules and it induces a map $f_n : H^1(G,\tilde\chi) \to H^1(G,\bar\chi_n)$. These maps induce a map $f : H^1(G,\tilde\chi) \to \varprojlim_{n} H^1(G,\bar\chi_n)$. As $H^0(G,\bar\chi_n)= 0$ for all $n > 0$, we get, by \cite[Corollary 2.2]{T} and its proof, that the natural map $f$ is an isomorphism.

Now, for every $n>0$, the natural exact sequence $0 \to m^n/m^{n+1} \to R/m^{n+1} \to R/m^n \to 0$ is an exact sequence of discrete $G$-modules. As the modules are discrete, we get an exact sequence $H^1(G,R/m^{n+1}) \to H^1(G,R/m^n) \to H^2(G,m^n/m^{n+1})$ from the exact sequence of cohomology groups (see \cite[Section 2]{T} for more details). Note that $H^1(G,R/m^{n+1}) = H^1(G,\bar\chi_{n+1})$ and $H^1(G,R/m^{n}) = H^1(G,\bar\chi_{n})$. As $\chi_{n+1} \pmod{m/m^{n+1}} = \chi$, we see that $H^2(G,m^n/m^{n+1}) \simeq H^2(G,\chi)^{\oplus r}$ for some $r > 0$. Therefore, $H^2(G,m^n/m^{n+1}) = 0$ which means the map $H^1(G,R/m^{n+1}) \to H^1(G,R/m^n)$ is surjective for every $n > 0$. Therefore, the natural map $H^1(G,\tilde\chi) \to H^1(G,\chi)$ is surjective. 

Given a non-zero $x \in H^1(G,\chi)$, there exists a $\tilde x \in H^1(G,\tilde\chi)$ such that $f_1(\tilde x)=x$. Therefore, the representation $\rho : G \to \GL_2(R)$ given by $\rho(g) = \begin{pmatrix} \tilde\chi_1(g) & \tilde\chi_2(g)\tilde x(g)\\ 0 & \tilde\chi_2(g)\end{pmatrix}$ is a deformation of $\rhob_x$ with trace $t$.
\end{proof}

\begin{thm}
\label{reduceprop}
 Suppose there exists an $i \in \{1,-1\}$ such that $\dim(H^1(G,\chi^i)) =1$ and $H^2(G,\chi^i)=0$. Fix such an $i$ and let $x \in H^1(G,\chi^i)$ be a non-zero element. Then the map $\Psi_x : \calR^{\pd}_{\rhob_0} \to \calR^{\defo}_{\rhob_x}$ induces an isomorphism between $(\calR^{\pd}_{\rhob_0})^{\red}$ and $(\calR^{\defo}_{\rhob_x})^{\red}$.
\end{thm}
\begin{proof}
Without loss of generality, suppose $\dim(H^1(G,\chi))=1$ and $H^2(G,\chi)=0$. Let $x \in H^1(G,\chi)$ be a non-zero element and let $P$ be a prime ideal of $\calR^{\pd}_{\rhob_0}$. From Proposition~\ref{primeprop}, there is a representation $\rho : G \to \GL_2(\calR^{\pd}_{\rhob_0}/P)$ deforming $\rhob_x$ such that $\tr(\rho) = T^{\univ} \pmod{P}$. Hence, there exists a map $f : \calR^{\defo}_{\rhob_x} \to \calR^{\pd}_{\rhob_0}/P$ such that $\rho = f \circ \rho^{\univ}_x$. Hence, we have $f \circ \tr(\rho^{\univ}_x) = T^{\univ} \pmod{P}$. Recall that $\Psi_x \circ T^{\univ} = \tr(\rho^{\univ}_x)$. Hence, from the universal property of $\calR^{\pd}_{\rhob_0}$, it follows that the natural surjective map $\calR^{\pd}_{\rhob_0} \to \calR^{\pd}_{\rhob_0}/P$ is same as $f \circ \Psi_x$. Hence, $\ker(\Psi_x) \subset P$ for every prime $P$ of $\calR^{\pd}_{\rhob_0}$. This finishes the proof of the theorem.
\end{proof}

\subsection{Existence of the representation over $R^{\MakeLowercase{{pd}}}_{\rhob_0}$}
\label{repsubsec}
It is natural to ask if the non-reduced version of Theorem~\ref{reduceprop} is true or not.
In order to get an idea about the answer, we will now study if there exists a representation over $R^{\pd}_{\rhob_0}$ with trace $t^{\univ}$. 
We first prove a lemma about the structure of $R^{\pd}_{\rhob_0}$:
\begin{lem}
\label{strlem}
Suppose $H^2(G,1) = 0$, $\dim(H^1(G,1)):=k$ and $\dim(H^1(G,\chi^i)) =1$ for some $i \in \{1,-1\}$. For such an $i$, let $\dim(H^1(G,\chi^{-i})) := m$, $\dim(H^2(G,\chi^{-i})):=m'$ and $\dim(H^2(G,\chi^i)):=n'$. Then, $R^{\pd}_{\rhob_0} \simeq \FF[[X_1,\cdots,X_{m+2k}]]/I$ where $I$ is an ideal of $ \FF[[X_1,\cdots,X_{m+2k}]]$ generated by at most $m'+mn'$ elements.
\end{lem}
\begin{proof}
By \cite[Theorem 3.3.1]{WE}, we see that $R^{\pd}_{\rhob_0}$ is a quotient of a certain ring $R^1_D$ by an ideal $I$ generated by at most $k_0$ elements, where $$k_0= \sum_{j=1}^{2}\dim(\Ext^2_{G}(\chi_j,\chi_j)) + \dim(\Ext^2_{G}(\chi_1,\chi_2)).\dim(\Ext^1_{G}(\chi_2,\chi_1)) +  \dim(\Ext^2_{G}(\chi_2,\chi_1)).\dim(\Ext^1_{G}(\chi_1,\chi_2)).$$ 
Recall that $\Ext^2_{G}(\eta,\delta) \simeq H^2(G,\delta/\eta)$ for any characters $\eta$, $\delta : G \to \FF^\times$ and we have assumed $H^2(G,1)=0$. Therefore, we see that $k_0 = \sum_{j=1}^{2} 0 + (m').1 + m.n' = m'+mn'$. 

The ring $R^1_D$ is defined in \cite[Definition 3.2.3]{WE}. 
From the definition, we see that $R^1_D$ is a quotient of power series ring in $m_0$ variables over $\FF$, where $$m_0=\sum_{i=1}^{2} \dim(\Ext^1_G(\chi_i,\chi_i)) + \dim(\Ext^1_G(\chi_1,\chi_2))\dim(\Ext^1_G(\chi_2,\chi_1)).$$ 
By \cite[Fact 3.2.6]{WE}, Krull dimension of $R^1_D$ is $1 - 2 + \sum_{1 \leq i,j \leq 2} \dim(\Ext^1_G(\chi_i,\chi_j))$.
Since we are assuming that $\dim(H^1(G,\chi^i)) =1$ for some $i \in \{1,-1\}$ and $\dim(\Ext^1_{G}(\chi_1,\chi_1)) = \dim(\Ext^1_{G}(\chi_2,\chi_2)) = k$, we get that $m_0=2k+m$ and Krull dimension of $R^1_D$ is $2k+m$. Hence, we have $R^1_D \simeq \FF[[X_1,\cdots,X_{2k+m}]]$. 
This completes the proof of the lemma.
\end{proof}

We are now ready to prove an improvement of Theorem~\ref{reduceprop}.
\begin{thm}
\label{generalthm}
Suppose $H^2(G,1)=0$. Suppose there exists an $i \in \{1,-1\}$ such that $\dim(H^1(G,\chi^i)) =1$, $H^2(G,\chi^i)=0$, $\dim(H^1(G,\chi^{-i})) \in \{1,2,3\}$ and $\dim(H^2(G,\chi^{-i})) < \dim(H^1(G,\chi^{-i}))$. Then, there exists a representation $\rho : G \to \GL_2(R^{\pd}_{\rhob_0})$ such that $\tr(\rho) = t^{\univ}$ and for any non-zero $x \in H^1(G,\chi^i)$, $R^{\pd}_{\rhob_0} \simeq R^{\defo}_{\rhob_x}$.
\end{thm}
\begin{proof}
Without loss of generality, assume $\dim(H^1(G,\chi)) =1$. So we have $\dim(H^1(G,\chi^{-1})) \in \{1,2,3\}$. Let $A=\begin{pmatrix} R^{\pd}_{\rhob_0} & B \\ C & R^{\pd}_{\rhob_0}\end{pmatrix}$ be the GMA attached to the pseudo-character $t^{\univ} : G_{\QQ,Np} \to R^{\pd}_{\rhob_0}$ in Lemma~\ref{genlem} and let $\rho$ be the corresponding representation. Define $I_{\rhob_0}:= m'(B \otimes_{R^{\pd}_{\rhob_0}} C)$. So, by Lemma~\ref{annihlem}, if $y \in R^{\pd}_{\rhob_0}$ and $y.B=0$ then $y.I_{\rhob_0}=0$. 

Suppose $\dim(H^1(G,1)) := k$ and $\dim(H^1(G,\chi^{-1})) := m$. Then, by Lemma~\ref{strlem}, $R^{\pd}_{\rhob_0} \simeq \FF[[X_1,X_2,\cdots,X_{m+2k}]]/I$, where $I$ is an ideal of $\FF[[X_1,X_2,\cdots,X_{m+2k}]]$ generated by at most $\dim(H^2(G,\chi^{-1}))$ elements. Note that, by assumption, $\dim(H^2(G,\chi^{-1})) \leq m-1$. As $m \neq 0$, it follows from Lemma~\ref{tandimlem} and Lemma~\ref{redlem}, that $\dim(\tan(R^{\pd}/I_{\rhob_0})) < \dim(\tan(R^{\pd}_{\rhob_0}))$ and hence, $I_{\rhob_0} \neq (0)$. Therefore, $B$ and $C$ are non-zero.

Let $y \in R^{\pd}_{\rhob_0}$ be such that $y.I_{\rhob_0} = 0$ in $R^{\pd}_{\rhob_0}$. Let $\tilde y$ be a lift of $y$ in $\FF[[X_1,X_2,\cdots,X_{m+2k}]]$ and $\tilde I$ be the inverse image of $I_{\rhob_0}$ in $\FF[[X_1,X_2,\cdots,X_{m+2k}]]$. So we have $\tilde y.\tilde I \subset I$. Let us denote $\FF[[X_1,\cdots,X_{m+2k}]]$ by $R$ for the rest of the proof.

By Lemma~\ref{redlem}, we know that if $ P$ is a prime ideal of $R$ containing $\tilde I$, then its height is at least $m$. Suppose $\tilde y \not\in I$. Then, it follows that the ideal $\tilde I$ of $R$ consists of zero-divisors for $R/I$. Hence, it is contained in the union of primes associated to the ideal $I$. It follows, from the prime avoidance lemma (\cite[Lemma $3.3$]{E}), that $\tilde I$ is contained in some prime associated to $I$. Now, we will do a case by case analysis.

Suppose $I=(0)$. Since $\tilde I \neq (0)$, $\tilde y.\tilde I \subset I$ implies $\tilde y= 0$ and hence, $y=0$.

Suppose $I = (\alpha)$ for some non-zero $\alpha \in R$. This means $m$ is either $2$ or $3$ as minimal number of generators of $I$ is at most $m-1$. As $\alpha \neq 0$, it follows that $\alpha$ is a regular element in $R$. Note that $R$ is a regular local ring and hence, a Cohen-Macaulay ring (\cite[Corollary $18.17$]{E}). Therefore, every prime associated to $(\alpha)$ is minimal over it and hence, has height $1$ (\cite[Corollary $18.14$]{E}). As the height of any prime ideal of $R$ containing $\tilde I$ is at least $2$, it can not be contained in any prime associated to $(\alpha)$. Therefore, we get that $\tilde y \in (\alpha)$ which means $y=0$. 

Suppose $I=(\alpha,\beta)$ with $\alpha \nmid \beta$ and $\beta \nmid \alpha$. In this case $m=3$ as minimal number of generators of $I$ is at most $m-1$. Now, $R$ is regular local ring and hence, a UFD (see \cite[Theorem 19.19]{E}). Let $f$ be a gcd of $\alpha$ and $\beta$. Let $\alpha'$ and $\beta' \in R$ be such that $f.\alpha'=\alpha$, $f.\beta'=\beta$. Hence, $\alpha'$ and $\beta'$ are co-prime. By the argument given in the previous case, we get  that if $\tilde y. \tilde I \in I$, then $f | \tilde y$. Let $\tilde y' = \tilde y/f \in R$. So $\tilde y' \in R$ and $\tilde y'.\tilde I \subset (\alpha',\beta')$. 

Suppose $\tilde y' \not\in (\alpha',\beta')$. Then, by the argument given above, $\tilde I$ is contained in some prime associated to $(\alpha',\beta')$.

 As $\alpha'$ and $\beta'$ are co-prime, it follows that $\alpha'$, $\beta'$ is a regular sequence in $R$. Using \cite[Corollary $18.14$]{E} again, we see that every prime associated to $(\alpha',\beta')$ is minimal over it and hence, has height $2$. As the height of any prime ideal of $R$ containing $\tilde I$ is at least $3$, it can not be contained in any prime associated to $(\alpha',\beta')$. Hence, we get contradiction. So we get that $\tilde y' \in (\alpha',\beta')$ which means $\tilde y \in (\alpha,\beta)$ and $y=0$.

So, in both cases, we have $y=0$ which means the annihilator ideal of $B$ is $(0)$. 

As we are assuming $\dim(H^1(G,\chi)) =1$, it follows, from Part~\eqref{partv} of Lemma~\ref{genlem}, that $B$ is generated by at most one element over $R^{\pd}_{\rhob_0}$. On the other hand, we know $B$ is non-zero which means $B$ is generated by one element over $R^{\pd}_{\rhob_0}$. This, combined with the fact that annihilator of $B$ is $(0)$, implies that $B$ is a free $R^{\pd}_{\rhob_0}$-module of rank $1$. Now second part of Lemma~\ref{annihlem} gives a representation $\rho : G \to \GL_2(R^{\pd}_{\rhob_0})$ with $\tr(\rho)=t^{\univ}$.

Moreover, from the second part of Lemma~\ref{annihlem}, we see that $\rho'$ is a deformation of $\rhob_x$ for some non-zero $x \in H^1(G,\chi)$. Therefore, it induces a map $\psi'_x : R^{\defo}_{\rhob_x} \to R^{\pd}_{\rhob_0}$. So we get a map $\psi'_x \circ \psi_x : R^{\pd}_{\rhob_0} \to R^{\pd}_{\rhob_0}$. Now for all $g \in G$, $\psi'_x \circ \psi_x(t^{\univ}(g)) = \psi'_x(\tr(\rho^{univ}_x(g)))=\tr(\rho'(g))=t^{\univ}(g)$. Therefore, the universal property of $R^{\pd}_{\rhob_0}$ implies that $\psi'_x \circ \psi_x$ is just the identity map. Hence, $\psi_x$ is injective which means $\psi_x$ is an isomorphism. This proves the theorem.
\end{proof}

\begin{rem}
More generally, if we remove the assumption $\dim(H^1(G,\chi^{-i})) \in \{1,2,3\}$, the proof of Theorem~\ref{generalthm} still works if we know that $R^{\pd}_{\rhob_0}$ is isomorphic to a quotient of $\FF[[X_1,\cdots,X_{2k+m}]]$ by an ideal $I$ such that  the height of any prime associated to $I$ is at most $m-1$. In particular, the proof works if $I$ is generated by at most $2$ elements. Note that if $m \geq 6$ and $I$ is generated by at most $2$ elements, then the Krull dimension of $R^{\pd}_{\rhob_0}$ is $\geq 4$. In \cite[Section $4$]{Bo3}, there are examples of $R^{\defo}_{\rhob_x}$ having arbitrary large Krull dimension. So the possibility that $I$ is generated by $2$ elements cannot be ruled out even when $m \geq 6$.
\end{rem}

\begin{rem}
Without the assumption $\dim(H^1(G,\chi^{-i})) \in \{1,2,3\}$, we know that $R^{\pd}_{\rhob_0} \simeq \FF[[X_1,\cdots,X_{m+2k}]]/I$, where $I$ is an ideal generated by at most $m-1$ elements. If $I$ is generated by at least $3$ elements and we do not know that the height of any prime associated to $I$ is at most $m-1$, then we can not use the method of the proof of Theorem~\ref{generalthm}. To be precise, the analysis of the annihilator of $B$ breaks down. The main reason of this breakdown is the following: if the minimal number of generators of an ideal $I$ of the ring $\FF[[X_1,\cdots,X_{m+2k}]]$ is at least $3$ and at most $m-1$, then for $y \in \FF[[X_1,\cdots,X_m]]$, $yP \subset I$ for a prime ideal of height $m$ does not necessarily imply that $y \in I$. For example, consider the ideal $I=(xu^2,yv^2,x^2u-y^2v)$ in $\FF[[x,y,u,v,z,w]]$ with $m=4$ and $k=1$. Now, $xyuv \not\in I$ but $\{xyuv.x,xyuv.y,xyuv.u,xyuv.v\} \subset I$. However, if we can prove that the annihilator of $B$ is $(0)$, then the proof of Theorem~\ref{generalthm} would imply the existence of such a representation.
\end{rem}

\subsection{Existence of the representation over $\calR^{\MakeLowercase{{pd}}}_{\rhob_0}$}
In this subsection, we will turn our attention to the characteristic $0$ deformation ring $\calR^{\pd}_{\rhob_0}$ to see if we can extend Theorem~\ref{generalthm} in characteristic $0$ to prove existence of the representation over $\calR^{\pd}_{\rhob_0}$ with trace $T^{\univ}$.

\begin{prop}
\label{nonpseudoprop}
Suppose $H^2(G,1)=0$. Suppose there exists an $i \in \{1,-1\}$ such that $\dim(H^1(G,\chi^i)) =1$, $H^2(G,\chi^i)=0$ and $\dim(H^2(G,\chi^{-i})) < \dim(H^1(G,\chi^{-i}))$. For such an $i$, let $x \in H^1(G,\chi^i)$ be a non-zero element. Suppose $p$ is not a zero-divisor in $\calR^{\pd}_{\rhob_0}$.  For such an $i$, if $\dim(H^1(G,\chi^{-i})) \in \{1,2,3\}$, then there exists a representation $\tau : G_{\QQ,Np} \to \GL_2(\calR^{\pd}_{\rhob_0})$ such that $\tr(\tau) = T^{\univ}$. As a consequence, the map $\Psi_x : \calR^{\pd}_{\rhob_0} \to \calR^{\defo}_{\rhob_x}$ is an isomorphism. 
\end{prop}
\begin{proof}
 Without loss of generality, assume $\dim(H^1(G,\chi)) =1$. 
Let $A=\begin{pmatrix} \calR^{\pd}_{\rhob_0} & \mathcal{B} \\ \mathcal{C} & \calR^{\pd}_{\rhob_0}\end{pmatrix}$ be the GMA attached to the pseudo-character $T^{\univ} : G \to \calR^{\pd}_{\rhob_0}$ in Lemma~\ref{genlem}. From Lemma~\ref{annihlem} and the proof of Theorem~\ref{generalthm}, we see that it is sufficient to prove that the annihilator of $\mathcal{B}$ is $(0)$.

Suppose $m'(\mathcal{B} \otimes_{\calR^{\pd}_{\rhob_0}} \mathcal{C}) = \mathcal{I}_{\rhob_0}$. Suppose $y \in \calR^{\pd}_{\rhob_0}$, $y\mathcal{B}=0$ and $y \neq 0$.  So, by Lemma~\ref{annihlem}, we get $y\calI_{\rhob_0}=0$. Let $I$ be the image of the ideal $(p,\calI_{\rhob_0})$ in $\calR^{\pd}_{\rhob_0}/(p)$ and $\bar y$ be the image of $y$ in $\calR^{\pd}_{\rhob_0}/(p)$. Hence, we get $\bar yI=0$. By Lemma~\ref{redlem} and Part~\eqref{partvi} of Lemma~\ref{genlem}, it follows that if $P$ is a prime of $R^{\pd}_{\rhob_0}$ minimal over $I$, then $\dim(R^{\pd}_{\rhob_0}/P) \leq 2k$, where $\dim(H^1(G,1)) := k$. Now from the proof of Theorem~\ref{generalthm} it follows that $\bar y=0$. 

Hence, we see that $y \in (p)$. As $y \neq 0$, there exists a positive integer $k_0$ such that $y=p^{k_0}y'$ with $y' \not\in (p)$. Since $p$ is not a zero divisor in $\calR^{\pd}_{\rhob_0}$, it follows that $y'\calI_{\rhob_0} = 0$. As $y' \neq 0$, the argument given in the previous paragraph implies $y' \in (p)$ and hence, gives us a contradiction. Therefore, we get $y=0$. This means that $\calI_{\rhob_0} \neq (0)$.

This, along with the fact $\dim(H^1(G,\chi)) =1$, implies that $\mathcal{B}$ is free $\calR^{\pd}_{\rhob_0}$-module of rank $1$. Following the proof of Theorem~\ref{generalthm} from here, we get the representation with trace $T^{\univ}$ and see that $\Psi_x$ is an isomorphism for all non-zero $x \in H^1(G,\chi)$.
\end{proof}

 Finally, we now give a result which will be used in the next section.

\begin{prop}
\label{nondefprop}
Suppose $H^2(G,1)=0$. Suppose there exists an $i \in \{1,-1\}$ such that $\dim(H^1(G,\chi^i)) =1$, $H^2(G,\chi^i)=0$, $\dim(H^1(G,\chi^{-i})) \in \{1,2,3\}$ and $\dim(H^2(G,\chi^{-i})) < \dim(H^1(G,\chi^{-i}))$. Let $x \in H^1(G,\chi^i)$ be a non-zero element. If $p$ is not a zero-divisor in $\calR^{\defo}_{\rhob_x}$, then the map $\Psi_x : \calR^{\pd}_{\rhob_0} \to \calR^{\defo}_{\rhob_x}$ is an isomorphism.
\end{prop}
\begin{proof}
We have the following commutative diagram:
\[ \begin{tikzcd}
\calR^{\pd}_{\rhob_0} \arrow{r}{\Psi_x} \arrow[swap]{d}{f_1} & \calR^{\defo}_{\rhob_x} \arrow{d}{f_2} \\
R^{\pd}_{\rhob_0} \arrow{r}{\psi_x}& R^{\defo}_{\rhob_x}
\end{tikzcd}
\]
Here the vertical maps $f_1$ and $f_2$ are the morphisms induced by $t^{\univ}$ and $\rho^{\univ}_x$, respectively. Now, $\ker(f_1)$ is the ideal generated by $p$ in $\calR^{\pd}_{\rhob_0}$, while $\ker(f_2)$ is the ideal generated by $p$ in $\calR^{\defo}_{\rhob_x}$. By Theorem~\ref{genthm}, $\psi_x$ is an isomorphism. So $\ker(\psi_x \circ f_1)=\ker(f_1)=(p)$. As $\psi_x \circ f_1 = f_2 \circ \Psi_x$, it follows that $\ker(f_2 \circ \Psi_x) = (p)$. Thus $\ker(\Psi_x) \subset (p)$.

Let $h \in \ker(\Psi_x)$. So $h \in (p)$. Suppose $h \neq 0$. As $\calR^{\pd}_{\rhob_0}$ is a complete local ring, $\cap_{n \geq 1}(p^n)=(0)$. Therefore, we have $h = p^{n_0}h'$ where $n_0 \geq 1$ is an integer, $h' \in \calR^{\pd}_{\rhob_0}$ and $h' \not\in (p)$. Thus, $h' \not\in \ker(\Psi_x)$ and hence, $\Psi_x(h') \neq 0$. But $\Psi_x(h) = 0$. So we get $\Psi_x(h) = \Psi_x(p^{n_0}.h')=p^{n_0}.\Psi_x(h')=0$. Thus, we get that $p$ is a zero-divisor in $\calR^{\defo}_{\rhob_x}$ which contradicts our assumption. Therefore, it follows that $\ker(\Psi_x)=(0)$. From Lemma~\ref{surjlem}, we know that $\Psi_x$ is surjective. Hence, it follows that $\Psi_x$ is an isomorphism.
\end{proof}

\subsection{Consequences for Galois groups}
\label{galrepsubsec}
In this subsection, we list the consequences of results proved in this section so far for $G_{\QQ,Np}$. To be precise, let $N$ be an integer not divisible by $p$ and $\rhob_0 : G_{\QQ,Np} \to \GL_2(\FF)$ be an odd, semi-simple, reducible representation. So there exist characters $\chi_1, \chi_2 : G_{\QQ,Np} \to \FF^\times$ such that $\rhob_0 = \chi_1 \oplus \chi_2$ and $\chi_1 \neq \chi_2$. Let $\chi = \chi_1\chi_2^{-1}$. We will now see the consequences of the main results of previous subsections in the present setup.

\begin{thm}
\label{redprop}
Suppose $\dim(H^1(G_{\QQ,Np},\chi^i)) =1$ for some $i \in \{1,-1\}$. Fix such an $i$ and let $x \in H^1(G_{\QQ,Np},\chi^i)$ be a non-zero element. Then the map $\Psi_x : \calR^{\pd}_{\rhob_0} \to \calR^{\defo}_{\rhob_x}$ induces an isomorphism between $(\calR^{\pd}_{\rhob_0})^{\red}$ and $(\calR^{\defo}_{\rhob_x})^{\red}$.
\end{thm}
\begin{proof}
This follows from Lemma~\ref{cohomlem} and Theorem~\ref{reduceprop}.
\end{proof}

\begin{thm}
\label{genthm}
Suppose $p \nmid \phi(N)$ and $\dim(H^1(G_{\QQ,Np},\chi^i)) =1$ for some $i \in \{1,-1\}$. Moreover, for such an $i$, assume that $\dim(H^1(G_{\QQ,Np},\chi^{-i})) \in \{1,2,3\}$. Then, there exists a representation $\rho : G_{\QQ,Np} \to \GL_2(R^{\pd}_{\rhob_0})$ such that $\tr(\rho) = t^{\univ}$ and for any non-zero $x \in H^1(G_{\QQ,Np},\chi^i)$, $R^{\pd}_{\rhob_0} \simeq R^{\defo}_{\rhob_x}$.
\end{thm}
\begin{proof}
The theorem follows from Lemma~\ref{cohomlem} and Theorem~\ref{generalthm}.
\end{proof}

\begin{prop}
\label{nonzeroprop}
Suppose $p \nmid \phi(N)$ and $\dim(H^1(G_{\QQ,Np},\chi^i)) =1$ for some $i \in \{1,-1\}$. For such an $i$, assume that $\dim(H^1(G_{\QQ,Np},\chi^{-i})) \in \{1,2,3\}$. Let $x \in H^1(G_{\QQ,Np},\chi^i)$ be a non-zero element. If $p$ is not a zero-divisor in either $\calR^{\pd}_{\rhob_0}$ or $\calR^{\defo}_{\rhob_x}$, then there exists a representation $\tau : G_{\QQ,Np} \to \GL_2(\calR^{\pd}_{\rhob_0})$ such that $\tr(\tau) = T^{\univ}$ and the map $\Psi_x : \calR^{\pd}_{\rhob_0} \to \calR^{\defo}_{\rhob_x}$ is an isomorphism.
\end{prop}
\begin{proof}
Follows from Lemma~\ref{cohomlem}, Proposition~\ref{nonpseudoprop} and Proposition~\ref{nondefprop}.
\end{proof}

\section{Increasing the ramification}
\label{ramsec}
From now on, we will focus on the case where $G= G_{\QQ,Np}$ and $\rhob_0$ is a reducible, odd, semi-simple representation of $G_{\QQ,Np}$.
Let $\ell$ be a prime such that $\ell \nmid Np$. As $G_{\QQ,Np}$ is a quotient of $G_{\QQ,N\ell p}$, the representations $\rhob_x$ with $x \in H^1(G_{\QQ,Np},\chi^i)$ with $i \in \{1,-1\}$ are also representations of $G_{\QQ,N\ell p}$ and $(\tr(\rhob_0),\det(\rhob_0))$ is also a pseudo-representation of $G_{\QQ,N\ell p}$. Let $\calR^{\pd,\ell}_{\rhob_0}$ and $R^{\pd,\ell}_{\rhob_0}$ be the universal deformation rings of $(\tr(\rhob_0),\det(\rhob_0))$ considered as a pseudo-representation of $G_{\QQ,N\ell p}$ in the categories $\mathcal{C}$ and $\mathcal{C}_0$ respectively. For a non-zero $x \in H^1(G_{\QQ,Np},\chi^i)$ with $i \in \{1,-1\}$,  Let $\calR^{\defo,\ell}_{\rhob_x}$ and $R^{\defo,\ell}_{\rhob_x}$ be the universal deformation rings of $\rhob_x$ considered as a representation of $G_{\QQ,N\ell p}$ in the categories $\mathcal{C}$ and $\mathcal{C}_0$, respectively.

We keep the notation from previous sections for $G_{\QQ,Np}$. In this section, we will study the relationship between $\calR^{\pd,\ell}_{\rhob_0}$ (resp. $R^{\pd,\ell}_{\rhob_0}$) and $\calR^{\pd}_{\rhob_0}$ (resp. $R^{\pd}_{\rhob_0}$) using the results obtained in the previous section and results from \cite{Bo}.

Before proceeding further, let us establish some more notation. Let $t^{\univ,\ell}$ be the universal pseudo-character from $G_{\QQ,N\ell p}$ to $R^{\pd,\ell}_{\rhob_0}$ deforming $\tr(\rhob_0)$ and $T^{\univ,\ell}$ be the universal pseudo-character from $G_{\QQ,N\ell p}$ to $\calR^{\pd,\ell}_{\rhob_0}$ deforming $\tr(\rhob_0)$. Denote the pseudo-character obtained by composing $t^{\univ,\ell}$ with the surjective map $R^{\pd,\ell}_{\rhob_0} \to (R^{\pd,\ell}_{\rhob_0})^{\red}$ by $(t^{\univ,\ell})^{\red}$.

\subsection{Comparison between $\calR^{\pd,\ell}_{\rhob_0}$ and $\calR^{\pd}_{\rhob_0}$}
\label{compsubsec}
We are now ready to compare $\calR^{\pd,\ell}_{\rhob_0}$ and $\calR^{\pd}_{\rhob_0}$. We begin with an easy case first.
\begin{lem}
If $p \nmid \ell-1$ and $\chi|_{G_{\QQ_{\ell}}} \neq \omega_p|_{G_{\QQ_{\ell}}},\omega^{-1}_p|_{G_{\QQ_{\ell}}},1$, then $\calR^{\pd,\ell}_{\rhob_0} \simeq \calR^{\pd}_{\rhob_0}$.
\end{lem}
\begin{proof}
From Lemma~\ref{gmalem}, there exists a faithful GMA $A^{\univ}$ over $\calR^{\pd,\ell}_{\rhob_0}$ and a representation $\rho : G_{\QQ,N\ell p} \to A^{\univ}$ such that $\tr(\rho)=T^{\univ,\ell}$, $\calR^{\pd,\ell}_{\rhob_0}[\rho(G_{\QQ,N\ell p})]=A^{\univ}$ and $\calR^{\pd,\ell}_{\rhob_0}[\rho(G_{\QQ_{\ell}})]$ is a sub $\calR^{\pd,\ell}_{\rhob_0}$-GMA of $A^{\univ}$. So $\calR^{\pd,\ell}_{\rhob_0}[\rho(G_{\QQ_{\ell}})] = \begin{pmatrix} \calR^{\pd,\ell}_{\rhob_0}  & B_{\ell}\\ C_{\ell} & \calR^{\pd,\ell}_{\rhob_0}\end{pmatrix}$, where $B_{\ell}$ and $C_{\ell}$ are $\calR^{\pd,\ell}_{\rhob_0}$-submodules of $B$ and $C$, respectively and hence, both of them are finitely generated $\calR^{\pd,\ell}_{\rhob_0}$-modules.

 As $\chi|_{G_{\QQ_{\ell}}} \neq \omega_p|_{G_{\QQ_{\ell}}},\omega^{-1}_p|_{G_{\QQ_{\ell}}}$,$1$, by local Euler characteristic formula, we get that $H^1(G_{\QQ_{\ell}},\chi|_{G_{\QQ_{\ell}}})=H^1(G_{\QQ_{\ell}},\chi^{-1}|_{G_{\QQ_{\ell}}})=0$. Therefore, we get, by Part~\eqref{partv} of Lemma~\ref{genlem}, that $B_{\ell} = C_{\ell}=0$.

Thus, we get characters $\tilde\chi_1, \tilde\chi_2 : G_{\QQ_{\ell}} \to (\calR^{\pd,\ell}_{\rhob_0})^*$ such that $\rho
(g) = \begin{pmatrix} \tilde\chi_1(g) & 0\\ 0 &\tilde\chi_2(g)\end{pmatrix}$ for all $g \in G_{\QQ_{\ell}}$. 
As $p \nmid \ell-1$, we get, by local class field theory, $\tilde\chi_1(I_{\ell})=\tilde\chi_2(I_{\ell})=1$. 
So the pseudo-character $t^{\univ,\ell}$ factors through $G_{\QQ,Np}$. Hence, this induces a map $f : \calR^{\pd}_{\rhob_0} \to \calR^{\pd,\ell}_{\rhob_0}$. Viewing $T^{\univ}$ as a pseudo-character of $G_{\QQ,N\ell p}$ gives us a map $f' : \calR^{\pd,\ell}_{\rhob_0} \to \calR^{\pd}_{\rhob_0}$. 

Now, for $g \in G_{\QQ,Np}$, $f(T^{\univ}(g)) = T^{\univ,\ell}(g')$ for any lift $g'$ of $g$ in $G_{\QQ,N\ell p}$. Thus, $f' \circ f(T^{\univ}(g)) = f'(T^{\univ,\ell}(g')) = T^{\univ}(g)$ for all $g \in G_{\QQ,Np}$. Therefore, $f' \circ f$ is the identity map. On the other hand, for $g \in G_{\QQ,N\ell p}$, $f'(T^{\univ,\ell}(g)) = T^{\univ}(g'')$, where $g''$ is the image of $g$ in $G_{\QQ,Np}$. So $f \circ f'(T^{\univ,\ell}(g)) = f(T^{\univ}(g'')) = T^{\univ,\ell}(g)$ for every $g \in G_{\QQ,N\ell p}$. Therefore, we get that $f \circ f'$ is identity. Hence, $f$ is an isomorphism. Thus, we get $ \calR^{\pd}_{\rhob_0} \simeq \calR^{\pd,\ell}_{\rhob_0}$.
\end{proof}

We now prove Theorem~\ref{thmb}.
\begin{proof}[Proof of Theorem~\ref{thmb}]
As $p \nmid \ell^2-1$ and $\chi^{-i}|_{G_{\QQ_\ell}} = \omega_p|_{G_{\QQ_\ell}}$, we see, from Lemma~\ref{cohomlem}, that
$\dim(H^1(G_{\QQ,N\ell p},\chi^i)) = 1$ and $\dim(H^1(G_{\QQ,N\ell p},\chi^{-i})) \leq m + 1$. Therefore, by Theorem~\ref{redprop}, we have for any non-zero $x \in H^1(G_{\QQ,Np},\chi^i)$, $(\calR^{\pd}_{\rhob_0})^{\red} \simeq (\calR^{\defo}_{\rhob_x})^{\red}$ and $(\calR^{\pd,\ell}_{\rhob_0})^{\red} \simeq (\calR^{\defo,\ell}_{\rhob_x})^{\red}$. The first part now follows from \cite[Theorem $4.7$]{Bo}.

If $m \leq 2$, then $\dim(H^1(G_{\QQ,N\ell p},\chi^{-i})) \leq 3$. Hence, in this case, by Theorem~\ref{genthm}, we have $R^{\pd}_{\rhob_0} \simeq R^{\defo}_{\rhob_x}$ and $R^{\pd,\ell}_{\rhob_0} \simeq R^{\defo,\ell}_{\rhob_x}$ for any non-zero $x \in H^1(G_{\QQ,Np},\chi^i)$. The second part now follows from \cite[Theorem $4.7$]{Bo}.
\end{proof}

Note that Theorem~\ref{thmb} does not give a precise description of the relations $r_i$'s even if we know how $\bar r_i$'s look like. So it is natural to ask if one can get results about the structure of $\calR^{\pd,\ell}_{\rhob_0}$ which are more precise than the ones obtained in Theorem~\ref{thmb}. We will focus on this question for the rest of the article. However, we will restrict ourself to the simplest case where $\rhob_0$ is unobstructed which will be introduced in the next subsection.
 
\subsection{Unobstructed pseudo-characters}
\label{unobssubsec}
We now introduce the notion of unobstructed pseudo-representations. In this case, we know the precise structure of $\calR^{\pd}_{\rhob_0}$ and our primary goal is to determine the structure of $\calR^{\pd,\ell}_{\rhob_0}$ as accurately as possible in this special scenario. Here we gather some results which will be used later on.
\begin{defi}
We say that the pseudo-character  associated to $\rhob_0$ (or by abuse of notation $\rhob_0$) is unobstructed if $\dim(H^1(G_{\QQ,Np},\chi)) = \dim(H^1(G_{\QQ,Np},\chi^{-1})) =1$.
\end{defi}
 Note that Vandiver's conjecture implies that $\rhob_0$ is unobstructed if $N=1$ (see \cite[Theorem 22]{BK}). Moreover, \cite[Theorem 22]{BK} also provides some examples of unobstructed $\rhob_0$ when $N=1$.
On the other hand, \cite[Lemma $2.3$]{D2} gives necessary and sufficient conditions for $\rhob_0$ to be unobstructed.
 In this case, by Lemma~\ref{shortlem}, Lemma~\ref{cohomlem} and Lemma~\ref{cohomlem2}, we know that $\dim(H^1(G_{\QQ,Np},\text{ad}(\rhob_x)))=3$ for any non-zero $x \in H^1(G_{\QQ,Np},\chi^i)$ with $i \in \{1,-1\}$. So we get the following result:
\begin{lem}
\label{unobslem}
Suppose $p \nmid \phi(N)$ and $\rhob_0$ is unobstructed. Then, for a non-zero $x \in H^1(G_{\QQ,Np},\chi^i)$ with $i \in \{1,-1\}$, the map $\Psi_x : \calR^{\pd}_{\rhob_0} \to \calR^{\defo}_{\rhob_x}$ is an isomorphism and both are isomorphic to $W(\FF)[[X,Y,Z]]$. 
\end{lem}
\begin{proof}
Since $\rhob_0$ is odd and $p \nmid \phi(N)$, we get, by the global Euler characteristic formula, that $H^2(G_{\QQ,Np},1)=H^2(G_{\QQ,Np},\chi)=H^2(G_{\QQ,Np},\chi^{-1})=H^2(G_{\QQ,Np},\text{ad}(\rhob_x))=0$. Therefore, we get, from \cite[Theorem $2.4$]{Bo2}, that $\calR^{\defo}_{\rhob_x} \simeq W(\FF)[[X,Y,Z]]$. The result now follows from Proposition~\ref{nonzeroprop}.
\end{proof}

\begin{lem}
\label{redobslem}
Suppose $\rhob_0$ is unobstructed. Then, there exists a $z \in \calR^{\pd}_{\rhob_0}$ such that $T^{\univ} \pmod{(z)}$ is reducible.
\end{lem}
\begin{proof}
Let $A = \begin{pmatrix} \calR^{\pd}_{\rhob_0} & B\\ C & \calR^{\pd}_{\rhob_0}\end{pmatrix}$ be the GMA attached to the pseudo-character $T^{\univ} : G \to \calR^{\pd}_{\rhob_0}$ in Lemma~\ref{genlem}. Since $\rhob_0$ is unobstructed, Part~\eqref{partv} of Lemma~\ref{genlem} implies that both $B$ and $C$ are generated over $\calR^{\pd}_{\rhob_0}$ by at most $1$ element. The lemma now follows from Part~\eqref{partvi} of Lemma~\ref{genlem}.
\end{proof}

Recall that we already know that the deformation ring does not change after allowing ramification at a prime $\ell$ such that $\chi|_{G_{\QQ_\ell}} \neq \omega_p, \omega_p^{-1}, 1$. So we are not going to consider them anymore in the rest of the article.

\subsection{Generators of the co-tangent space of $\calR^{\defo,\ell}_{\rhob_x}$}
Now suppose $\rhob_0$ is unobstructed, $p \nmid \phi(N)$ and $\ell$ is a prime such that $\ell \nmid Np$, $p \nmid \ell-1$ and $\chi^i|_{G_{\QQ_\ell}} = \omega_p$ for some $i \in \{1,-1\}$. For such an $i$, let $x \in H^1(G_{\QQ,Np},\chi^{-i})$ be a non-zero element. Throughout this subsection, we are going to fix this set-up without mentioning it again. We will now give a set of generators for the co-tangent space of $\calR^{\defo,\ell}_{\rhob_x}$.

We first fix some more notation. Fix a lift $g_\ell$ of $\text{Frob}_\ell$ in $G_{\QQ_\ell}$ and fix a topological generator $i_\ell$ of the unique $\ZZ_p$-quotient of the tame inertia group at $\ell$. Let $\rho^{\univ,\ell}_x : G_{\QQ,N\ell p} \to \GL_2(\calR^{\defo,\ell}_{\rhob_x})$ be a universal deformation of $\rhob_x$ for $G_{\QQ,N\ell p}$ and let $\rho^{\univ}_x : G_{\QQ,N\ell p} \to \GL_2(\calR^{\defo}_{\rhob_x})$ be a universal deformation of $\rhob_x$ for $G_{\QQ,N p}$.

We now combine \cite[Lemma $4.8$]{Bo} and \cite[Lemma $4.9$]{Bo} to get the following:
\begin{lem}
\label{ramlem}
Suppose we are in the set-up fixed above. 
Then $\rho^{\univ,\ell}_x|_{I_\ell}$ factors through the unique $\ZZ_p$-quotient of the tame inertia group at $\ell$.
Moreover, after conjugation if necessary, we get $\rho^{\univ,\ell}_x(g_{\ell}) = \begin{pmatrix} \widehat{\chi_1(g_\ell)}(1+y) & 0\\ 0 & \widehat{\chi_2(g_\ell)}(1+y')\end{pmatrix}$  for some $y,y' \in \calR^{\defo,\ell}_{\rhob_x}$ and
\begin{enumerate}
\item If $i=1$ and $p \nmid \ell+1$, then $\rho^{\univ,\ell}_x (i_\ell) = \begin{pmatrix} 1 & w\\ 0 & 1 \end{pmatrix}$ for some $w \in \calR^{\defo,\ell}_{\rhob_x}$,
\item If $i=-1$ and $p \nmid \ell+1$, then $\rho^{\univ,\ell}_x (i_\ell) = \begin{pmatrix} 1 & 0\\ w & 1 \end{pmatrix}$ for some $w \in \calR^{\defo,\ell}_{\rhob_x}$,
\item If $p | \ell+1$, then $\rho^{\univ,\ell}_x (i_\ell) = \begin{pmatrix} \sqrt{1+uv} & u\\ v & \sqrt{1+uv} \end{pmatrix}$ for some $u,v \in \calR^{\defo,\ell}_{\rhob_x}$.
\end{enumerate}
\end{lem}

Viewing $\rho^{\univ}_x$ as a representation of $G_{\QQ,N\ell p}$, we get a map $f : \calR^{\defo,\ell}_{\rhob_x} \to \calR^{\defo}_{\rhob_x}$.
\begin{lem}
\label{surjectlem}
 The morphism $f : \calR^{\defo,\ell}_{\rhob_x} \to \calR^{\defo}_{\rhob_x}$ is surjective and $\ker(f)$ is generated by the entries of the matrix $\rho^{\univ,\ell}_x(i_{\ell})-Id$.
\end{lem}
\begin{proof}
Let $J$ be the ideal of $\calR^{\defo,\ell}_{\rhob_x}$ generated by the entries of the matrix $\rho^{\univ,\ell}_x(i_{\ell})-Id$ and $\phi : \calR^{\defo,\ell}_{\rhob_x} \to \calR^{\defo,\ell}_{\rhob_x}/J$ be the natural surjective map. As $\rho^{\univ}_x(i_{\ell})=Id$, we get that $J \subset \ker(f)$ which gives us a map $f' : \calR^{\defo,\ell}_{\rhob_x}/J \to \calR^{\defo}_{\rhob_x}$ such that $f' \circ \phi =f$. 
On the other hand, $\rho^{\univ,\ell}_x \pmod{J}$ is unramified at $\ell$ and hence, is a representation of $G_{\QQ,Np}$. Thus it induces a map $g : \calR^{\defo}_{\rhob_x} \to \calR^{\defo,\ell}_{\rhob_x}/J$ such that $g \circ \rho^{\univ}_x = \rho^{\univ,\ell}_x \pmod{J}$. Now $f' \circ g \circ \rho^{\univ}_x = \rho^{\univ}_x$ as representations of $G_{\QQ,Np}$ and $g \circ f' \circ \rho^{\univ,\ell}_x \pmod{J} = \rho^{\univ,\ell}_x \pmod{J}$ as representations of $G_{\QQ,N\ell p}$. Hence, we see that both $f' \circ g$ and $g \circ f'$ are identity maps. Hence, $f'$ is an isomorphism which proves the lemma.
\end{proof}

We are now ready to state the main result of this subsection.
\begin{lem}
\label{tangenlem}
 Suppose $\rhob_0$ is unobstructed and $p \nmid \phi(N)$. Let $\ell$ be a prime such that $\ell \nmid Np$, $p \nmid \ell-1$ and $\chi^i|_{G_{\QQ_\ell}} = \omega_p$ for some $i \in \{1,-1\}$. For such an $i$, let $x \in H^1(G_{\QQ,Np},\chi^{-i})$ be a non-zero element. Moreover, assume $\ell/\tilde\ell$ is a topological generator of $1+p\ZZ_p$.
Suppose $\rho^{\univ,\ell}_x(g_{\ell}) = \begin{pmatrix} \widehat{\chi_1(g_\ell)}(1+y) & 0\\ 0 & \widehat{\chi_2(g_\ell)}(1+y')\end{pmatrix}$. Then there exists an element $z \in \calR^{\defo,\ell}_{\rhob_x}$ such that the ideal generated by $p$, $y$, $y'$, $z$ and $\ker(f)$ is the maximal ideal of $\calR^{\defo,\ell}_{\rhob_x}$.
\end{lem}
\begin{proof}
Let $z_0 \in \calR^{\pd}_{\rhob_0}$ be an element such that $T^{\univ} \pmod{(z_0)}$ is reducible. Such an element exists by Lemma~\ref{redobslem}. By Lemma~\ref{unobslem}, the map $\Psi_x : \calR^{\pd}_{\rhob_0} \to \calR^{\defo}_{\rhob_x}$ is an isomorphism. Hence, $\tr(\rho^{\univ}_x) \pmod{(\Psi_x(z_0))}$ is reducible. 

Viewing $\rho^{\univ}_x$ as a representation of $G_{\QQ,N\ell p}$, we get $f \circ \rho^{\univ,\ell}_x = \rho^{\univ}_x$. So we have $\rho^{\univ}_x(g_\ell) = \begin{pmatrix} \widehat{\chi_1(g_\ell)}(1+f(y)) & 0\\ 0 & \widehat{\chi_2(g_\ell)}(1+f(y'))\end{pmatrix}$. 

Following the proof of the last part of Lemma~\ref{gmalem}, we get that the set $\{p,f(y),f(y'),\Psi_x(z_0)\}$ generate the maximal ideal of $\calR^{\defo}_{\rhob_x}$. By Lemma~\ref{surjectlem}, $f$ is surjective. Hence, if $z \in \calR^{\defo,\ell}_{\rhob_x}$ is an element such that $f(z)= \Psi_x(z_0)$, then the ideal generated by $p$, $y$, $y'$, $z$ and $\ker(f)$ is the maximal ideal of $\calR^{\defo,\ell}_{\rhob_x}$.
\end{proof}

\subsection{Structure of $\calR^{\pd,\ell}_{\rhob_0}$ with unobstructed $\rhob_0$ and $p \nmid \ell^2-1$}
As we saw in Lemma~\ref{unobslem}, $\calR^{\pd}_{\rhob_0} \simeq W(\FF)[[X,Y,Z]]$ when $\rhob_0$ is unobstructed and $p \nmid \phi(N)$. In this sub-section, we are going to analyze how its structure changes after allowing ramification at a prime $\ell$ such that $\ell \nmid Np$ and $p \nmid \ell^2-1$.

For a non-zero $x \in H^1(G_{\QQ,Np},\chi^i)$ with $i \in \{1,-1\}$, let $\rho^{\univ,\ell}_x : G_{\QQ,N\ell p} \to \GL_2(R^{\defo,\ell}_{\rhob_x})$ be the universal deformation of $\rhob_x$ over $R^{\defo,\ell}_{\rhob_x}$. 
\begin{prop}
\label{structprop}
Suppose $\rhob_0$ is unobstructed and $p \nmid \phi(N)$. Let $\ell$ be a prime such that $p \nmid \ell^2-1$, $\chi^i|_{G_{\QQ_{\ell}}} = \omega_p|_{G_{\QQ_{\ell}}}$ for some $i \in \{1,-1\}$. Then, for any non-zero $x \in H^1(G_{\QQ,Np},\chi^{-i})$, $\calR^{\pd,\ell}_{\rhob_0} \simeq \calR^{\defo,\ell}_{\rhob_x}$.
\end{prop}
\begin{proof}
 Without loss of generality, suppose $\chi|_{G_{\QQ_{\ell}}} = \omega_p|_{G_{\QQ_{\ell}}}$. By Lemma~\ref{cohomlem}, we have $\dim(H^1(G_{\QQ,N\ell p},\chi)) = 2$ and $\dim(H^1(G_{\QQ,N\ell p},\chi^{-1})) =1$. So by Proposition~\ref{nonzeroprop}, it suffices to prove that $p$ is not a zero divisor in $\calR^{\defo,\ell}_{\rhob_x}$ for any non-zero $x \in H^1(G_{\QQ,Np},\chi^{-1})$.

 By Lemma~\ref{tandimlem}, $\dim(\tan(R^{\pd,\ell}_{\rhob_0})) = 4$. By Theorem~\ref{generalthm}, $R^{\pd,\ell}_{\rhob_0} \simeq R^{\defo,\ell}_{\rhob_x}$ for any non-zero $x \in H^1(G_{\QQ,Np},\chi^{-1})$. Hence, we have $\dim(H^1(G_{\QQ,N\ell p},\text{ad}(\rhob_x))) = 4$ for any non-zero $x \in H^1(G_{\QQ,Np},\chi^{-1})$. By Lemma~\ref{cohomlem2}, this means that $\dim(H^2(G_{\QQ,N\ell p},\text{ad}(\rhob_x)))=1$. Therefore, by \cite[Theorem $2.4$]{Bo2}, $\calR^{\defo,\ell}_{\rhob_x} \simeq W(\FF)[[X,Y,Z,W]]/I$ where $I$ is either $(0)$ or a principal ideal of $W(\FF)[[X,Y,Z,W]]$.

Suppose $p$ is a zero divisor in $\calR^{\defo,\ell}_{\rhob_x}$. As $W(\FF)[[X,Y,Z,W]]$ is a regular local ring, it is a UFD (\cite[Theorem 19.19]{E}). This means that $I=(pf)$ for some $f \in W(\FF)[[X,Y,Z,W]]$. Thus, we get $R^{\defo,\ell}_{\rhob_x} \simeq \FF[[X,Y,Z,W]]$.

Fix a lift $g_\ell$ of $\text{Frob}_\ell$ in $G_{\QQ_\ell}$. From Lemma~\ref{ramlem}, we know that $\rho_x^{\univ,\ell}(g_\ell) = \begin{pmatrix} \phi_1 & 0 \\ 0 & \phi_2 \end{pmatrix}$, $\rho_x^{\univ,\ell}|_{I_{\ell}}$ factors through the $\ZZ_p$-quotient of the tame inertia group at ${\ell}$ and $\rho_x^{\univ,\ell}(i_{\ell}) = \begin{pmatrix} 1 & w \\ 0 & 1 \end{pmatrix}$ for some $w \in R^{\defo,\ell}_{\rhob_x}$. From the action of $\text{Frob}_{\ell}$ on the tame inertia group at $\ell$, we see that $(\phi_1/\phi_2-\ell)w=0$.

 If $w=0$, then the universal deformation $\rho^{\univ,\ell}_x$ factors through $G_{\QQ,Np}$. This would imply that $R^{\defo,\ell}_{\rhob_x} \simeq R^{\defo}_{\rhob_x}$ which is not true as we know $\dim(\tan(R^{\defo,\ell}_{\rhob_x})) =4$. Therefore, we see that $w \neq 0$. As $R^{\defo,\ell}_{\rhob_x}$ is an integral domain, we get that $\phi_1/\phi_2=\ell$.

By Lemma~\ref{tangenlem} and Lemma~\ref{surjectlem}, it follows that there exists a $z \in R^{\defo,\ell}_{\rhob_x}$ such that $w$, $z$ and $\phi_1-\chi_1(\text{Frob}_{\ell})$ generate the maximal ideal of $R^{\defo,\ell}_{\rhob_x}$ which contradicts the fact that $\dim(\tan(R^{\defo,\ell}_{\rhob_x})) =4$.
Hence, $R^{\defo,\ell}_{\rhob_x} \not\simeq \FF[[X,Y,Z,W]]$ and $p$ is not a zero-divisor in $\calR^{\defo,\ell}_{\rhob_x}$. This finishes the proof of the proposition.
\end{proof}

As a corollary, we get:
\begin{cor}
\label{strcor}
Suppose $\rhob_0$ is unobstructed and $p \nmid \phi(N)$. Let $\ell$ be a prime such that $p \nmid \ell^2-1$ and $\chi^i|_{G_{\QQ_{\ell}}} = \omega_p|_{G_{\QQ_{\ell}}}$ for some $i \in \{1,-1\}$. Then $\calR^{\pd,\ell}_{\rhob_0} \simeq W(\FF)[[X_1,X_2,X_3,X_4]]/(X_4 f)$ for some non-zero, non-unit $f \in W(\FF)[[X_1,X_2,X_3,X_4]]$.
\end{cor}
\begin{proof}
From the proof of Proposition~\ref{structprop}, we see that $\calR^{\pd,\ell}_{\rhob_0} \simeq W(\FF)[[X_1,X_2,X_3,X_4]]/I$ where $I$ is a non-zero principal ideal contained in $(p,(X_1,X_2,X_3,X_4)^2)$. Since the natural map $ \calR^{\pd,\ell}_{\rhob_0} \to \calR^{\pd}_{\rhob_0}$ is surjective (\cite[Proposition 6.1]{R}) and $\calR^{\pd}_{\rhob_0} \simeq W(\FF)[[X,Y,Z]]$, it follows that its kernel is a minimal prime of $\calR^{\pd,\ell}_{\rhob_0}$ and it is a principal ideal. This finishes the proof of the corollary.
\end{proof}

We will now prove an improvement of Corollary~\ref{strcor} in certain cases.

\begin{thm}
\label{ramunobsprop}
Suppose $\rhob_0$ is unobstructed and $p \nmid \phi(N)$. Let $\ell$ be a prime such that $p \nmid \ell^2-1$, $\chi^i|_{G_{\QQ_{\ell}}} = \omega_p|_{G_{\QQ_{\ell}}}$ for some $i \in \{1,-1\}$ and $\ell/\tilde \ell$ is a topological generator of $1 + p\ZZ_p$. Then $\calR^{\pd,\ell}_{\rhob_0} \simeq W(\FF)[[X_1,X_2,X_3,X_4]]/(X_4 X_2)$.
\end{thm}
\begin{proof}
 Without loss of generality assume $\chi|_{G_{\QQ_{\ell}}} = \omega_p|_{G_{\QQ_{\ell}}}$. By Proposition~\ref{structprop}, we have $\calR^{\pd,\ell}_{\rhob_0} \simeq \calR^{\defo,\ell}_{\rhob_x}$ for any non-zero $x \in H^1(G_{\QQ,N\ell p},\chi^{-1})$. Therefore, there exists a representation $\rho : G_{\QQ,N\ell p} \to \GL_2(\calR^{\pd,\ell}_{\rhob_0})$ such that $\tr(\rho) = T^{\univ,\ell}$. 

Fix a lift $g_\ell$ of $\text{Frob}_\ell$ in $G_{\QQ_\ell}$. From Lemma~\ref{ramlem}, we know that $\rho(g_\ell) = \begin{pmatrix} \phi_1 & 0 \\ 0 & \phi_2 \end{pmatrix}$, $\rho|_{I_{\ell}}$ factors through the $\ZZ_p$-quotient of the tame inertia group at ${\ell}$ and $\rho(i_{\ell}) = \begin{pmatrix} 1 & w \\ 0 & 1 \end{pmatrix}$ for some $w \in R^{\pd,\ell}_{\rhob_0}$. 

From the proof of Proposition~\ref{structprop}, we also get that $w \neq 0$ and $w(\phi_1/\phi_2 -\ell)=0$ i.e. $w(\phi_1-\ell\phi_2)=0$. By Lemma~\ref{ramlem}, there exist $y$, $y' \in \calR^{\pd,\ell}_{\rhob_0}$ such that $\phi_1 = \widehat{\chi_1(g_\ell)}(1+y)$ and $\phi_2 = \widehat{\chi_2(g_\ell)}(1+y')$. Now, $\phi_1 - \ell\phi_2 = \widehat{\chi_1(g_\ell)}-\ell\widehat{\chi_2(g_\ell)} +  \widehat{\chi_1(g_\ell)}y - \ell \widehat{\chi_2(g_\ell)}y'$ and $\widehat{\chi_1(g_\ell)} = \tilde\ell\widehat{\chi_2(g_\ell)}$. As $\ell/\tilde\ell$ is a topological generator of $1+p\ZZ_p$, it follows that $1-\ell/\tilde\ell = pu$ for some $u \in \ZZ_p^*$. Hence, $\widehat{\chi_1(g_\ell)}^{-1}(\phi_1 - \ell\phi_2) = pu+y-(1-pu) y'$. So we have $w(pu+y-(1-pu) y')=0$.

By Lemma~\ref{tangenlem}, there exists a $z \in \calR^{\pd,\ell}_{\rhob_0}$ such that the set $\{p,y,y',z,w\}$ generates the maximal ideal of $\calR^{\pd,\ell}_{\rhob_0}$.
Therefore, the set $\{p,pu+y-(1-pu) y',y,z,w\}$ also generates the maximal ideal of $\calR^{\pd,\ell}_{\rhob_0}$. Hence, by \cite[Theorem 7.16 (b)]{E}, we get a surjective map $\psi : W(\FF)[[X,Y,Z,W]] \to  \calR^{\pd,\ell}_{\rhob_0}$ sending $X$ to $pu+y-(1-pu) y'$, $Y$ to $y$, $Z$ to $ z$ and $W$ to $w$. The relation $w(pu+y-(1-pu) y')=0$ implies that $W X \in J$.

By Corollary~\ref{strcor}, it follows that $\calR^{\pd,\ell}_{\rhob_0} \simeq W(\FF)[[X,Y,Z,W]]/I$ where $I$ is a principal ideal. Therefore, $J$ is also a principal ideal.
We already have $WX \in J$. Note that $W(\FF)[[X,Y,Z,W]]$ is a UFD (by \cite[Theorem 19.19]{E}) and both $W$, $X$ are irreducible elements of it. Hence, $J$ is either $(W)$, $(X)$ or $(WX)$. Since $\dim(\tan(R^{\pd,\ell}_{\rhob_0}))=4$, $J$ cannot be $(W)$ or $(X)$. Hence, $\calR^{\pd,\ell}_{\rhob_0} \simeq W(\FF)[[X,Y,Z,W]]/(WX)$.
\end{proof}

\begin{rem}
By Theorem~\ref{ramunobsprop}, we know that $\calR^{\defo,\ell}_{\rhob_x} \simeq W(\FF)[[X,Y,Z,W]]/(WX)$ for a suitable $\rhob_x$. It is not clear how to get this explicit structure of $\calR^{\defo,\ell}_{\rhob_x}$ directly from \cite[Theorem $4.7$]{Bo} or its proof.
\end{rem}

\subsection{Structure of $R^{\pd,\ell}_{\rhob_0}$ with unobstructed $\rhob_0$ and $p | \ell + 1$}
We now turn to the case where $\rhob_0$ is unobstructed and $\ell$ is a prime such that $\ell \nmid Np$ and $p | \ell + 1$. As we will see, this case is a bit more complicated than the previous case. This is also the case in the study undertaken in \cite{Bos} and \cite{Bo}. We begin by determining the explicit structure of $R^{\defo,\ell}_{\rhob_x}$ under certain hypotheses. 

Before proceeding further, we need a piece of notation. Let $\{h_i | i \in \ZZ, i \geq 0\}$ be the set of polynomials in $\FF[\sqrt{1+UV}]$ satisfying the recurrence relation $b_{i+1}-2(\sqrt{1+UV})b_i+b_{i-1} =0$ with $h_0=0$ and $h_1=1$ (see \cite{Bos} for more details). So $\{h_i | i \in \ZZ, i \geq 0\} \subset \FF[[U,V]]$. Note that $h_{\ell} \equiv \ell \pmod{(UV)}$. For a non-zero $x \in H^1(G_{\QQ,Np},\chi^i)$ with $i \in \{1,-1\}$, let $\tau^{\univ,\ell}_x : G_{\QQ,N\ell p} \to \GL_2(R^{\defo,\ell}_{\rhob_x})$ be the universal deformation of $\rhob_x$.

Note that if $p | \ell+1$ but $p^2 \nmid \ell+1$, then $\ell/\tilde\ell$ is a topological generator of $1+p\ZZ_p$.
\begin{lem}
\label{defstrlem}
Suppose $\rhob_0$ is unobstructed and $p \nmid \phi(N)$. Let $\ell$ be a prime such that $p | \ell+1$, $p^2 \nmid \ell+1$ and $\chi|_{G_{\QQ_{\ell}}} = \omega_p|_{G_{\QQ_{\ell}}}$. Let $x \in H^1(G_{\QQ,Np},\chi^i)$ be a non-zero element for $i \in \{1,-1\}$. Then, $$R^{\defo,\ell}_{\rhob_x} \simeq \FF[[X,Y,Z,U,V]]/(U((1+X)+h_{\ell}(1+Y)), V((1+Y)+h_{\ell}(1+X))).$$
\end{lem}
\begin{proof}
By Lemma~\ref{ramlem}, it follows that $\tau^{\univ,\ell}_x|_{I_{\ell}}$ factors through the $\ZZ_p$-quotient of the tame inertia group at ${\ell}$, $\tau^{\univ,\ell}_x(i_{\ell}) = \begin{pmatrix} \sqrt{1+uv} & u \\ v & \sqrt{1+uv} \end{pmatrix}$ and $\tau^{\univ,\ell}_x(g_\ell) = \begin{pmatrix} \phi_1 & 0 \\ 0 & \phi_2 \end{pmatrix}$ for a fixed lift $g_\ell$ of $\text{Frob}_\ell$ in $G_{\QQ_\ell}$. Note that there exist $m, n \in R^{\defo,\ell}_{\rhob_x}$ such that $\phi_1=\chi_1(\text{Frob}_\ell)(1+m)$ and  $\phi_2=\chi_2(\text{Frob}_\ell)(1+n)$.

By Lemma~\ref{tangenlem}, there exists a $z \in R^{\defo,\ell}_{\rhob_x}$ such that the set $\{m,n,u,v,z\}$ generates the maximal ideal of $R^{\defo,\ell}_{\rhob_x}$.
Thus, by \cite[Theorem 7.16 (b)]{E}, we have a surjective map $\phi : \FF[[X,Y,Z,U,V]] \to R^{\defo,\ell}_{\rhob_x}$ of $W(\FF)$-algebras sending $X$ to $m$, $Y$ to $n$, $Z$ to $z$, $U$ to $u$ and $V$ to $v$. Let $J_0=\ker(\phi)$. 

From the action of $\text{Frob}_{\ell}$ on the tame inertia group at $\ell$, we see that $(\phi_1/\phi_2-h_{\ell})u=0$ and $(\phi_2/\phi_1-h_{\ell})v=0$. Note that, as $p | \ell+1$ and $\chi|_{G_{\QQ_{\ell}}}=\omega_p|_{G_{\QQ_{\ell}}}$, we have $\chi_1(\text{Frob}_{\ell})=-\chi_2(\text{Frob}_{\ell})$. Therefore, we have $((1+m)+h_{\ell}(1+n))u=0$ and $((1+n)+h_{\ell}(1+m))v=0$. So $((1+X)+h_{\ell}(1+Y))U$, $((1+Y)+h_{\ell}(1+X))V \in J_0$. 

By Lemma~\ref{cohomlem}, we know that $\dim(H^1(G_{\QQ,N\ell p}, \text{ad}(\rhob_x)))=5$ and $\dim(H^2(G_{\QQ,N\ell p}, \text{ad}(\rhob_x)))=2$. By \cite[Theorem $2.4$]{Bo2}, $R^{\defo,\ell}_{\rhob_x} \simeq \FF[[X_1,X_2,X_3,X_4,X_5]]/J$, where $J$ is generated by at most $2$ elements and $J \subset (X_1,X_2,X_3,X_4,X_5)^2$. 
Denote $\FF[[X,Y,Z,U,V]]$ by $R$ and its maximal ideal $(X,Y,Z,U,V)$ by $m_0$. Therefore, $J_0$ is generated by at most $2$ elements and $J_0 \subset m_0^2$.  

Note that $h_{\ell} \equiv \ell \pmod{(UV)}$. Since $p | \ell+1$, we get $((1+X)+h_{\ell}(1+Y)) \equiv (X- Y) \pmod{(UV)}$ and $((1+Y)+h_{\ell}(1+X)) \equiv (Y- X) \pmod{(UV)}$. So $(1+X)+h_{\ell}(1+Y)$, $(1+Y)+h_{\ell}(1+X) \in m_0 \setminus m_0^2$. As $m_0J_0 \subset m_0^3$, we see that the images of the elements $((1+Y)+h_{\ell}(1+X))V$ and $((1+X)+h_{\ell}(1+Y))U$ in $J_0/m_0J_0$ are linearly independent over $\FF$. As $J_0$ is generated by at most 2 elements, the dimension of $J_0/m_0J_0$ as a vector space over $\FF$ is at most $2$. Hence, it follows, from Nakayama's lemma, that $J_0 = (((1+Y)+h_{\ell}(1+X))V, ((1+X)+h_{\ell}(1+Y))U)$.
\end{proof}

We now turn our attention to the problem of finding the structure of $R^{\pd,\ell}_{\rhob_0}$ when $\rhob_0$ is unobstructed, $p | \ell+1$ and $\chi|_{G_{\QQ_\ell}} = \omega_p$. Note that in this case, we have $\dim(H^1(G_{\QQ,N\ell p},\chi))= \dim(H^1(G_{\QQ,N\ell p},\chi^{-1})) =2 $. So this case is different from the cases we have dealt with so far. Hence, we can not use the results obtained so far. 
However, we can still use the technique of comparing $R^{\pd,\ell}_{\rhob_0}$ with the universal deformation rings of residually non-split reducible representations.

\begin{thm}
\label{reducestrthm}
Suppose $\rhob_0$ is unobstructed and $p \nmid \phi(N)$. Let $\ell$ be a prime such that $p | \ell + 1$, $p^2 \nmid \ell+1$ and $\chi|_{G_{\QQ_{\ell}}} = \omega_p|_{G_{\QQ_{\ell}}}$. Then, $$(R^{\pd,\ell}_{\rhob_0})^{\red} \simeq \FF[[X,Y,Z,T_1,T_2]]/(T_1T_2,T_1Z,T_2Z).$$
\end{thm}

We will first prove a series of lemmas which will be used to prove Theorem~\ref{reducestrthm}. 

Let $P$ be a prime $R^{\pd,\ell}_{\rhob_0}$. Fix a lift $g_{\ell}$ of $\text{Frob}_{\ell}$ in $G_{\QQ_{\ell}}$. Let $A_P$ be the GMA obtained in Lemma~\ref{gmalem} for the tuple $(R^{\pd,\ell}_{\rhob_0}/P, \ell, t^{\univ,\ell} \pmod{P},g_{\ell})$. Let $A_P = \begin{pmatrix} R^{\pd,\ell}_{\rhob_0}/P & B_P\\ C_P & R^{\pd,\ell}_{\rhob_0}/P\end{pmatrix}$ and $\rho_P : G_{\QQ,N\ell p} \to A_P^*$ be the corresponding representation. By Part~\eqref{bhaag3} of Lemma~\ref{gmalem},  we see that $\rho_P|_{I_{\ell}}$ factors through the $\ZZ_p$-quotient of the tame inertia group at $\ell$. Fix a generator $i_{\ell}$ of this $\ZZ_p$-quotient. We will now use this notation throughout the paper.
\begin{lem}
\label{reducelem}
 Suppose $\ell$ is a prime such that $\ell \nmid Np$, $p\nmid \ell-1$ and $\chi|_{G_{\QQ_{\ell}}} \neq 1$. If $P$ is a prime of $R^{\pd,\ell}_{\rhob_0}$, then $t^{\univ,\ell}(gh) - t^{\univ,\ell}(g) \in P$ for all $g \in G_{\QQ_\ell}$ and $h \in I_{\ell}$.
\end{lem}
\begin{proof}
By Lemma~\ref{redgmalem}, we can choose $A_P$ to be a subalgebra of $M_2(K_P)$(see \cite[Lemma $2.2.2$]{Bel} as well).
By the action of $\text{Frob}_{\ell}$ on the tame inertia group by conjugation, we see that $\rho_P(i_{\ell})$ is conjugate to $\rho_P(i_{\ell})^{\ell}$. So if $a \in \bar K_P$ is an eigenvalue of $\rho_P(i_{\ell})$, then $a^{\ell}$ is also an eigenvalue of $\rho_P(i_{\ell})$. As $p \nmid \ell-1$, $\det(\rho_P(I_{\ell}))=1$. Hence, we get that either $a^{\ell}=a$ or $a^{\ell}=a^{-1}$ which means $a$ is an $m$-th root of unity for some $m \in \mathbb{N}$.
Since $K_P$ has characteristic $p$ and $i_{\ell}$ is a generator of the $\ZZ_p$-quotient of $I_{\ell}$, it follows that $1$ is the only eigenvalue of  $\rho_P(i_{\ell})$. 

So there exists some $Q \in \GL_2(K_P)$ such that $Q\rho_P(i_{\ell})Q^{-1} = \begin{pmatrix} 1 & w\\ 0 & 1 \end{pmatrix}$ for some $w \in K_P$. Thus, $Q\rho_P(I_{\ell})Q^{-1} = \{\begin{pmatrix} 1 & n.w\\ 0 & 1 \end{pmatrix} | 0 \leq n \leq p-1\}$. As $I_{\ell}$ is normal in $G_{\QQ_{\ell}}$, we see that $Q\rho_P(G_{\QQ_{\ell}})Q^{-1}$ is a subgroup of the group of upper triangular matrices in $\GL_2(K_P)$. Hence, we conclude that $\tr(\rho_P(gh))-\tr(\rho_P(g)) = 0$ for all $g \in G_{\QQ_\ell}$ and $h \in I_{\ell}$. Since $t^{\univ,\ell} \pmod{P} = \tr(\rho_P)$, the lemma follows.
\end{proof}

\begin{lem}
\label{redsurjlem}
Suppose $\rhob_0$ is unobstructed and $p \nmid \phi(N)$. Let $\ell$ be a prime such that $p | \ell + 1$, $p^2 \nmid \ell+1$ and $\chi|_{G_{\QQ_{\ell}}} = \omega_p|_{G_{\QQ_{\ell}}}$.
Then $(R^{\pd,\ell}_{\rhob_0})^{\red}$ is a quotient of $\FF[[X,Y,Z,X_1,X_2]]/(X_1Y,X_2Y,X_1X_2)$.
\end{lem}
\begin{proof}
Fix a lift $g_{\ell}$ of $\text{Frob}_{\ell}$ in $G_{\QQ_{\ell}}$. Let $A^{\red} = \begin{pmatrix} (R^{\pd,\ell}_{\rhob_0})^{\red} & B^{\red}\\ C^{\red} & (R^{\pd,\ell}_{\rhob_0})^{\red}\end{pmatrix}$ be the GMA for the tuple $((R^{\pd,\ell}_{\rhob_0})^{\red}, \ell, (t^{\univ,\ell})^{\red}, g_{\ell})$ obtained in Lemma~\ref{gmalem} and $\rho^{\red}$ be the corresponding representation. Let $K_0$ be the total fraction field of $(R^{\pd,\ell}_{\rhob_0})^{\red}$. By Lemma~\ref{redgmalem}, we can take $B^{\red}$ and $C^{\red}$ to be the fractional ideals of $K_0$ such that the map $m'(B^{\red} \otimes_{(R^{\pd,\ell}_{\rhob_0})^{\red}} C^{\red})$ coincides with the multiplication in $K_0$. 

From Lemma~\ref{gmalem}, we know that $\rho^{\red}(g_{\ell})=\begin{pmatrix} a^{\red} & 0 \\ 0 & d^{\red}\end{pmatrix}$ with $a^{\red}$ and $d^{\red}$ \emph{not} congruent modulo the maximal ideal of $(R^{\pd,\ell}_{\rhob_0})^{\red}$. From Part~\eqref{bhaag3} of Lemma~\ref{gmalem}, it follows that $\rho^{\red}(I_{\ell})$ is topologically generated by $\rho^{\red}(i_\ell)$ which means $\rho^{\red}(G_{\QQ_{\ell}})$ is topologically generated by $\rho^{\red}(g_{\ell})$ and $\rho^{\red}(i_{\ell})$. 

Suppose $\rho^{\red}(i_{\ell}) = \begin{pmatrix} a & b\\ c & d \end{pmatrix}$. From Lemma~\ref{reducelem}, we get that $a+d=2$, $ad-bc=1$ and $a^{\red}a+d^{\red}d=a^{\red}+d^{\red}$. If $a=1+\alpha$ and $d=1-\alpha$, then we have $a^{\red}(1+\alpha)+d^{\red}(1-\alpha) = a^{\red}+d^{\red}$. Simplifying, we get $\alpha(a^{\red}-d^{\red})=0$. As $a^{\red}-d^{\red} \in ((R^{\pd,\ell}_{\rhob_0})^{\red})^*$, we get $\alpha=0$. Hence, $a=d=1$ and $bc=0$.

By Lemma~\ref{gengmalem}, we see that $C^{\red}$ and $B^{\red}$ are generated by at most two elements and there exists $b' \in B^{\red}$ and $c' \in C^{\red}$ such that $\{b,b'\}$ is a set of generators of $B^{\red}$, while $\{c,c'\}$ is a set of generators of $C^{\red}$. Let $z=b'c'$, $x_1=bc'$ and $x_2= b'c$. Now, $a^{\red} = \chi_1(\text{Frob}_{\ell})(1+a_0)$ and $d^{\red} = \chi_2(\text{Frob}_{\ell})(1+d_0)$ for some $a_0,d_0 \in m^{\red}$ where $m^{\red}$ is the maximal ideal of $(R^{\pd,\ell}_{\rhob_0})^{\red}$. 

By Lemma~\ref{genlem} and Lemma~\ref{gmalem}, the ideal generated by the set $\{a_0,d_0,z,x_1,x_2\}$ is $m^{\red}$. Thus, by \cite[Theorem 7.16 (b)]{E}, we get a surjective local morphism of $\FF$-algebras $g_0 : \FF[[X,Y,Z,X_1,X_2]] \to (R^{\pd,\ell}_{\rhob_0})^{\red}$ such that $g_0(X) = a_0 + d_0$, $g_0(Y) = a_0 - d_0$, $g_0(Z) = z$, $g_0(X_1)=x_1$ and $g_0(X_2)=x_2$. 

Let $I_0 = \ker(g_0)$. As $bc=0$, we get $x_1.x_2=bc'.b'c=0$. So $X_1X_2 \in I_0$. Note that, from the action of $\text{Frob}_{\ell}$ on the tame inertia group, we get $\rho^{\red}(g_{\ell}i_{\ell}g_{\ell}^{-1})=\rho^{\red}(i_{\ell})^{\ell}$. Now, $\rho^{\red}(g_{\ell}i_{\ell}g_{\ell}^{-1}) = \begin{pmatrix} 1 & (a^{\red}/d^{\red})b \\ (d^{\red}/a^{\red})c & 1\end{pmatrix}$. As $bc=0$, we have $\rho^{\red}(i_{\ell})^{\ell} = \begin{pmatrix} 1 & \ell.b \\ \ell.c & 1\end{pmatrix}$. Thus, we have $(a^{\red}/d^{\red} - \ell)b=0$ i.e. $(a^{\red}-\ell.d^{\red})b=0$ and $(d^{\red}/a^{\red} - \ell)c=0$ i.e. $(d^{\red}-\ell.a^{\red})c=0$. As $\chi_1(\text{Frob}_{\ell})/\chi_2(\text{Frob}_{\ell})=\omega_p(\text{Frob}_{\ell})=\ell$, we get $(a_0-d_0)b=0$ and $(d_0-a_0)c=0$. Thus, $(a_0-d_0)x_1=(a_0-d_0)x_2=0$ and hence, $YX_1, YX_2 \in I_0$.
\end{proof}

\begin{lem}
\label{nominprimelem}
Suppose $\rhob_0$ is unobstructed and $p \nmid \phi(N)$. Let $\ell$ be a prime such that $p | \ell + 1$, $p^2 \nmid \ell+1$ and $\chi|_{G_{\QQ_{\ell}}} = \omega_p|_{G_{\QQ_{\ell}}}$. Then there exist distinct prime ideals $P_0$, $P_1$ and $P_2$ of $R^{\pd,\ell}_{\rhob_0}$ such that $\dim(R^{\pd,\ell}_{\rhob_0}/(P_i)) \geq 3$ for $i=0,1,2$.
\end{lem}
\begin{proof}
Fix a non-zero element $x_0 \in H^1(G_{\QQ,Np},\chi)$. 
Recall that, in Lemma~\ref{defstrlem}, we constructed an isomorphism $\phi : R:= \FF[[X,Y,Z,U,V]]/(U(X+h_\ell Y),V(Y+h_\ell X)) \to R^{\defo,\ell}_{\rhob_{x_0}}$ which sends images of $X$, $Y$, $U$ and $V$ in $R$ to $x$, $y$, $u$ and $v$, respectively, where $\tau^{\univ,\ell}_{x_0}(i_\ell) = \begin{pmatrix} \sqrt{1+uv} & u\\ v & \sqrt{1+uv}\end{pmatrix}$ and $\tau^{\univ,\ell}_{x_0}(g_\ell) = \begin{pmatrix}\chi_1(\text{Frob}_\ell)(1+x) & 0\\ 0 & \chi_2(\text{Frob}_\ell)(1+y)\end{pmatrix}$. Here $i_{\ell}$ is a topological generator of the $\ZZ_p$-quotient of the tame inertia group at $\ell$ and $g_\ell$ is a lift of $\text{Frob}_{\ell}$ in $G_{\QQ,N\ell p}$.

Hence, it follows that $Q_0=(u,v)$, $Q_1=(u,x-y)$ and $Q_2=(v,x-y)$ are $3$ distinct primes ideals of $R^{\defo,\ell}_{\rhob_{x_0}}$ such that $R^{\defo,\ell}_{\rhob_{x_0}}/Q_i \simeq \FF[[X,Y,Z]]$ for $i=0,1,2$.

Let $g : R^{\pd,\ell}_{\rhob_0} \to R^{\defo,\ell}_{\rhob_{x_0}}$ be the map induced by $\tr(\tau^{\univ,\ell}_{x_0})$. For $i=0,1,2$, we get a morphism $g_i : R^{\pd,\ell}_{\rhob_0} \to R^{\defo,\ell}_{\rhob_{x_0}}/Q_i$ composing $g$ with the natural surjective morphism $R^{\defo,\ell}_{\rhob_{x_0}} \to R^{\defo,\ell}_{\rhob_{x_0}}/Q_i$. Let $P_i$ be $\ker(g_i)$ for $i=0,1,2$.

By Lemma~\ref{unobslem} and Lemma~\ref{surjlem}, there is a surjective map $f : R^{\defo,\ell}_{\rhob_{x_0}} \to R^{\pd}_{\rhob_0}$ such that $f \circ \tr(\tau^{\univ,\ell}_{x_0}) = t^{\univ}$ and $\ker(f) = (u,v)$.  So $f \circ g \circ t^{\univ,\ell}=t^{\univ}$. Hence, by \cite[Proposition 6.1]{R}, $f \circ g : R^{\pd,\ell}_{\rhob_0} \to R^{\pd}_{\rhob_0}$ is surjective. From the definition of $P_0$, we see that $P_0 = \ker(f \circ g)$. Since $\rhob_0$ is unobstructed and $p \nmid \phi(N)$, Lemma~\ref{unobslem} implies that $\dim(R^{\pd,\ell}_{\rhob_0}/P_0) = 3$.

We will denote $\tau^{\univ,\ell}_{x_0}$ by $\rho$ for the rest of the proof. 
From the description of $\rho(g_{\ell})$ and \cite[Lemma $2.4.5$]{Bel}, it follows that there exist ideals $B$ and $C$ of $R^{\defo,\ell}_{\rhob_{x_0}}$ such that $R^{\defo,\ell}_{\rhob_{x_0}}[\rho(G_{\QQ,N\ell p})] = \begin{pmatrix} R^{\defo,\ell}_{\rhob_{x_0}} & B \\ C & R^{\defo,\ell}_{\rhob_{x_0}} \end{pmatrix}$.
 As $\rho$ is a deformation of $\rhob_{x_0}$, it follows that $B = R^{\defo,\ell}_{\rhob_{x_0}}$.

Now let $h := \begin{pmatrix} 1 & 1 \\ 0  & 1\end{pmatrix} \in R^{\defo,\ell}_{\rhob_{x_0}}[\rho(G_{\QQ,N\ell p})]$.
Then $\tr(h.\rho(i_{\ell}))-\tr(h) = v.\alpha$ for some $\alpha \in (R^{\defo,\ell}_{\rhob_{x_0}})^{\times}$.
Observe that $\tr(h.\rho(i_{\ell}))-\tr(h) \in \text{Im}(g)$, $\tr(h.\rho(i_{\ell}))-\tr(h) \in Q_2$ but $\tr(h.\rho(i_{\ell}))-\tr(h) \not\in Q_1$.
Therefore, $\tr(h.\rho(i_{\ell}))-\tr(h) \in P_2$ but $\tr(h.\rho(i_{\ell}))-\tr(h) \not\in P_1$ which means $P_1 \neq P_2$.

From above, we know that the map $g_0$ induces an isomorphism $R^{\pd,\ell}_{\rhob_0}/P_0 \simeq R^{\defo,\ell}_{\rhob_{x_0}}/(u,v)$. Hence, the map $\eta : R^{\pd,\ell}_{\rhob_0} \to  R^{\defo,\ell}_{\rhob_{x_0}}/(u,v,x-y)$ obtained by composing $g$ with the natural map $R^{\defo,\ell}_{\rhob_{x_0}} \to  R^{\defo,\ell}_{\rhob_{x_0}}/(u,v,x-y)$ is a surjective map. Now, $R_0:= R^{\defo,\ell}_{\rhob_{x_0}}/(u,v,x-y) \simeq \FF[[X,Y]]$. 
Denote the $R_0$-valued representation $\rho \pmod{(u,v,x-y)}$ by $\rho_0$.

Now, $\rho (i_\ell) \pmod{Q_1}$ is a non-identity lower triangular matrix with diagonal entries $1$. 
So if $t^{\univ,\ell} \pmod{P_1} = \tr(\rho) \pmod{Q_1}$ is unramified at $\ell$, then $\tr(\rho) \pmod{Q_1}$ is reducible which means $\tr(\rho_0)$ is also reducible.
However, $\rho_0(g_\ell) = \begin{pmatrix} \chi_1(\text{Frob}_\ell)(1+\alpha) & 0 \\ 0 &  \chi_2(\text{Frob}_\ell)(1+\alpha) \end{pmatrix}$ for some $\alpha \in R_0$.
So the last part of Lemma~\ref{gmalem} implies that $(\alpha)$ is the maximal ideal of $R_0$ contradicting the fact that $R_0 \simeq \FF[[X,Y]]$. 
Hence, $\tr(\rho)$ is not reducible which means $t^{\univ,\ell} \pmod{P_1}$ is not unramified at $\ell$.

On the other hand, $\rho (i_\ell) \pmod{Q_2}$ is a non-identity upper triangular matrix with diagonal entries $1$. 
Then, using the logic of the previous paragraph, we conclude that $t^{\univ,\ell} \pmod{P_2}$ is not unramified at $\ell$.
Therefore, we get that $P_0 \not\subset P_i$ for $i=1,2$ which means $P_0$, $P_1$ and $P_2$ are distinct.

Note that $\ker(\eta)$ is a prime ideal of $R^{\pd,\ell}_{\rhob_0}$ and $P_0 \neq \ker(\eta)$.
Now $P_i \subset \ker(\eta)$ for $i=0,1,2$. Hence, we conclude, using previous paragraph that $P_i \neq \ker(\eta)$ for $i=1,2$.

Thus we conclude that all $P_0$, $P_1$ and $P_2$ are proper subsets of $\ker(\eta)$.
As $\dim(R^{\pd,\ell}_{\rhob_0}/\ker(\eta)) = 2$ and $P_i$'s are prime ideals for $i=0,1,2$, we get that $\dim(R^{\pd,\ell}_{\rhob_0}/P_i) \geq 3$ for $i=1,2$.
\end{proof}

We are now ready to prove Theorem~\ref{reducestrthm}.
\begin{proof}[Proof of Theorem~\ref{reducestrthm}]

 From Lemma~\ref{redsurjlem}, we know that there exists a surjective morphism $g: \FF[[X,Y,Z,X_1,X_2]] \to (R^{\pd,\ell}_{\rhob_0})^{\red}$ such that $(X_1X_2,X_1Y,X_2Y) \subset \ker(g)$. We will denote $\ker(g)$ by $I_0$ for the rest of the proof. For $i=0,1,2$, let $P'_i$ be the kernel of the map $g_i :  \FF[[X,Y,Z,X_1,X_2]] \to R^{\pd,\ell}_{\rhob_0}/P_i$ obtained by composing $g$ with the surjective map $(R^{\pd,\ell}_{\rhob_0})^{\red} \to R^{\pd,\ell}_{\rhob_0}/P_i$. Here, the primes $P_i$ are the ones appearing in Lemma~\ref{nominprimelem}. Each $P'_i$ is a prime of $\FF[[X,Y,Z,X_1,X_2]]$ containing $I_0$ and in particular, $(X_1X_2, YX_1, YX_2) \subset P'_i$ for $i=0,1,2$. So each $P'_i$ contains one of the $(Y,X_1)$, $(Y,X_2)$ or $(X_1,X_2)$. 

Now, the Krull dimension of $R^{\pd,\ell}_{\rhob_0}/P_i$ and hence, that of $\FF[[X,Y,Z,X_1,X_2]]/P'_i$ is at least $3$ for $i=0,1,2$. Therefore, every $P'_i$ is either $(Y,X_1)$, $(Y,X_2)$ or $(X_1,X_2)$. Since $P_0$, $P_1$ and $P_2$ are distinct prime ideals of $R^{\pd,\ell}_{\rhob_0}$ (by Lemma~\ref{nominprimelem}), $P'_0$, $P'_1$ and $P'_2$ are distinct prime ideals of $\FF[[X,Y,Z,X_1,X_2]]$. Hence, we have $\{P'_0,P'_1,P'_2\}= \{(Y,X_1), (Y,X_2),(X_1,X_2)\}$. So $I_0 \subset P'_0 \cap P'_1 \cap P'_2 = (Y,X_1) \cap (Y,X_2) \cap (X_1,X_2)$.

 Note that $ (Y,X_2) \cap (Y,X_1) = (Y, X_1X_2)$. If $Yf \in (X_1,X_2)$, then $f \in (X_1,X_2)$ and hence, $Yf \in (YX_1,YX_2)$. Therefore, $(Y,X_1X_2) \cap (X_1,X_2)=(YX_1, YX_2, X_1X_2)$. Hence, $I_0 \subset (YX_1,YX_2,X_1X_2)$. This implies that $I_0= (YX_1,YX_2,X_1X_2)$ and hence, $(R^{\pd,\ell}_{\rhob_0})^{\red} \simeq \FF[[X,Y,Z,X_1,X_2]]/(YX_1, YX_2, X_1X_2)$.
\end{proof}

\begin{rem}
The proof of Theorem~\ref{reducestrthm}, description of the GMA $A^{\red}$, and \cite[Proposition $1.7.4$]{BC} together imply that there does not exists a representation $\rho : G_{\QQ,N\ell p} \to \GL_2((R^{\pd,\ell}_{\rhob_0})^{\red})$ such that $\tr(\rho)= (t^{\univ,\ell})^{\red}$.
\end{rem}

It is natural to ask if the same approach can give us the structure of $(\calR^{\pd,\ell}_{\rhob_0})^{\red}$ as well. But the method does not work. More specifically, Lemma~\ref{reducelem} is not true for $\calR^{\pd,\ell}_{\rhob_0}$. Indeed, let $x \in H^1(G_{\QQ,Np},\chi^i)$ be a non-zero element with $i \in \{1,-1\}$ and $\mathcal{O}$ be the ring of integers in the finite extension of $\QQ_p$ obtained by attaching all the $p$-th roots of unity to $\QQ_p$. Let $\zeta_p$ be a primitive $p$-th root of unity. It can be checked that there exists a $W(\FF)$-algebra morphism $\calR^{\defo,\ell}_{\rhob_x} = W(\FF)[[X,Y,Z,U,V]]/(U((1+X)+h_{\ell}(1+Y)), V((1+Y)+h_{\ell}(1+X))) \to \mathcal{O}[[Z]]$ sending both $U$ and $V$ to $\frac{\zeta_p -\zeta^{-1}_p}{2}$, $X$ and $Y$ to $0$ and $Z$ to $Z$. Composing this map with the map $\calR^{\pd,\ell}_{\rhob_0} \to \calR^{\defo,\ell}_{\rhob_x}$, we get a map $f : \calR^{\pd,\ell}_{\rhob_0} \to \mathcal{O}[[Z]]$. Observe that $f \circ T^{\univ,\ell}|_{G_{\QQ_{\ell}}}$ is not reducible and $\ker(f)$ is a prime ideal. See \cite[Section $3$]{Bos} for a similar analysis. Thus, the ring $\calR^{\pd,\ell}_{\rhob_0}$ has more than $3$ minimal primes and probably has a more complicated structure.

\begin{cor}
Suppose $\rhob_0$ is unobstructed and $p \nmid \phi(N)$. Let $\ell$ be a prime such that $p | \ell + 1$, $p^2 \nmid \ell+1$ and $\chi|_{G_{\QQ_{\ell}}} = \omega_p|_{G_{\QQ_{\ell}}}$. Then $R^{\pd,\ell}_{\rhob_0}$ is not reduced ring.
\end{cor}
\begin{proof}
 Lemma~\ref{tandimlem} and Lemma~\ref{cohomlem} imply $\dim(\tan(R^{\pd,\ell}_{\rhob_0}))=6$. Now the corollary follows directly from Theorem~\ref{reducestrthm}.
\end{proof}

Though we do not determine the explicit structure of $R^{\pd,\ell}_{\rhob_0}$ in this case, we can still prove the following theorem:
\begin{thm}
\label{lcithm}
Suppose $\rhob_0$ is unobstructed and $p \nmid \phi(N)$. Let $\ell$ be a prime such that $p | \ell + 1$, $p^2 \nmid \ell+1$ and $\chi|_{G_{\QQ_{\ell}}} = \omega_p|_{G_{\QQ_{\ell}}}$. Then $R^{\pd,\ell}_{\rhob_0}$ is not a local complete intersection ring.
\end{thm}
\begin{proof}
We use a strategy similar to the one used in the proof of Theorem~\ref{reducestrthm}. Namely, we first find a set of generators of the co-tangent space of $R^{\pd,\ell}_{\rhob_0}$ and then find the relations between them using GMAs. After assuming that $R^{\pd,\ell}_{\rhob_0}$ is a local complete intersection ring, we will find a subset of these relations which will generate all the relations in $R^{\pd,\ell}_{\rhob_0}$. But the description of this subset will give a contradiction to Theorem~\ref{reducestrthm} which will complete the proof.

 Fix a lift $g_{\ell}$ of $\text{Frob}_{\ell}$ in $G_{\QQ_{\ell}}$. Let $A^{\pd} = \begin{pmatrix} R^{\pd,\ell}_{\rhob_0} & B^{\pd}\\ C^{\pd} & \calR^{\pd,\ell}_{\rhob_0}\end{pmatrix}$ be the GMA associated to the tuple $(R^{\pd,\ell}_{\rhob_0}, \ell, t^{\univ,\ell},g_{\ell})$ in Lemma~\ref{gmalem} and $\rho : G_{\QQ,N\ell p} \to (A^{\pd})^*$ be the corresponding representation. By Part~\eqref{bhaag3} of Lemma~\ref{gmalem}, $\rho|_{I_{\ell}}$ factors through the $\ZZ_p$ quotient of the tame inertia group at $\ell$. Suppose $\rho(i_{\ell}) = \begin{pmatrix} a & b\\ c & d\end{pmatrix}$. By Lemma~\ref{gmalem}, we know that $\rho(g_{\ell}) = \begin{pmatrix} a_0 & 0\\ 0 & d_0\end{pmatrix}$. 

Let ${I}^{\ell}_{\rhob_0}:=m(B^{\pd} \otimes_{R^{\pd,\ell}_{\rhob_0}} C^{\pd})$. From Lemma~\ref{gengmalem}, it follows that there exists $b' \in B^{\pd}$ and $c' \in C^{\pd}$ such that $\{b,b'\}$ is a set of generators of $B^{\pd}$, while $\{c,c'\}$ is a set of generators of $C^{\pd}$. Thus, the ideal $I^{\ell}_{\rhob_0}$ is generated by the set $\{m'(b \otimes c), m'(b' \otimes c), m'(b \otimes c'), m'(b' \otimes c')\}$. Let $z=m'(b' \otimes c')$, $x_1=m'(b \otimes c')$, $x_2=m'(b' \otimes c)$ and $x_3=m'(b \otimes c)$. 

Now, $a_0 = \chi_1(\text{Frob}_{\ell})(1+a'_0)$ and $d_0 = \chi_2(\text{Frob}_{\ell})(1+d'_0)$ for some $a'_0,d'_0 \in \mathfrak{m}^{\ell}$ where $\mathfrak{m}^{\ell}$ is the maximal ideal of $R^{\pd,\ell}_{\rhob_0}$. 
From last part of Lemma~\ref{gmalem}, we see that the ideal generated by the set $\{a'_0,d'_0,z,x_1,x_2,x_3\}$ is $\mathfrak{m}^{\ell}$. 
Thus, we get a surjective local morphism of $\FF$-algebras $g_0 : \FF[[X,Y,Z,X_1,X_2,X_3]] \to R^{\pd,\ell}_{\rhob_0}$ such that $g_0(X) = a'_0 + d'_0$, $g_0(Y) = a'_0 - d'_0$, $g_0(Z) = z$, $g_0(X_1)=x_1$, $g_0(X_2)=x_2$ and $g_0(X_3)=x_3$. Let $J_0 = \ker(g_0)$. Denote the maximal ideal $(X,Y,Z,X_1,X_2,X_3)$ by $m_0$ and $ \FF[[X,Y,Z,X_1,X_2,X_3]]$ by $R_0$. We know that $\dim(\tan(R^{\pd,\ell}_{\rhob_0})) = 6$. Hence, $J_0 \subset m_0^2$. Suppose $R^{\pd,\ell}_{\rhob_0}$ is a local complete intersection ring. The Krull dimension of $R^{\pd,\ell}_{\rhob_0}$ is $3$ by Theorem~\ref{reducestrthm}. This means that $J_0$ is generated by $3$ elements.

Note that if $g \in G_{\QQ_{\ell}}$ and $\rho(g) = \begin{pmatrix} a_g & b_g\\c_g & d_g\end{pmatrix}$, then we get two characters $c_1$, $c_2 : G_{\QQ_{\ell}} \to (R^{\pd,\ell}_{\rhob_0}/(x_3))^*$ sending $g$ to $a_g \pmod{(x_3)}$ and $d_g \pmod{(x_3)}$, respectively. Moreover, $c_1$ and $c_2$ are deformations of $\chi_1|_{G_{\QQ_{\ell}}}$ and $\chi_2|_{G_{\QQ_{\ell}}}$, respectively. As $p \nmid \ell-1$, this means that $c_1(I_{\ell})=c_2(I_{\ell})=1$. So we have $a=1+x_3a'$ and $d=1+x_3d'$.

 From the action of the Frobenius on the tame inertia, we get that $\rho(zi_{\ell}z^{-1})=\rho(i_{\ell})^{\ell}$. As $x_3=m'(b \otimes c)$, we see, by induction, that for a positive integer $n$, $$\rho(i_{\ell})^n = \begin{pmatrix} 1 + x_3a'_n & b(n+x_3b'_n)\\ c(n+x_3c'_n) & 1+x_3d'_n\end{pmatrix}$$ for some $a'_n,b'_n,c'_n,d'_n \in R^{\pd,\ell}_{\rhob_0}$. Therefore, we get that $$\begin{pmatrix} a & (a_0/d_0)b\\ (d_0/a_0)c & d\end{pmatrix} =  \begin{pmatrix} 1 + x_3a'_{\ell} & b(\ell+x_3b'_{\ell})\\ c(\ell+x_3c'_{\ell}) & 1+x_3d'_{\ell}\end{pmatrix}.$$

Thus, $(a_0/d_0)b = b(\ell+x_3b'_{\ell})$ implies that $m'((a_0/d_0 -\ell-x_3b'_{\ell})b \otimes C^{\pd})=0$ and $(d_0/a_0)c = c(\ell+x_3c'_{\ell})$ implies that $m'((d_0/a_0 -\ell-x_3c'_{\ell})c \otimes B^{\pd})=0$. Therefore, we have $x_3(a_0/d_0 -\ell-x_3b'_{\ell})=0$, $x_1(a_0/d_0 -\ell-x_3b'_{\ell})=0$, $x_3(d_0/a_0 -\ell-x_3c'_{\ell})=0$ and $x_2(d_0/a_0 -\ell-x_3b'_{\ell})=0$. As $p | \ell+1$ and $\chi_1(\text{Frob}_{\ell})=\ell\chi_2(\text{Frob}_{\ell})$, we get the following relations from the relations above: there exists $b'', c'' \in R^{\pd,\ell}_{\rhob_0}$ such that
$x_3(a'_0-d'_0+x_3b'')=0$, $x_1(a'_0-d'_0+x_3b'')=0$, $x_3(d'_0-a'_0+x_3c'')=0$ and $x_2(d'_0-a'_0+x_3c'')=0$.

Thus, $J_0$ contains the elements $X_3Y + X_3^2q_1$, $X_1Y + X_1X_3q_2$ and $-X_2Y + X_2X_3q_3$ for some $q_1$, $q_2$, $q_3 \in R_0$. As minimum number of generators of $J_0$ is $3$, it follows, by Nakayama's lemma, that $J_0/m_0J_0$ is an $\FF$ vector space of dimension $3$. Since $m_0J_0 \subset m_0^3$, we see that the images of $X_3Y + X_3^2q_1$, $X_1Y + X_1X_3q_2$ and $-X_2Y + X_2X_3q_3$ inside $J_0/m_0J_0$ are linearly independent over $\FF$. Therefore, they form an $\FF$-basis of the vector space $J_0/m_0J_0$. Hence, by Nakayama's lemma, we get that $J_0 = (X_3Y + X_3^2q_1, X_1Y + X_1X_3q_2, -X_2Y + X_2X_3q_3)$.

In particular, $J_0 \subset (X_3,Y)$. This implies that the Krull dimension of $R^{\pd,\ell}_{\rhob_0}$ is $4$. However, we know that  the Krull dimension of $R^{\pd,\ell}_{\rhob_0}$ is $3$. Hence, we get a contradiction to the hypothesis that $J_0$ is generated by $3$ elements. Therefore, $R^{\pd,\ell}_{\rhob_0}$ is not a local complete intersection ring.
\end{proof}

\begin{cor}
\label{lcicor}
Suppose $\rhob_0$ is unobstructed and $p \nmid \phi(N)$. Let $\ell$ be a prime such that $\ell \equiv -1 \pmod{p}$, $\chi|_{G_{\QQ_{\ell}}} = \omega_p|_{G_{\QQ_{\ell}}}$ and $-\ell$ is a topological generator of $1 + p\ZZ_p$. Then $\calR^{\pd,\ell}_{\rhob_0}$ is not a local complete intersection ring.
\end{cor}
\begin{proof}
Since $\calR^{\pd,\ell}_{\rhob_0}/(p) \simeq R^{\pd,\ell}_{\rhob_0}$, we see, from Theorem~\ref{reducestrthm}, that the Krull dimension of $\calR^{\pd,\ell}_{\rhob_0}$ is either $3$ or $4$. As $\rhob_0$ is unobstructed and $p \nmid \phi(N)$, we know that $\calR^{\pd}_{\rhob_0} \simeq W(\FF)[[X,Y,Z]]$. We have surjective map $\calR^{\pd,\ell}_{\rhob_0} \to \calR^{\pd}_{\rhob_0}$ induced from the surjection $G_{\QQ,N\ell p} \to G_{\QQ,Np}$. Hence, the Krull dimension of $\calR^{\pd,\ell}_{\rhob_0}$ is $4$. As $\dim(\tan(R^{\pd,\ell}_{\rhob_0}))=6$, we know that $\calR^{\pd,\ell}_{\rhob_0} \simeq W(\FF)[[X,Y,Z,X_1,X_2,X_3]]/J$ for some ideal $J$ of $W(\FF)[[X,Y,Z,X_1,X_2,X_3]]$. If $\calR^{\pd,\ell}_{\rhob_0}$ is a local complete intersection ring, then $J$ is generated by $3$ elements. But this would imply that $R^{\pd,\ell}_{\rhob_0}$ is a local complete intersection ring which is not true by Theorem~\ref{lcithm}. Hence, we see that $\calR^{\pd,\ell}_{\rhob_0}$ is not a local complete intersection ring.
\end{proof}

\section{Applications to Hecke algebras}
\label{heckesec}
In this section, we will use the results proved so far to determine structure of big $p$-adic Hecke algebras in some cases and prove `big' $R=\TT$ in those cases. We begin by defining the big $p$-adic Hecke algebra.

Let $M_k(N, W(\FF))$ be the space of modular cuspforms of level $\Gamma_1(N)$ and weight $k$ with Fourier coefficients in $W(\FF)$. We view it as a subspace of $W(\FF)[[q]]$ via $q$-expansions. Let $M_{\leq k}(N,W(\FF)) := \sum_{i=0}^{k} M_k(N, W(\FF)) \subset W(\FF)[[q]]$. Let $\TT^{\Gamma_1(N)}_k$ be the $W(\FF)$-subalgebra of $\text{End}_{W(\FF)}(M_{\leq k}(N, W(\FF)))$ generated by the Hecke operators $T_q$ and $S_q$ for primes $q \nmid Np$ (see \cite[Definition 1.7, Definition 1.8]{Em} for the action of these Hecke operators on $q$-expansions). Let $\TT^{\Gamma_1(N)} := \varprojlim_k \TT^{\Gamma_1(N)}_k$.

Given a modular form $f$, let $\mathcal{O}_f$ be the ring of integers of the finite extension of $\QQ_p$ containing all the Fourier coefficients of $f$. Now suppose $\rhob_0$ is modular of level $N$ i.e. there exists an eigenform $f$ of level $\Gamma_1(N)$ such that the semi-simplification of the reduction of the $p$-adic Galois representation attached to $f$ modulo the maximal ideal of $\mathcal{O}_f$ is $\rhob_0$. Then we get a maximal ideal $m_{\rhob_0}$ of $\TT^{\Gamma_1(N)}$ corresponding to $\rhob_0$ (see \cite[Section 1]{D} and \cite[Section 1.2]{BK}). Let $\TT^{\Gamma_1(N)}_{\rhob_0}$ be the localization of $\TT^{\Gamma_1(N)}$ at $m_{\rhob_0}$. So $\TT^{\Gamma_1(N)}_{\rhob_0}$ is a complete Noetherian local $W(\FF)$-algebra with residue field $\FF$ (see \cite[Section 1]{D} and \cite[Section 1.2]{BK}).

Let $\ell$ be a prime not dividing $Np$. After replacing $N$ by $N\ell$ everywhere in the construction of $\TT^{\Gamma_1(N)}_{\rhob_0}$, we get $\TT^{\Gamma_1(N\ell)}_{\rhob_0}$. Thus we have a natural morphism $\psi : \TT^{\Gamma_1(N\ell)}_{\rhob_0} \to \TT^{\Gamma_1(N)}_{\rhob_0}$ obtained by restriction of the Hecke operators acting on the space of modular forms of level $\Gamma_1(N\ell)$ to the space of modular forms of level $\Gamma_1(N)$.

\begin{prop}
\label{surjprop}
\begin{enumerate}
\item There exists a pseudo-representation $(\tau^{\Gamma_1(N)}, \delta^{\Gamma_1(N)}) : G_{\QQ,Np} \to \TT^{\Gamma_1(N)}_{\rhob_0}$ deforming $(\tr(\rhob_0), \det(\rhob_0))$ such that $\tau^{\Gamma_1(N)}(\text{Frob}_q) = T_q$ for all primes $q \nmid Np$ and the morphism $\phi' : \calR^{\pd}_{\rhob_0} \to \TT^{\Gamma_1(N)}_{\rhob_0}$ induced from it is surjective.
\item There exists a pseudo-representation $(\tau^{\Gamma_1(N\ell)}, \delta^{\Gamma_1(N\ell)}) : G_{\QQ,N\ell p} \to \TT^{\Gamma_1(N\ell)}_{\rhob_0}$ deforming $(\tr(\rhob_0), \det(\rhob_0))$ such that $\tau^{\Gamma_1(N\ell)}(\text{Frob}_q) = T_q$ for all primes $q \nmid N\ell p$ and the morphism $\phi: \calR^{\pd,\ell}_{\rhob_0} \to \TT^{\Gamma_1(N\ell)}_{\rhob_0}$ induced from it is surjective.
\item The natural morphism $\psi : \TT^{\Gamma_1(N\ell)}_{\rhob_0} \to \TT^{\Gamma_1(N)}_{\rhob_0}$ is surjective.
\end{enumerate}
\end{prop}
\begin{proof}
The first two parts follow from \cite[Lemma 4]{D} and \cite[Section 2]{D}. 
For the last part, we view $\tau^{\Gamma_1(N)}$ as a pseudo-character of $G_{\QQ,N\ell p}$ and denote it by $\tau$. We know that $\tau^{\Gamma_1(N)}(\text{Frob}_q) = T_q$ and that $\tau^{\Gamma_1(N\ell)}(\text{Frob}_q) = T_q$ for all primes $q \nmid N\ell p$. By Cebotarev density theorem, we know that the set $\{\text{Frob}_q | q \nmid N\ell p\}$ is dense in $G_{\QQ,N p}$. Hence, we have $\psi \circ \tau^{\Gamma_1(N\ell)} = \tau$ which means $\psi \circ \phi \circ T^{\univ,\ell} = \tau$. 

On the other hand, if $f : \calR^{\pd,\ell}_{\rhob_0} \to  \calR^{\pd}_{\rhob_0}$ is the natural morphism obtained by viewing $T^{\univ}$ as pseudo-character of $G_{\QQ,N\ell p}$, then $\phi' \circ f \circ T^{\univ,\ell} = \tau$. The universal property of $\calR^{\pd,\ell}_{\rhob_0}$ implies that $\psi \circ \phi = \phi' \circ f$. Therefore, the surjectivity of $\phi'$ implies the surjectivity of $\psi$.
\end{proof}

\begin{rem}
Suppose $p \nmid \phi(N)$, $\rhob_0$ is modular of level $N$ and unobstructed. Let $\ell$ be a prime such that $\ell \nmid Np$ and $\chi^i|_{G_{\QQ_\ell}} = \omega_p$ for some $i \in \{1,-1\}$. Moreover assume that either $p \nmid \ell^2-1$ or $p | \ell + 1$ and $p^2 \nmid \ell+1$.
Then combining Proposition~\ref{surjprop}, Corollary~\ref{strcor}, proof of Corollary~\ref{lcicor} and the Gouvea-Mazur infinite fern argument (\cite[Corollary 2.28]{Em}), we get that $\TT^{\Gamma_1(N\ell)}_{\rhob_0}$ is equidimensional of Krull dimension $4$. This proves \cite[Conjecture 2.9]{Em} in some special cases.
\end{rem}

We say that an eigenform $h$ of level $N\ell$ lifts $\rhob_0$ if the semi-simplification of the reduction of the $p$-adic Galois representation attached to it modulo the maximal ideal of $\mathcal{O}_h$ is isomorphic to $\rhob_0$.
\begin{thm}
\label{rthm}
Suppose $p \nmid \phi(N)$, $\rhob_0$ is modular of level $N$ and unobstructed. Let $\ell$ be a prime such that $\ell \nmid Np$, $p \nmid \ell^2-1$, $\chi^i|_{G_{\QQ_\ell}} = \omega_p|_{G_{\QQ_\ell}}$ for some $i \in \{1,-1\}$ and $\ell/\tilde\ell$ is a topological generator of $1+p\ZZ_p$. Suppose there exists an eigenform $g$ of level $\Gamma_1(N\ell)$ lifting $\rhob_0$ which is new at $\ell$. Then the surjective morphism $\phi : \calR^{\pd,\ell}_{\rhob_0} \to \TT^{\Gamma_1(N\ell)}_{\rhob_0}$ is an isomorphism and $$\TT^{\Gamma_1(N\ell)}_{\rhob_0} \simeq W(\FF)[[X_1,X_2,X_3,X_4]]/(X_2X_4).$$
\end{thm}
\begin{proof}
Without loss of generality, assume $\chi|_{G_{\QQ_\ell}} = \omega_p$. Suppose $\phi$ is not an isomorphism.
By Theorem~\ref{ramunobsprop}, we know that $\calR^{\pd,\ell}_{\rhob_0} \simeq W(\FF)[[X_1,X_2,X_3,X_4]]/(X_2X_4)$. 
By Gouvea-Mazur infinite fern argument (\cite[Corollary 2.28]{Em}), we know that if $P$ is a minimal prime of $\TT^{\Gamma_1(N\ell)}_{\rhob_0}$, then $\TT^{\Gamma_1(N\ell)}_{\rhob_0}/P$ has Krull dimension at least $4$. 
Hence, we have $\TT^{\Gamma_1(N\ell)}_{\rhob_0} \simeq W(\FF)[[X,Y,Z]]$.

As $\rhob_0$ is unobstructed and $p \nmid \phi(N)$, it follows from \cite[Corollary 2.28]{Em} and Lemma~\ref{unobslem}, that $\phi' : \calR^{\pd}_{\rhob_0} \to \TT^{\Gamma_1(N)}_{\rhob_0}$ is an isomorphism and both are isomorphic to $W(\FF)[[X,Y,Z]]$.
Therefore, we get that the surjective map $\psi : \TT^{\Gamma_1(N\ell)}_{\rhob_0} \to \TT^{\Gamma_1(N)}_{\rhob_0}$ is an isomorphism. 

By Lemma~\ref{ramlem} and Proposition~\ref{structprop}, there exists a representation $\rho : G_{\QQ, N\ell p} \to \GL_2(\calR^{\pd,\ell}_{\rhob_0})$ such that $\tr(\rho) = T^{\univ,\ell}$ and there exists a $w \in \calR^{\pd,\ell}_{\rhob_0}$ such that $\rho(I_\ell)$ is the cyclic group generated by $\begin{pmatrix} 1 & w \\ 0 & 1 \end{pmatrix}$. Moreover, Lemma~\ref{surjectlem} implies that $(w)$ is the kernel of the natural surjective map $f : \calR^{\pd,\ell}_{\rhob_0} \to \calR^{\pd}_{\rhob_0}$. As $\psi \circ \phi = \phi' \circ f$ and $\psi$ is an isomorphism, we see that $\phi(w)=0$.

Let $g$ be an eigenform of level $\Gamma_1(N\ell)$ lifting $\rhob_0$ which is new at $\ell$. So we get a morphism $\phi_g : \TT^{\Gamma_1(N\ell)}_{\rhob_0} \to \mathcal{O}_g$ sending each Hecke operator to its $g$ eigenvalue. Let $\rho_g : G_{\QQ,N\ell p} \to \GL_2(\mathcal{O}_g)$ be the $p$-adic Galois representation attached to $g$. Let $\rho'_g = \phi_g \circ \phi \circ \rho$. Then $\rho'_g : G_{\QQ,N\ell p} \to \GL_2(\mathcal{O}_g)$ is a representation such that $\tr(\rho'_g) = \tr(\rho_g)$ and $\rho'_g$ is unramified at $\ell$. As $\rho_g$ is absolutely irreducible, we see, by Brauer Nesbitt theorem, that $\rho_g \simeq \rho'_g$ over $\bar\QQ_p$. This means $\rho_g$ is unramified at $\ell$ contradicting the assumption that $g$ is new at $\ell$. Hence, $\phi$ is an isomorphism.
\end{proof}

As corollaries, we get:

\begin{cor}
\label{rtcor}
Suppose $\rhob_0$ is unobstructed, $p \nmid \phi(N)$, the Artin conductor of $\rhob_0$ divides $N$, $\chi_2$ is unramified at $p$ and $\det(\rhob_0) = \psi\omega_p^{k_0-1}$ with $2 < k_0 < p$ and $\psi$ unramified at $p$. 
Let $\ell$ be a prime such that $\ell \nmid Np$, $p \nmid \ell^2-1$, $\ell/\tilde\ell$ is a topological generator of $1+p\ZZ_p$ and $\chi|_{G_{\QQ_{\ell}}} = \omega_p^{-1}|_{G_{\QQ_{\ell}}}$.
Then, we have: $$\calR^{\pd,\ell}_{\rhob_0} \simeq \TT^{\Gamma_1(N\ell)}_{\rhob_0} \simeq W(\FF)[[X_1,X_2,X_3,X_4]]/(X_2X_4).$$
\end{cor} 
\begin{proof}
From \cite[Lemma $2.5$]{D2}, it follows that $\rhob_0$ is modular of level $N$ and by \cite[Theorem B]{D2}, we get the existence of an eigenform $g$ of level $\Gamma_1(N\ell)$ lifting $\rhob_0$ which is new at $\ell$.
The corollary now follows from Theorem~\ref{rthm}.
\end{proof}

\begin{cor}
\label{rcor}
Suppose $N=1$, $\rhob_0 = 1 \oplus \omega_p^k$ for some odd $2 < k < p-3$ and $\ell$ is a prime such that $\ell \nmid Np$, $p \nmid \ell^2-1$ and $p \mid \mid \ell^{k+1}-1$. Moreover suppose either $p$ is a regular prime or $p$ does not divide $B_{k+1}B_{p-k}$, where $B_k$ is the $k$-th Bernoulli number.
Then, we have: $$\calR^{\pd,\ell}_{\rhob_0} \simeq \TT^{\Gamma_1(\ell)}_{\rhob_0} \simeq W(\FF)[[X_1,X_2,X_3,X_4]]/(X_2X_4).$$
\end{cor}
\begin{proof}
Note that if $\ell \not\equiv \pm 1 \pmod{p}$ and $p \mid \mid \ell^{k+1}-1$, then $p \nmid \ell^2-1$ and $\ell/\tilde\ell$ is a topological generator of $1+p\ZZ_p$.
If either $p$ is regular or $p \nmid B_{k+1}B_{p-k}$, then either \cite[Lemma $21$]{BK} or \cite[Theorem 22]{BK} implies that $1 + \omega_p^k$ is an unobstructed pseudo-character of $G_{\QQ,p}$.
Since $p \mid \ell^{k+1}-1$, we have $\omega_p^{k}|_{G_{\QQ_\ell}} = \omega_p^{-1}|_{G_{\QQ_\ell}}$.
The corollary now follows directly from Corollary~\ref{rtcor}.
\end{proof}

\begin{rem}
One can also use \cite[Theorem $1$]{BM} instead of Corollary~\ref{rtcor} to prove Corollary~\ref{rcor}.
\end{rem}

{\bf Examples:} The hypotheses of Corollary~\ref{rcor} are satisfied in the following cases:
\begin{enumerate}
\item $p=13$, $\rhob_0 = 1 \oplus \omega_p^3$ and $\ell \equiv 5 \pmod{169}$,
\item $p=17$, $\rhob_0 = 1 \oplus \omega_p^3$ and $\ell \equiv 4 \pmod{289}$,
\item $p=37$, $\rhob_0 = 1 \oplus \omega_p^3$ and $\ell \equiv 6 \pmod{1369}$.
\end{enumerate}

We now give some examples satisfying the hypotheses of Theorem~\ref{rthm} for $\rhob_0= 1 \oplus \omega_p$.
Note that these cases are not covered in \cite[Theorem A]{D2}.
Let $E_k$ be the Eisenstein series of weight $k$ and for a modular form $f$, denote its $n$-th Fourier coefficient by $a_n(f)$. We now consider $M_i(N,\ZZ_p)$ as a submodule of $\ZZ_p[[q]]$ via $q$-expansions. Let $M_i(N,\FF_p)$ be the image of $M_i(N,\ZZ_p)$ in $\FF_p[[q]]$ under the reduction modulo $p$ map $\ZZ_p[[q]] \to \FF_p[[q]]$.
\begin{lem}
Let $p=5,7,11$ and $\ell$ be a prime such that $\ell \not\equiv \pm 1 \pmod{p}$ and $p^2 \nmid \ell^{p-1}-1$. Then the tuple $(p,\ell,1 \oplus \omega_p)$ satisfies the hypotheses of Theorem~\ref{rthm}.
\end{lem}
\begin{proof}
By \cite[Theorem 22]{BK}, we know that $1 \oplus \omega_p$ is unobstructed. So we only need to check that there exists a newform of level $\Gamma_0(\ell)$ lifting $\rhob_0$.

Let $f_\ell := \frac{-B_{p-1}}{4(p-1)}(E_{p-1}(q)-E_{p-1}(q^{\ell}))$. Now $f_\ell \in M_{p-1}(\ell, \ZZ_p)$. Let $\bar f_\ell$ be the image of $f_\ell$ in $M_{p-1}(\ell,\FF_p)$. So we have $F_\ell := \Theta\bar f_\ell \in M_{2p}(\ell,\FF_p)$, where $\Theta$ is the Ramanujan theta operator. Note that $F_\ell \neq 0$.

Note that the action of the Hecke operators $T_q$ for primes $q \neq \ell$ and $U_\ell$ on $M_{2p}(\ell,\ZZ_p)$ descends to $M_{2p}(\ell,\FF_p)$. Moreover, the action of $T_p$ on $M_{2p}(\ell,\FF_p)$ coincides with action of $U_p$ i.e. if $f \in M_{2p}(\ell,\FF_p)$ and $f = \sum a_n(f)q^n$, then $T_p f= \sum a_{pn}(f)q^n$. 

By \cite[Fact 1.6]{J}, it follows that for a prime $q \neq \ell,p$, $T_q F_\ell = (1+q)F_\ell$, $U_\ell F_\ell = F_\ell$ and $T_p F_\ell = 0$. As all these Hecke operators commute with each other, we get, by Deligne-Serre Lemma, that there exists a $G_\ell \in M_{2p}(\Gamma_0(\ell), \bar\QQ_p)$ such that:
\begin{enumerate}
\item $G_\ell$ is an eigenform for $U_\ell$ and for all $T_q$ where $q \neq \ell$ is a prime,
\item Modulo the maximal ideal of $\mathcal{O}_{G_\ell}$, its $T_q$ eigenvalue reduces to $1+q$ for $q \nmid p\ell$, $T_p$ eigenvalue reduces to $0$ and $U_\ell$ eigenvalue reduces to $1$.
\end{enumerate}

Thus $G_\ell$ is an eigenform lifting $1 \oplus \omega_p$. As $\ell \not\equiv 1 \pmod{p}$, the only Eisenstein series of weight $2p$ and level $\Gamma_0(\ell)$ with $U_\ell$ eigenvalue $1 \pmod{p}$ is $E_{2p}(q) - \ell^{2p-1}E_{2p}(q^{\ell})$. But the $T_p$ eigenvalue of $E_{2p}(q) - \ell^{2p-1}E_{2p}(q^{\ell})$ is $1+p^{2p-1}$. Hence, $G_\ell$ is a cuspform. 

If $p=5,7$, then there are no cuspforms of weight $2p$ and level $1$. Hence, $G_\ell$ has to be a newform when $p=5,7$. Now suppose $p=11$. Then the only cusp eigenform of weight $22$ and level $1$ is $\Delta E_{10}$. As $E_{10} \equiv 1 \pmod{11}$, $\Delta E_{10} \equiv \Delta \pmod{11}$. Let $\rho_{\Delta}$ be the $11$-adic Galois representation attached to $\Delta$. As $\tau(2) = -24 \not\equiv 3 \pmod{11}$, it follows that the semi-simplification of $\rho_{\Delta} \pmod{11}$ is not $1 \oplus \omega_p$. Hence, we see that $\Delta E_{10} \neq G_\ell$. Hence, $G_\ell$ has to be a newform when $p=11$. This finishes the proof of the lemma.
\end{proof}

\end{document}